\documentclass[12pt]{article}
\usepackage[margin=1.5cm]{geometry}
\usepackage{authblk}
\usepackage{amssymb}
\usepackage{amsthm}
\usepackage{amsmath}
\usepackage{bbm}
\usepackage{amscd}
\usepackage{stmaryrd}
\usepackage{array}
\usepackage{url}
\usepackage{mathtools}
\usepackage{hyperref}
\usepackage[utf8]{inputenc}
\usepackage{t1enc}
\usepackage{enumerate}
\usepackage[mathscr]{eucal}
\usepackage[british]{babel}
\usepackage[dvips]{graphicx}
\usepackage[dvips]{epsfig}
\usepackage{epsf}
\usepackage{color}
\usepackage{indentfirst}
\usepackage[all]{xy}
\usepackage{comment}
\hypersetup{colorlinks=true, linkcolor=blue, filecolor=blue, urlcolor=blue, citecolor=blue}

\newcommand{\Ker}{\operatorname{Ker}}
\newcommand{\GQpD}{G_{\mathbb{Q}_p,\Delta}}
\newcommand{\GQpa}{G_{\mathbb{Q}_p,\alpha}}
\newcommand{\HQpD}{H_{\mathbb{Q}_p,\Delta}}

\newcommand{\GQpDa}{G_{\mathbb{Q}_p,\Delta\setminus\{\alpha\}}}

\newcommand{\soc}{\operatorname{soc}}

\newcommand{\Zp}{\mathbb{Z}_p}
\newcommand{\nr}{\mathrm{nr}}
\newcommand{\Fp}{\mathbb{F}_p}
\newcommand{\Fq}{\mathbb{F}_q}
\newcommand{\Qp}{\mathbb{Q}_p}

\newcommand{\OED}{\mathcal{O}_{\mathcal{E}_\Delta}}

\newcommand{\diag}{\operatorname{diag}}
\newcommand{\genlength}{{\operatorname{length}_{gen}}}

\newcommand{\Tr}{\operatorname{Tr}}
\newcommand{\Ext}{\operatorname{Ext}}

\newcommand{\ord}{\operatorname{ord}}

\newcommand{\Ind}{\operatorname{Ind}}
\newcommand{\Fil}{\operatorname{Fil}}
\newcommand{\GL}{\operatorname{GL}}
\newcommand{\Gal}{\operatorname{Gal}}
\newcommand{\Rep}{\operatorname{Rep}}

\newcommand{\res}{\operatorname{res}}
\newcommand{\id}{\operatorname{id}}

\newcommand{\rk}{\operatorname{rk}}
\newcommand{\pr}{\operatorname{pr}}

\newcommand{\SP}{\operatorname{SP}}
\newcommand{\Sp}{\operatorname{Sp}}
\newcommand{\Fin}{\operatorname{Fin}}
\newcommand{\gr}{\operatorname{gr}^{\cdot}}
\newcommand{\bg}{(\hspace{-0.06cm}(}
\newcommand{\jg}{)\hspace{-0.06cm})}

\newcommand{\bs}{[\hspace{-0.04cm}[}
\newcommand{\js}{]\hspace{-0.04cm}]}
\newtheorem{thm}{Theorem}[subsection]
\newtheorem{thma}{Theorem}

\newtheorem{pro}[thm]{Proposition}
\newtheorem{lem}[thm]{Lemma}
\newtheorem{cor}[thm]{Corollary}

\newtheorem{con}[thm]{Conjecture}

\newtheorem{que}{Question}

\theoremstyle{definition}
\newtheorem*{rem}{Remark}
\newtheorem*{rems}{Remarks}
\title{Finiteness properties of generalized Montréal functors with applications to mod $p$ representations of $\mathrm{GL}_n(\mathbb{Q}_p)$}
\author{Gergely Jakov\'ac\footnote{School of Mathematics, University of Bristol}\ \ and Gergely Z\'abr\'adi\footnote{Institute of Mathematics, Eötvös Loránd University, Budapest \& HUN-REN Alfréd Rényi Institute of Mathematics, MTA–RI Lendület ``Momentum'' Analytic Number Theory and Representation Theory Research Group}}
\date{\today}

\begin{document}

\maketitle
\begin{abstract}
The second named author previously \cite{MultVar,MultVarGal} constructed a functor $\mathbb{V}^\vee\circ D^\vee_\Delta$ from the category of smooth $p$-power torsion representations of $\GL_n(\Qp)$ to the category of inductive limits of continuous representations on finite $p$-primary abelian groups of the direct product $\GQpD\times \Qp^\times$ of $(n-1)$ copies of the absolute Galois group of $\Qp$ and one copy of the multiplicative group $\Qp^\times$. In the present work we show that this functor attaches finite dimensional representations on the Galois side to smooth $p$-power torsion representations of finite length on the automorphic side. This has some implications on the finiteness properties of Breuil's functor \cite{Breuil}, too. Moreover, $\mathbb{V}^\vee\circ D^\vee_\Delta$ produces irreducible representations of $\GQpD\times \Qp^\times$ when applied to irreducible objects on the automorphic side and detects isomorphisms unless it vanishes. Further, we determine the kernel of $D^\vee_\Delta$ when restricted to successive extensions of subquotients of principal series. We use this to characterize representations that are parabolically induced from the product of a torus and $\mathrm{GL}_2(\mathbb{Q}_p)$. Finally, we formulate a conjecture and prove partial results on the essential image.
\end{abstract}

\section{Introduction}

\subsection{Background and motivation}

The Langlands programme is one of the most central areas in modern number theory. Since the Millenium it has been increasingly clear that in order to prove general (global) modularity results, one needs the extension of the local Langlands correspondence to the case $\ell=p$ (see \cite{BreuilSurvey} for an overview) where $p$ stands for the residue characteristic of the local field and $\ell$ is the residue characteristic of the coefficient field. The difficulty in the $p$-adic local Langlands programme (as opposed to the classical $\ell\neq p$ case) is that representations on both the automorphic- and Galois sides are very far from being semisimple. Hence it does not suffice to provide a bijection for the irreducible objects, one also has to understand various (successive) extensions. Therefore it is natural, though perhaps optimistic, to expect that the---still hypothetical---correspondence is realized by a functor.

By now the $p$-adic local Langlands correspondence for $\GL_2(\mathbb{Q}_p)$ is very well understood through the works of many  people \cite{Berger,Breuil03a,Breuil03b,BE,Colmez2,Colmez1,E,Kisin,Paskunas,CDP,CEGGPS}. It is indeed realized by a functor \cite{Colmez2,Colmez1} (the so-called Montréal functor of Colmez, first announced at a conference in Montréal in 2005) which makes it possible to extend the correspondence to extensions of representations and to deformations on both sides. Later on there have been attempts \cite{SchVig,Breuil} (see also \cite{links}) to generalize Colmez' functor to groups of higher rank. The goal of the papers \cite{MultVar} and \cite{MultVarGal} was to unify the advantages of these approaches, especially to lose as little information as possible when passing from the automorphic- to the Galois side. In order to do so, one has to work with representations of products of Galois groups (and the center of the reductive group). Note that such representations also played a central role in Drinfeld's work \cite{DrinfeldICM} (and subsequent works by  L.\  Lafforgue \cite{LLafforgue} and by V.\ Lafforgue \cite{VLafforgue1, VLafforgue2}) on the geometric Langlands in positive characteristic (see also \cite{KedlayaSSS} and \cite{ScholzeWeinstein} for shtukas in mixed characteristic).

In order to briefly summarize the results of \cite{MultVar,MultVarGal} let $K/\Qp$ be a finite extension and let ${\bf G}$ be a $K$-split reductive group defined over $K$ with connected centre.  Assume for now (as it is also assumed in both papers \cite{MultVar,MultVarGal}) that $K=\Qp$ and put $G:={\bf G}(\Qp)$. In \cite{MultVar} we define a functor $D^\vee_\Delta$ from the category of smooth mod $p^h$-representations of $G$ to the category of projective limits of so-called étale $T_+$-modules over a multivariate Laurent-series ring in $|\Delta|$ variables. Here $\Delta$ is the (finite) set of simple roots corresponding to a split Borel subgroup $B=TN$ and $T_+\subseteq T$ is the submonoid of dominant elements in the torus $T$. The functor $D^\vee_\Delta$ is right exact contravariant and compatible with tensor products and with parabolic induction \cite{MultVar}. In \cite{MultVar} (see also \cite{CKZ}) a Fontaine-style equivalence of categories is proven---at least in case ${\bf G}=\GL_n$---between the category of étale $T_+$-modules as above and mod $p^h$ representations of the group $\GQpD\times\Qp^\times$ where $\GQpD$ is the $|\Delta|$-fold direct power of the absolute Galois group $G_{\Qp}=\Gal(\overline{\Qp}/\Qp)$ and $\Qp^\times$ is identified with the center $Z(\GL_n(\Qp))$. The equivalence is realized by the exact functors $\mathbb{V}=\mathbb{V}_\Delta$ and $\mathbb{D}=\mathbb{D}_\Delta$. In accordance with the original papers of Colmez \cite{Colmez2,Colmez1} we prefer to apply another Pontryagin duality $(\cdot)^\vee$ so that the composite functor $\mathbb{V}^\vee\circ D^\vee_\Delta$ is covariant and left exact. See \cite{CKZ} for a more conceptual proof via Drinfeld's lemma for perfectoid spaces and \cite{RWZ} for a generalization to higher dimensional local fields.

In the present paper our goal is to prove further properties of the functor $\mathbb{V}^\vee\circ D^\vee_\Delta$, most notably that it attaches finite dimensional representations on the Galois side to objects of finite length on the autmorphic side (for details we refer the reader to section \ref{mainresults}). We concentrate on the case of $G=\GL_n(\Qp)$ even though most results can be generalized with some effort to ${\bf G}(\Qp)$ where ${\bf G}$ is a $\Qp$-split reductive group with connected center. The reason why we stick to $\GL_n(\Qp)$ is the following. Whenever ${\bf G}$ is a more general reductive group, the precise connection of étale $T_+$-modules to actual representations on the Galois side is yet to be understood. Note that from general Langlands philosophy, one would expect representations valued in (the $\Qp$-points of) the dual group ${\bf G}^\vee$. We expect that some kind of tannakian formalism should exist in order to answer this question. On the other hand, whenever $K\neq \Qp$ is a proper extension then the classical cyclotomic $(\varphi,\Gamma)$-modules are less well suited, one either has to work with Lubin--Tate theory or with a different notion of multivariable $(\varphi,\mathcal{O}_K^\times)$-modules \cite{BHHMS,BHHMS2,BHHMS3} (and possibly combine these with ideas in this paper and \cite{MultVar} when the group ${\bf G}$ is not $\GL_2$).

\subsection{Main results and outline of the paper}\label{mainresults}

Our first main theorem is the following finiteness result (Corollary \ref{finitenessofDveeDelta}).

\begin{thma}\label{thmA}
Assume that $\pi$ is a smooth representation of $\GL_n(\mathbb{Q}_p)$ over $\mathcal{O}_F/\varpi^h$ of finite length. Then the attached representation $\mathbb{V}^\vee\circ D_\Delta^\vee(\pi)$ of $\GQpD\times \Qp^\times$ is a finitely generated $\mathcal{O}_F/\varpi^h$-module.
\end{thma}

The analogue of Theorem \ref{thmA} in case $n=2$ relies heavily on the finite presentation \cite[Thm.\ III.3.1]{Colmez1} of irreducible smooth representations of $\GL_2(\Qp)$ by compactly induced representations. The approach \cite{SchVig} of Schneider and Vigneras also uses the resolution of smooth representations by compactly induced representations, but the condition \cite[Prop.\ 9.9]{SchVig} needed for the finiteness properties seems very hard to verify in general (even though there exist positive results \cite{Ollivier} in case of principal series). However, in the approach of Breuil \cite{Breuil} and the second named author \cite{MultVar} (originating in Emerton's \cite{EmertonCoh} reformulation of Colmez' functor) builds the corresponding $(\varphi,\Gamma)$-modules from relatively small subspaces (the set $\mathcal{M}_\Delta(\pi^{H_{\Delta,0}})$ and its analogues---see section \ref{notation:DveeDelta} for the definition) of the representation $\pi$ so the finiteness relies on bounding these subspaces inside $\pi$. The idea \cite{MultVar} behind working with multivariable $(\varphi,\Gamma)$-modules instead of classical ones is to keep as much information from $\pi$ as possible. 

The method of the proof is the following. Assume that $\pi$ is a smooth representation of $\GL_n(\Qp)$ over $\Fq$ of finite length such that $D^\vee_\Delta(\pi)\neq 0$. Then for each finitely generated quotient $D$ of $D^\vee_\Delta(\pi)$ in the category of étale $T_+$-modules, we may pass to a noncommutative deformation $\mathbb{M}_{\infty,0}(D)$ (a finitely generated module over some topological localization of $\Fq\bs N_0\js$ together with a $T_+$-action) along with a $T_+^{-1}$-equivariant $\Fq\bs N_0\js$-module homomorphism $\beta_{D,\mathcal{C}_0}\colon\pi^\vee\to \mathbb{M}_{\infty,0}(D)$\footnote{In the text we index $\beta$ by $M$ instead of $D$ where $D=M^\vee[X_\Delta^{-1}]$}. Using the étale hull of the image, one can \cite[Cor.\ 4.22]{MultVar}, in fact, extend $\beta_{D,\mathcal{C}_0}$ to a continuous $G$-equivariant map $\beta_{D,G/B}$ from $\pi^\vee$ to the space $\mathfrak{Y}_{D,\pi}(G/B)$ of global sections of a sheaf $\mathfrak{Y}_{D,\pi}$ (depending on both $D$ and $\pi$) on the flag variety $G/B$ (in the $p$-adic topology). Now $D^\vee_\Delta(\pi)$ is the projective limit of all its finitely generated quotients, so---by the assumption that $\pi$ has finite length---the image of $\beta_{D,G/B}$ must eventually stabilize (when varying $D$). The technical part of the paper concerns the proof that we have $D^\vee_\Delta(\pi)=D$ if $D$ is the smallest finitely generated quotient where this stabilization takes place. In order to prove this we have to compute the value of the functor $D^\vee_\Delta$ at the Pontryagin dual of a smooth representation embedded in some $G$-equivariant sheaf on $G/B$. In such a situation the Bruhat decomposition and Bruhat order gives rise to a filtration by open subsets on the topological space $G/B$ inducing a filtration $\Fil^\bullet(\pi)$ on $\pi$ by $B$-subrepresentations (indexed by the elements of the Weyl group $N_G(T)/T$). We show (Theorem \ref{DveeDelta0gr}) that only the graded piece $\mathrm{gr}^{w_0}(\pi)$ corresponding to the full-dimensional ``big open cell'' $Bw_0 B$ (where $w_0\in N_G(T)/T$ is the longest element) contributes to the value of the functor $D^\vee_\Delta$. Here we use a noncommutative analogue of $D^{\nr}=\bigcap_{k\geq 1}\varphi^k(D)$ noting that under the identification of a classical $(\varphi,\Gamma)$-module $D$ with the global sections of a sheaf of abelian groups on $\mathbb{Z}_p$, $D^{\nr}$ equals the kernel of the restriction of $D$ to the open subspace $\Zp\setminus\{0\}$. The technical difficulty in proving the finiteness of $D^\vee_\Delta(\mathrm{gr}^{w_0}(\pi))$ is the following. In the commutative setting, any compact $\Fq\bs X\js$-submodule of a finitely generated $\Fq\bg X\jg$-module (and even any compact $E_\Delta^+$-submodule of a finitely generated $E_\Delta$-module) is finitely generated over $\Fq\bs X\js$ (resp.\ over $E_\Delta^+$). However, the (noncommutative) topological localization $\Fq\bg N_{\Delta,\infty}\jg$ of $\Fq\bs N_0\js$ contains compact $\Fq\bs N_0\js$-submodules which are \emph{not} finitely generated. This is analogous to the fact that Fontaine's ring $\mathcal{O_E}=\varprojlim \Zp/(p^h)\bg X\jg$ also contains compact $\Zp\bs X\js$-submodules that are not finitely generated. The way we overcome this difficulty is to use the $\psi$-action of $T_+$ (Corollary \ref{compact+psi-fingen}). The key point is a filtration argument (Proposition \ref{fingen+kertors}) combined with the existence of an element $t_{\overline{\alpha}}\in T_+$ for each simple root $\alpha\in\Delta$ such that $\alpha(t_\alpha)$ is a unit, but the subgroups $(t_{\overline{\alpha}}^kH_{\Delta,0}t_{\overline{\alpha}}^{-k})_{k\geq 0}$ form a system of open neighbourhoods of $1$ in $H_{\Delta,0}$ (see the proof of Lemma \ref{onlyonetermnotinkerpr}).

As a corollary, we also prove the finiteness of Breuil's functor conditional on the non-vanishing of $D^\vee_\Delta$ (see Corollary \ref{Breuilfin}):

\begin{thma}\label{thmB}
Assume that $\pi$ is an irreducible smooth representation of $\GL_n(\mathbb{Q}_p)$ over $\Fq$ such that $D_\Delta^\vee(\pi)\neq 0$. Then we have $D^\vee_\xi(\pi)\cong \Fq\bg X\jg\otimes_{\Fq\bg N_{\Delta,0}\jg,\ell}D^\vee_\Delta(\pi)$. In particular, $D_\xi^\vee(\pi)$ is a finitely generated $\Fq\bg X\jg$-module and the Galois representation $\mathbb{V}^\vee\circ D^\vee_\xi(\pi)$ attached by Fontaine's equivalence is finite dimensional and isomorphic to the restriction of $\mathbb{V}_\Delta^\vee\circ D_\Delta^\vee(\pi)$ to the diagonal embedding $G_{\Qp}\hookrightarrow \GQpD\times\Qp^{\times}$.
\end{thma}

By the same methods, we obtain our third main theorem (Corollaries \ref{irredtoirred} and \ref{detectsiso}) that the functor $\mathbb{V}^\vee\circ D^\vee_\Delta$ attaches irreducible objects on the Galois side to irreducible smooth representations of $\GL_n(\Qp)$ and detects isomorphisms:
\begin{thma}\label{thmC}\begin{enumerate}
\item Assume that $\pi$ is an irreducible smooth representation of $\GL_n(\mathbb{Q}_p)$ over $\Fq$ such that $D_\Delta^\vee(\pi)\neq 0$. Then $\pi$ is admissible and $\mathbb{V}^\vee(D_\Delta^\vee(\pi))$ is an irreducible representation of $\GQpD\times \Qp^\times$ over $\Fq$.
\item Assume that $\pi_1$ and $\pi_2$ are irreducible smooth representations of $\GL_n(\mathbb{Q}_p)$ over $\Fq$ such that $\mathbb{V}^\vee(D_\Delta^\vee(\pi_1))\cong \mathbb{V}^\vee(D_\Delta^\vee(\pi_2))\neq 0$. Then we have $\pi_1\cong \pi_2$.
\end{enumerate}
\end{thma}

Note that we do not need to assume admissibility in Theorems \ref{thmA}, \ref{thmB}, and \ref{thmC}. The finiteness condition is provided by the nonvanishing of $D^\vee_\Delta(\pi)$ which \emph{implies} admissibility in case $\pi$ is irreducible. On the Galois side the absolutely irreducible representations of $\GQpD\times \Qp^\times$ over $\Fq$ are of the form $V=V_{Z(G)}\otimes_{\Fq}\bigotimes_{\alpha\in\Delta}V_\alpha$ where $V_{Z(G)}$ is a character of the center $Z(G)\cong \Qp^\times$ and $V_\alpha$ is an irreducible representation of $\GQpa$ ($\alpha\in\Delta$). Whenever $\pi$ is induced from a standard parabolic $B\leq P=LU\lneq G$ then we may pick a simple root $\alpha\in\Delta$ such that the root subgroup $N^{(\alpha)}$ is not contained in $L$. In this case $V_\alpha$ must be $1$-dimensional for $V=\mathbb{V}^\vee(D^\vee_\Delta(\pi))$ by \cite[Thm.\ B]{MultVar}. For the converse (hence a characterization of supercuspidals) one would need the nonvanishing of $\mathbb{V}^\vee\circ D^\vee_\Delta$ on supercuspidal irreducible admissible representations of $\GL_n(\Qp)$ over $\Fq$ and a generalization of Theorem \ref{thmE} below to maximal parabolics and representations $\pi$ not necessarily in $\SP_{\Fq}$.

In section \ref{sec:princser} we turn to the applications on possible successive extensions of principal series. Note that in case $n=2$ the kernel of Colmez' Montréal functor consists of the finite dimensional representations of $\GL_2(\Qp)$ over $\Fq$. We give a similar description of the irreducible subquotients of principal series representations on which the functor $D^\vee_\Delta$ vanishes. Successive extensions of such representations form a full subcategory $\Fin^{>B}_h(G)$ of $\SP_h(G)$ and we put $\overline{\SP}_h(G)$ for the quotient category with quotient functor $Q\colon \SP_h\to \overline{\SP}_h$. We have the following

\begin{thma}[Theorem \ref{fullyfaithful}]\label{thmD}
The functor $D^\vee_\Delta$ is fully faithful on $\overline{\SP}_h$ for all $h\geq 1$.
\end{thma}

Theorem \ref{thmD} generalizes the earlier result \cite[Theorem G]{MultVar} of the second named author where only successive extensions of \emph{irreducible} principal series (ie.\ objects in the category $\SP^0_h$) were considered. The key point in both proofs of the full property of $D^\vee_\Delta$ is to show the following. Assume $\pi$ is an object in $\SP_h$ whose image $Q(\pi)$ has length $r>1$ in the quotient category $\overline{\SP}_h$. Then one cannot embed the dual $\pi^\vee$ into the global sections of a sheaf on $G/B$ attached to a rank $1$ étale $T_+$-module over $E_\Delta$. Since $Q$ is fully faithful on the subcategory $\SP^0_h$, it sufficed to understand the $\Ext^1$ of two irreducible principal series and one could apply results of Hauseux \cite{Hauseux1,Hauseux}. However, extensions in the quotient category $\overline{\SP}_h$ may not come from genuine extensions of principal series in general, so in the present paper we had to rely on a precise quantitative version of Theorem \ref{thmA} (see Theorem \ref{DveeDeltagrw_0}).

In section \ref{sec:charparind} we apply the previous results to a characterization of smooth representations of $\GL_n(\Qp)$ that are parabolically induced from a representation of $\GL_2(\Qp)^{(\alpha)}\times T^{(\alpha)}$: 
\begin{thma}[Theorem \ref{lookslikeparindisparind}]\label{thmE}
Let $\pi$ be an object in $\SP_{\Fq}$. There is a smooth representation $\pi_\alpha$ of $\GL_2(\mathbb{Q}_p)^{(\alpha)}\times T^{(\alpha)}$ such that the image $Q(\pi)$ in $\overline{\SP}_{\mathbb{F}_q}$ is isomorphic to $Q(\Ind_{\overline{P}^{(\alpha)}}^{\GL_n(\mathbb{Q}_p)}\pi_\alpha)$ if and only if the subgroup $\GQpDa\leq \GQpD\times\Qp^\times$ acts on $\mathbb{V}^\vee(D^\vee_\Delta(\pi))$ via an abelian quotient.
\end{thma}
Theorem \ref{thmE} is proven via a generalization of Colmez' Jacquet functor \cite[sections IV.3, VII.1]{Colmez1} and uses the full force of the previous sections. In fact, we propose two possible generalizations of the Jacquet functor $J$ which coincide in case $G=\GL_2(\Qp)$ \cite[Prop.\ IV.3.2]{Colmez1}. One of them is simply the coinvariants $J_P(\pi)=\pi_U$ under the unipotent radical $U$ of the parabolic subgroup $P=LU\leq G$. The other one is a more complicated construction $J_P^{\mathrm{sheaf}}$ via the $G$-equivariant sheaf on $G/B$ associated to $\pi$. The precise relationship between these functors is mysterious, but we provide some partial results (see Corollary \ref{twojacquetsame} for the case when $L$ is the product of $\GL_2(\Qp)$ and a torus; and see Theorem \ref{torabel} for the case $P=B$ is a Borel). Theorem \ref{thmE} is the key (see Proposition \ref{onlyif}) in the proof of one direction of our Conjecture \ref{conimage} on the essential image of $D^\vee_\Delta$ restricted to $\SP_h$.

One application of the functor $\mathbb{V}^\vee\circ D^\vee_\Delta$ is a possible new method for computing (higher) extension classes of principal series. The idea is that extension groups $\Ext^k(\eta,\eta')$ of characters $\eta,\eta'$ of the group $\GQpD\times \Qp^\times$ can be identified with group cohomology groups $H^k(\GQpD\times\Qp^\times,\eta'\eta^{-1})$. Moreover, the latter can be explicitly computed by a generalization of Herr's complex \cite{PalZabradi,CKZ}. As an illustration how elementary techniques from group theory can be applied to the nonvanishing of certain higher $\Ext$ groups we reprove (see Theorem \ref{nolength3}) the mod $p$ case of a Theorem of Hauseux \cite[Thm.\ 3.2.1]{Hauseux3} on the nonexistence of certain length $3$ towers of principal series. We believe that this method can also be used to show the nonvanishing of some higher $\Ext$ groups of principal series recently proven by Heyer \cite{Heyer} using the geometrical lemma (see also \cite[Thm.\ 5.3.1]{Hauseux1} for a result depending on \cite[Conjecture 3.7.2]{EmertonOrd2} regarding the derived functor of Emerton's ordinary parts).

\subsection{The special case of $\GL_3(\Qp)$}

We briefly explain our results in the special case of $n=3$. In this case the set $\Delta$ of simple roots consists of two elements $\{\alpha,\beta\}=\Delta$ and there is a single non-simple positive root $\gamma\in \Phi^+\setminus \Delta$. The parabolic subgroups $\overline{P}^{(\alpha)}$ and $\overline{P}^{(\beta)}$ are maximal in this case. Assume that $\pi$ is a smooth irreducible representation of $\GL_3(\Qp)$. The basic question is what $\mathbb{V}^\vee\circ D^\vee_\Delta(\pi)$ is. A special case of Theorem \ref{thmC} above is that whenever $\pi$ is non-admissible then we simply have $\mathbb{V}^\vee\circ D^\vee_\Delta(\pi)=0$. Now assuming admissibility, we may use the classicfication \cite{Herzigclass} of irreducible admissible smooth representations of $\GL_3(\Qp)$ in characteristic $p$ to distinguish the following cases. If $\pi$ is parabolically induced from $\overline{P}^{(\alpha)}$ or from $\overline{P}^{(\beta)}$ then one can use the compatibility \cite[Thm.\ B]{MultVar} of the functor $D^\vee_\Delta$ with parabolic induction to reduce the computation of $D^\vee_\Delta(\pi)$ to the computation of Colmez' Montréal functor at representations of $\GL_2(\Qp)$. On the other hand, if $\pi$ is supercuspidal then not much was known before. The difficulty lies in the fact that we only know \cite{HKV} (quite recently) the mere existence of such representations, but no explicit examples. Unfortunately, we still do not know whether or not $D^\vee_\Delta(\pi)\neq 0$ for a single example of a smooth admissible irreducible supercuspidal representation of $\GL_3(\Qp)$ (or of $\GL_n(\Qp)$ for any $n>2$). However, if $D^\vee_\Delta(\pi)\neq 0$ then Theorem \ref{thmC} implies that $\mathbb{V}^\vee\circ D^\vee_\Delta(\pi)$ is finite dimensional, irreducible, and detects isomorphisms. Further, in case $n=3$ a refinement of Theorem \ref{thmE} leads to the following description of supercuspidals:

\begin{thma}[Corollary \ref{gl3supercuspidal}]\label{thmF}
Assume that $\pi$ is a smooth irreducible representation of $\GL_3(\Qp)$ such that $D^\vee_\Delta(\pi)\neq 0$. Then $\pi$ is supercuspidal if and only if the action of $\GQpa$ on $\mathbb{V}^\vee\circ D^\vee_\Delta(\pi)$ does \emph{not} factor through the maximal abelian quotient $\GQpa^{ab}$ for any $\alpha\in\Delta$.
\end{thma}

Since $\Fq$ is finite, any irreducible representation of $\GQpD\times \Qp^\times$ is of the form $\mathbb{V}^\vee\circ D^\vee_\Delta(\pi)\cong V_\alpha\otimes_{\Fq} V_\beta\otimes_{\Fq} V_{\Qp^\times}$ where $V_\alpha$ is an irreducible representation of $\GQpa$, $V_\beta$ is an irreducible representation of $G_{\Qp,\beta}$, and $V_{\Qp^\times}$ is an irreducible representation of $\Qp^\times$ (hence a character as $\Qp^\times$ is abelian). Theorem \ref{thmF} above shows that---under our standing assumption that $D^\vee_\Delta(\pi)\neq 0$---whenever $\pi$ is supercuspidal, we have $\dim_{\Fq} V_\alpha>1$ and $\dim_{\Fq}V_\beta>1$.

Finally, we describe our Conjecture (see \ref{conimage}) on the essential image of $\mathbb{V}^\vee\circ D^\vee_\Delta$ restricted to the category $\SP_{\Fq}$ in the special case of $G=\GL_3(\Qp)$. In view of Theorem \ref{thmD} we regard this as a conjecture \emph{on the possible shapes of successive extensions} of principal series representations. Let $\pi$ be an object in $\SP_{\Fq}$ (for simplicity, but note that in the text we also allow mod $p$-power coefficients) and put $V:=\mathbb{V}^\vee\circ D^\vee_\Delta(\pi)$. Then $V$ is a successive extension of $1$-dimensional representations of $\GQpD\times \Qp^\times$. The role of the two simple roots $\alpha,\beta$ is \emph{not} entirely symmetric here so in case our fixed Borel subgroup $B\leq \GL_3(\Qp)$ is that of the upper triangular matrices, we let $\alpha$ (resp.\ $\beta$) be the root sending a diagonal matrix $\diag(a,b,c)\in T$ ($a,b,c\in\Qp^\times$) to $ab^{-1}$ (resp.\ to $bc^{-1}$). At first consider the socle-filtration of $V$ restricted to $G_{\Qp,\beta}\times \Qp^\times$ and put $\mathrm{gr}^{\overline{\alpha}}_i(V)$ for the graded pieces ($i=1,\dots,r_\alpha$ for some integer) which are semisimple as representations of $G_{\Qp,\beta}\times \Qp^\times$, but not necessarily as representations of $\GQpa$. So $\mathrm{gr}^{\overline{\alpha}}_i(V)$ is the direct sum of its isotypical components $\mathrm{gr}^{\overline{\alpha}}_i(V)(\eta_{\overline{\alpha}})$ for some (finitely many) characters $\eta_{\overline{\alpha}}=\eta_\beta\otimes \eta_{\Qp^\times}\colon G_{\Qp,\beta}\times \Qp^\times\to \Fq^\times$. By Theorem \ref{thmE} (see Proposition \ref{onlyif}) the representation $\mathrm{gr}^{\overline{\alpha}}_i(V)(\eta_{\overline{\alpha}})$ of $\GQpa$ is in the image of Colmez' Montréal functor at some representation $\pi_\alpha(i,\eta_{\overline{\alpha}})$ of $\GL_2(\Qp)$ (necessarily a successive extension of principal series) admitting a central character which corresponds to $\eta_\beta$ via local class field theory. Similarly, we may take the socle filtration of $V$ restricted to $\GQpa\times \Qp^\times$ with graded pieces $\mathrm{gr}^{\overline{\beta}}_i(V)$ ($i=1,\dots,r_\beta$) and $\GQpa\times\Qp^\times$-isotypical components $\mathrm{gr}^{\overline{\beta}}_i(V)(\eta_{\overline{\beta}})$ for some characters $\eta_{\overline{\beta}}=\eta_\alpha\otimes\eta_{\Qp^\times}\colon \GQpa\times\Qp^\times\to \Fq^\times$. Proposition \ref{onlyif} means that for each $i=1,\dots,r_\beta$ and character $\eta_{\overline{\beta}}$ the representation $\mathrm{gr}^{\overline{\beta}}_i(V)(\eta_{\overline{\beta}})$ of $G_{\Qp,\beta}$ is in the image of Colmez' Montréal functor at some representation $\pi_\beta(i,\eta_{\overline{\beta}})$ of $\GL_2(\Qp)$ admitting a central character $\eta_{\Qp^\times}\widetilde{\eta}^{-1}_\alpha$ where $\widetilde{\eta}_\alpha\colon \Qp^\times\to\Fq^\times$ is the character corresponding to $\eta_\alpha$ via local class field theory. Conjecture \ref{conimage} states that the above conditions on $V$ are not only necessary, but also sufficient for $V$ to be in the image of $\mathbb{V}^\vee\circ D^\vee_\Delta$ restricted to $\SP_{\Fq}$.

\subsection{Notations and preliminaries}

\subsubsection{The group}
Let $G:=\GL_{n}(\mathbb{Q}_p)$ for some $n\geq 2$, $T\leq G$ be the subgroup of diagonal matrices, $B=TN\leq G$ be the Borel subgroup of upper triangular matrices containing $T$ with unipotent radical $N$. Further denote by $\Phi$ the set of roots of $T$, by $\Phi^+\subset \Phi$ the set of positive roots with respect to $B$, and by $\{\alpha_1,\dots,\alpha_{n-1}\}=\Delta\subset \Phi^+$ the set of simple roots so $|\Delta|=n-1$. For a positive root $\alpha\in \Phi^+$ we denote by $N_\alpha$ the root subgroup of $N$, ie.\ $T$ acts on $N_\alpha\cong\mathbb{Q}_p$ via the character $\alpha$. Let $N_0\leq N$ be the open compact subgroup consisting of strictly upper triangular matrices with entries in $\Zp$. Then $N_0$ is totally decomposed (ie.\ we have $N_0=\prod_{\alpha\in\Phi^+}(N_0\cap N_\alpha)$ for any fixed ordering of $\Phi^+$) and put $T_+:=\{t\in T\mid tN_0t^{-1}\leq N_0\}$, $T_0:=T_+\cap T_+^{-1}$, $B_+:=N_0T_+$. The subgroup $H_{\Delta,0}:=\prod_{\alpha\in\Phi^+\setminus\Delta}(N_0\cap N_\alpha)$ is normal in $N_0$ and invariant under conjugation by elements in $T_+$. So the quotient group $N_{\Delta,0}:=N_0/H_{\Delta,0}\cong \prod_{\alpha\in\Delta}N_{\alpha,0}$ also inherits the action of $T_+$. 

\subsubsection{The coefficient rings}
Let $F/\Qp$ be a finite extension with ring of integers $\mathcal{O}_F$, uniformizer $\varpi\in\mathcal{O}_F$ and residue field $\Fq:=\mathcal{O}_F/\varpi$. We consider smooth representations of $G$ on $\mathcal{O}_F/\varpi^h$-modules for $h\geq 1$. In the technical part of the paper we only need to treat the case $h=1$, ie.\ the coefficient field $\Fq$. The Iwasawa algebra $\mathcal{O}_F/\varpi^h\bs N_{\Delta,0}\js$ can be identified with the multivariable power series ring $\mathcal{O}_F/\varpi^h\bs X_\alpha\mid \alpha\in\Delta\js$ by the map $n_\alpha-1\mapsto X_\alpha$ ($\alpha\in\Delta$). We define $\mathcal{O}_F/\varpi^h\bg N_{\Delta,0}\jg$ as the localization $\mathcal{O}_F/\varpi^h\bs N_{\Delta,0}\js [X_\alpha^{-1},\alpha\in\Delta]$ and put $E_\Delta^+:=\Fq\bs X_\alpha\mid \alpha\in\Delta\js$, $E_\Delta:=E_\Delta^+[X_\Delta^{-1}]\cong \Fq\bg N_{\Delta,0}\jg$, $\OED:=\varprojlim_{h}\mathcal{O}_F/\varpi^h\bg N_{\Delta,0}\jg$. We also denote by $\varphi_t\colon \mathcal{O}_F/\varpi^h\bs N_{\Delta,0}\js\to \mathcal{O}_F/\varpi^h\bs N_{\Delta,0}\js$ (resp.\ $\varphi_t\colon \mathcal{O}_F/\varpi^h\bg N_{\Delta,0}\jg \to  \mathcal{O}_F/\varpi^h\bg N_{\Delta,0}\jg$) the induced action of $t\in T_+$ on these rings. By an \'etale $T_+$-module over $\mathcal{O}_F/\varpi^h\bg N_{\Delta,0}\jg$ we mean a (unless otherwise mentioned) finitely generated module $M$ over $\mathcal{O}_F/\varpi^h\bg N_{\Delta,0}\jg$ together with a semilinear action of the monoid $T_+$ (also denote by $\varphi_t$ for $t\in T_+$) such that the maps $$\id\otimes\varphi_t\colon \varphi_t^*M:=\mathcal{O}_F/\varpi^h\bg N_{\Delta,0}\jg\otimes_{\mathcal{O}_F/\varpi^h\bg N_{\Delta,0}\jg,\varphi_t}M\to M$$ are isomorphisms for all $t\in T_+$. For $j=1,\dots,n-1$ we put $t_j:=\diag(p,\dots,p,1,\dots,1)\in T_+$ where the first $j$ entries of $t_j$ in the diagonal equals $p$ and the las $n-j$ entries equal $1$. Further put $s:=t_1t_2\dots t_{n-1}\in T_+$.

\subsubsection{The functor $D^\vee_\Delta$}\label{notation:DveeDelta}
We consider the skew polynomial ring $\mathcal{O}_F/\varpi^h\bs N_{\Delta,0}\js [F_\Delta]:=\mathcal{O}_F/\varpi^h\bs N_{\Delta,0}\js [F_1,F_2,\dots,F_{n-1}]$ where the variables $F_j$ commute with each other and we have $F_j\lambda=(t_j\lambda t_j^{-1})F_j$ for $\lambda\in \mathcal{O}_F/\varpi^h\bs N_{\Delta,0}\js$ and $j=1,\dots,n-1$. For a smooth representation $\pi$ of $B$ over $\mathcal{O}_F/\varpi^h$ we denote by $\mathcal{M}_\Delta(\pi^{H_{\Delta,0}})$ the set of finitely generated $\mathcal{O}_F/\varpi^h\bs N_{\Delta,0}\js [F_\Delta]$-submodules of $\pi^{H_\Delta,0}$ that are stable under the action of $T_0$ and admissible as a representation of $N_{\Delta,0}=N_0/H_{\Delta,0}$. Here $F_j$ acts on $\pi^{H_\Delta,0}$ by the Hecke action of $t_j\in T_+$, ie.\ $F_j v:=\Tr_{H_{\Delta,0}/t_j H_{\Delta,0}t_j^{-1}}(t_j v)$ for $v\in\pi^{H_\Delta,0}$ and $j=1,\dots,n-1$. Then the functor $D^\vee_\Delta$ is defined by the projective limit $$D^\vee_\Delta(\pi):=\varprojlim_{M\in \mathcal{M}_\Delta(\pi^{H_{\Delta,0}})}M^\vee [1/X_{\Delta}]$$ where $X_\Delta=\prod_{\alpha\in\Delta}X_\alpha$ is the product of all the variables $X_\alpha=n_\alpha-1$ in the power series ring $\mathcal{O}_F/\varpi^h\bs N_{\Delta,0}\js$. Finally, we put $\mathcal{M}^0_\Delta(\pi^{H_{\Delta,0}})\subset \mathcal{M}_\Delta(\pi^{H_{\Delta,0}})$ for the subset of those objects $M\in \mathcal{M}_\Delta(\pi^{H_{\Delta,0}})$ such that $M^\vee$ is $X_\Delta$-torsion free, ie.\ the map $M^\vee\to M^\vee[X_\Delta^{-1}]$ is injective.

\subsubsection{The noncommutative deformations}
We define the ring $\mathcal{O}_F/\varpi^h\bg N_{\Delta,\infty}\jg$ as a projective limit $\varprojlim_k \mathcal{O}_F/\varpi^h\bg N_{\Delta,k}\jg$ where the finite layers $\mathcal{O}_F/\varpi^h\bg N_{\Delta,k}\jg:=\mathcal{O}_F/\varpi^h\bs N_{\Delta,k}\js[\varphi_s^{kn_0}(X_\Delta)^{-1}]$ are defined as localisations of the Iwasawa algebra $\mathcal{O}_F/\varpi^h\bs N_{\Delta,k}\js$. Here the group $N_{\Delta,k}:=N_0/H_{\Delta,k}$ is the extension of $N_{\Delta,0}$ by a finite $p$-group $H_{\Delta,0}/H_{\Delta,k}$ where $H_{\Delta,k}$ is the smallest normal subgroup in $N_0$ containing $s^kH_{\Delta,0}s^{-k}$. We have a functor $\mathbb{M}_{k,0}$ from the category $\mathcal{D}^{et}(T_+,\mathcal{O}_F/\varpi^h\bg N_{\Delta,0}\jg)$ of finitely generated \'etale $T_+$-modules over $\mathcal{O}_F/\varpi^h\bg N_{\Delta,0}\jg$ to the category $\mathcal{D}^{et}(T_+,\mathcal{O}_F/\varpi^h\bg N_{\Delta,k}\jg)$ of finitely generated \'etale $T_+$-modules over $\mathcal{O}_F/\varpi^h\bg N_{\Delta,k}\jg$. We set $\mathbb{M}_{\infty,0}:=\varprojlim_k\mathbb{M}_{k,0}$ and $\mathbb{D}_{0,\infty}$ to be the functor from the category $\mathcal{D}^{et}(T_+,\mathcal{O}_F/\varpi^h \bg N_{\Delta,\infty}\jg)$ of finitely generated \'etale $T_+$-modules over $\mathcal{O}_F/\varpi^h\bg N_{\Delta,\infty}\jg$ to $\mathcal{D}^{et}(T_+,\mathcal{O}_F/\varpi^h \bg N_{\Delta,0}\jg)$ induced by the reduction map $\mathcal{O}_F/\varpi^h\bg N_{\Delta,\infty}\jg\to \mathcal{O}_F/\varpi^h\bg N_{\Delta,0}\jg$. By \cite[Theorem D]{MultVar} the functors $\mathbb{M}_{\infty,0}$ and $\mathbb{D}_{0,\infty}$ are quasi-inverse equivalences of categories. We denote the natural projection map (see \cite[Lemma 4.1]{MultVar}) by $\pr_{0,k}\colon \mathbb{M}_{k,0}(D)\twoheadrightarrow D$ and by $\pr_{0,\infty}\colon \mathbb{M}_{\infty,0}(D)\twoheadrightarrow D$. The maps $\pr_{0,k}$ and $\pr_{0,\infty}$ are $\varphi_t$-equivariant for each $t\in T_+$ and $k\geq 0$. Note that for $\infty>k>0$ the map $\varphi_t\colon \mathbb{M}_{k,0}(D)\to \mathbb{M}_{k,0}(D)$ is not injective in general. However, $\varphi_t$ is injective on both $D$ and $\mathbb{M}_{\infty,0}(D)$ and admits a distinguished left inverse $\psi_t^{(D)}\colon D\to D$ and $\psi_t\colon \mathbb{M}_{\infty,0}(D)\to \mathbb{M}_{\infty,0}(D)$. Further, we have 
\begin{align}\label{psireduction}
\psi_t^{(D)}(\pr_{0,\infty}(x))=\pr_{0,\infty}\left(\sum_{u\in J(H_{\Delta,0}/tH_{\Delta,0}t^{-1})}\psi_t(u^{-1}x)\right)
\end{align}
for any set $J(H_{\Delta,0}/tH_{\Delta,0}t^{-1})\subset H_{\Delta,0}$ of representatives of $H_{\Delta,0}/tH_{\Delta,0}t^{-1}$. For $b=ut\in N_0T_+=B_+$ and $x\in \mathbb{M}_{\infty,0}(D)$ we put $\varphi_b(x):=u\varphi_t(x)$ and $\psi_b(x):=\psi_t(u^{-1}x)$.

\subsubsection{Subrings corresponding to Weyl elements}
By an abuse of notation we identify the Weyl group $N_G(T)/T$ with the subgroup of permutation matrices in $G$. For any element $w\in N_G(T)/T$ we put $\Phi^+_w:=\Phi^+\cap w(\Phi^+)$, $\Delta_w:=\Delta\cap w(\Phi^+)$, $N_{0,w}:=N_0\cap wN_0w^{-1}=\prod_{\alpha\in \Phi^+_w}(N_0\cap N_\alpha)$,  $H_{\Delta,k,w}:=H_{\Delta,k}\cap N_{0,w}$, and $N_{\Delta,k,w}:=N_{0,w}/H_{\Delta,k,w}$. We define the ring $\Fq\bg N_{\Delta,\infty,w}\jg$ as a projective limit $\varprojlim_k \Fq\bg N_{\Delta,k,w}\jg$ where the finite layers $\Fq\bg N_{\Delta,k,w}\jg:=\Fq\bs N_{\Delta,k,w}\js[\varphi_s^{kn_0}(X_{\Delta_w})^{-1}]$ are defined as localisations of the Iwasawa algebra $\Fq\bs N_{\Delta,k,w}\js$ by where $X_{\Delta_w}:=\prod_{\alpha\in\Delta_w}X_\alpha$. Since $N_{\Delta,k,w}$ is a subgroup in $N_{\Delta,k}$ and localization is exact, we have a natural injective ring homomorphism $\Fq\bg N_{\Delta,k,w}\jg\hookrightarrow \Fq\bg N_{\Delta,k}\jg$. Taking the limit as $k\to \infty$ we realize $\Fq\bg N_{\Delta,\infty,w}\jg$ as a subring in $\Fq\bg N_{\Delta,\infty}\jg$.

\subsubsection{Various categories of representations}

For any integer $h\geq 1$ we denote by $\Rep_h(\GQpD\times \mathbb{Q}_p^\times):=\Rep_{\mathcal{O}_F/\varpi^h}(\GQpD\times \mathbb{Q}_p^\times)$ the category of continuous representations of $\GQpD\times \mathbb{Q}_p^\times$ on finitely generated $\mathcal{O}_F/\varpi^h$-modules. In case $h=1$ we put $\Rep_{\Fq}(\GQpD\times \mathbb{Q}_p^\times):=\Rep_1(\GQpD\times \mathbb{Q}_p^\times)$. Further, we let $\Rep^{\varinjlim}_h(\GQpD\times \mathbb{Q}_p^\times)$ be the category of countable inductive limits of objects in $\Rep_h(\GQpD\times \mathbb{Q}_p^\times)$. Fix $\alpha\in \Delta$ and a character $\xi\colon \GQpDa\times \mathbb{Q}_p^\times\to \mathbb{F}_q^\times$. Denote by $\Rep^{\alpha,\xi}_{\mathbb{F}_q}(\GQpD\times \mathbb{Q}_p^\times)$ (resp.\ by $\Rep^{\alpha,ab}_{\mathbb{F}_q}(\GQpD\times \mathbb{Q}_p^\times)$) the full subcategory of $\Rep_{\mathbb{F}_q}(\GQpD\times \mathbb{Q}_p^\times)$ consisting of objects $V$ on which the subgroup $\GQpDa\times \mathbb{Q}_p^\times\leq \GQpD\times \mathbb{Q}_p^\times$ acts via the character $\xi$ (resp.\ acts via its maximal abelian quotient $\GQpDa^{ab}\times \mathbb{Q}_p^\times$). Further, put $\Rep^{\ord}_{\mathbb{F}_q}(\GQpD\times \mathbb{Q}_p^\times)$ (resp.\ $\Rep^{\ord}_h(\GQpD\times \mathbb{Q}_p^\times)$, resp.\ $\Rep^{\alpha,\xi,\ord}_{\mathbb{F}_q}(\GQpD\times \mathbb{Q}_p^\times)$, resp.\ $\Rep^{\alpha,ab,\ord}_{\mathbb{F}_q}(\GQpD\times \mathbb{Q}_p^\times)$) for the full subcategory of $\Rep_{\mathbb{F}_q}(\GQpD\times \mathbb{Q}_p^\times)$ (resp.\ of $\Rep_h(\GQpD\times \mathbb{Q}_p^\times)$, resp.\ of $\Rep^{\alpha,\xi}_{\mathbb{F}_q}(\GQpD\times \mathbb{Q}_p^\times)$, resp.\ of $\Rep^{\alpha,ab}_{\mathbb{F}_q}(\GQpD\times \mathbb{Q}_p^\times)$) whose objects are successive extensions of $1$-dimensional representations over $\Fq$. 

For any integer $h\geq 1$ let $\Rep_h(G):=\Rep_{\mathcal{O}_F/\varpi^h}(G)$ be the category of smooth representations of $G$ on modules over $\mathcal{O}_F/\varpi^h$. We denote by $\Rep_h^{fin}(G)$ the full subcategory of finite length representations. Further, we put $\SP_h:=\SP_h(G):=\SP_{\mathcal{O}_F/\varpi^h}(G)$ for the full subcategory of $\Rep_h^{fin}(G)$ whose objects are representations with all Jordan-Hölder factors isomorphic to subquotients of principal series representations. The combination of the main results of \cite{MultVar} and \cite{MultVarGal} (see also \cite{CKZ}) gives rise to the composite left exact \cite[Thm.\ A]{MultVar} functor
$$\mathbb{V}^\vee\circ D^\vee_\Delta\colon \Rep_h(G)\to \Rep_h^{\varinjlim}(\GQpD\times\Qp^\times)\ .$$
By Theorem \ref{thmA} the restriction of $\mathbb{V}^\vee\circ D^\vee_\Delta$ to $\Rep_h^{fin}(G)$ lands in the full subcategory $\Rep_h(\GQpD\times\Qp^\times)$. Recall \cite[Corollary 3.4 and Theorem 3.19]{MultVar}, \cite[Theorem 3.15]{MultVarGal} that the functor $\mathbb{V}^\vee\circ D^\vee_\Delta$ is exact on $\SP_h(G)$ and lands in the subcategory $\Rep^{\ord}_h(\GQpD\times \mathbb{Q}_p^\times)$. Let $\SP_{\mathbb{F}_q}^{\alpha,\xi}(G)$ (resp.\ $\SP_{\mathbb{F}_q}^{\alpha,ab}(G)$) be the full subcategory of $\SP_{\mathbb{F}_q}(G)$ whose objects are representations $\pi$ in $\SP_{\mathbb{F}_q}(G)$ such that $\mathbb{V}^\vee\circ D^\vee_\Delta(\pi)$ lies in $\Rep^{\alpha,\xi,\ord}_{\mathbb{F}_q}(\GQpD\times \mathbb{Q}_p^\times)$ (resp.\ in $\Rep^{\alpha,ab,\ord}_{\mathbb{F}_q}(\GQpD\times \mathbb{Q}_p^\times)$).

\subsection{Acknowledgements}

G.\ J.\ was supported by a PhD Scholarship of the EPSRC at the University of Bristol. G.\ Z.\ was supported by an NKFIH (National Research, Development and Innovation Office) Grant K-135885, by the MTA–RI Lendület ``Momentum'' Analytic Number Theory and Representation Theory Research Group, and by the János Bolyai Scholarship of the Hungarian Academy of Sciences.

We thank Christophe Breuil, Pierre Colmez, Elmar Gro{\ss}e-Klönne, Nataniel Marquis, Stefano Morra, and Peter Schneider for feedback. We especially thank Peter Schneider for pointing us to the recent preprint \cite{Heyer} of Heyer, along with Stefano Morra and Christophe Breuil for useful remarks on the exposition.

\section{Finiteness properties of generalized Montréal functors}\label{sec:finiteness}

\subsection{Lattices of multivariable $(\varphi_\Delta,\Gamma_\Delta)$-modules}

The goal here is to define the analogues of Colmez' treillis $D^\#$ and $D^{\nr}$ and to show their first properties in the context of $(\varphi_\Delta,\Gamma_\Delta)$-modules over $\mathbb{F}_q\bs X_\alpha,\alpha\in\Delta\js[X_\Delta^{-1}]$. 

Let $D \in D^{et}(\varphi_\Delta, \Gamma_\Delta, E_\Delta)$, and fix some $\alpha \in \Delta$.
Then for any $d \in D$, there exist unique $d_i \in D$, such that $d = \sum_{i = 0}^{p-1} (1 + X_\alpha)^i\varphi_\alpha^{(D)}(d_i)$. We define $\psi_\alpha(d) = d_0$. 
Then $\psi_\alpha^{(D)} : D \to D$ is an $\mathbb F_p$-linear map, and it satisfies that $\psi_\alpha^{(D)} \circ \varphi_\alpha^{(D)} = id_D$, and for any $\beta \in \Delta$, $\gamma \in \Gamma_\beta \simeq \mathbb Z_p$, $\psi_\alpha^{(D)}$ commutes with $\varphi_\beta^{(D)}, \psi_\beta^{(D)}$, and $\gamma_\beta^{(D)}$. More generally, if $S \subseteq \Delta$, we have maps $\psi_S^{(D)} : D \to D$, such that $\psi_S^{(D)} \circ \varphi_S^{(D)} = id_D$. $\psi_S^{(D)}$ is simply the composition of the $\psi_\alpha^{(D)}$ with $\alpha \in S$. 
Explicitly, if $S = \{\alpha_1, \ldots, \alpha_N\}$, for any $d \in D$, we have the unique expansion $$d = \sum_{i_1, \ldots, i_N = 0}^{p-1}\prod_{k = 1}^N(1+X_{\alpha_k})^{i_k} \varphi_S^{(D)}(d_{i_1 \ldots i_N}),$$
and $\psi_S^{(D)}(d) = d_{0 \ldots 0}$.
Note that $\psi_S(E_\Delta^+) \subseteq E_\Delta^+$. 

We construct a lattice in $D$ with good properties.
In the following, we often omit the superscript and use the notation $\psi_S$ (resp.\ $\varphi_S$) instead of $\psi_S^{(D)}$ (resp.\ instead of $\varphi_S^{(D)}$) if it does not lead to confusion.

\begin{pro}\label{Dhashalpha}
For each $\alpha\in\Delta$ there exists a maximal element $D^\#_\alpha$ among the finitely generated $E_\Delta^+[X_{\Delta\setminus\{\alpha\}}^{-1}]$-submodules $D_0$ of $D$ such that $\psi_\alpha(D_0)=D_0$. We have $D=D^\#_\alpha[X_\alpha^{-1}]$ and for any $x\in D$ there exists an integer $k>0$ such that $\psi_\alpha^k(x)\in D^\#_\alpha$. Further, $D^\#_\alpha$ is stable under the action of $\Gamma_\Delta$, $\varphi_\beta$, and $\psi_\beta$ for all $\beta\neq \alpha\in\Delta$. In particular, we have $D^\#_\alpha=\bigoplus_{j=0}^{p-1}(1+X_\beta)^j\varphi_\beta(D^\#_\alpha)$ for all $\beta\neq \alpha\in\Delta$.
\end{pro}
\begin{proof}
By \cite[Proposition 3.2.5]{PalZabradi} and the following remark, 
    $D$ has a basis $e_1, \ldots e_N$ (depending on $\alpha$) 
    such that the matrices of $\varphi_\alpha$ and of $\gamma_\alpha\in \Gamma_\alpha$ have all entries in $\mathbb{F}_q \bg X_\alpha \jg$. In other words, the $\mathbb{F}_q \bg X_\alpha \jg$-span $D_{\alpha,0}$ of $e_1,\dots,e_N$ is an étale $(\varphi_\alpha,\Gamma_\alpha)$-module over $\mathbb{F}_q \bg X_\alpha \jg$ such that $D\cong E_{\Delta}\otimes {\mathbb{F}_q \bg X_\alpha \jg}D_{\alpha,0}$ as $(\varphi_\alpha,\Gamma_\alpha)$-modules. Note, however, that $D_{\alpha,0}$ is not stable under the action of $\varphi_\beta$ and $\Gamma_\beta$ for $\alpha\neq \beta\in \Delta$. By \cite[Proposition II.4.2]{Colmez2} there is a maximal $\mathbb{F}_q\bs X_\alpha\js$-lattice $D^\#_{\alpha,0}\subset D_{\alpha,0}$ such that $\psi_\alpha\colon D^\#_{\alpha,0}\to D^\#_{\alpha,0}$ is onto. We claim that $D^\#_\alpha:=E_\Delta^+[X_{\Delta\setminus\{\alpha\}}^{-1}]D^\#_{\alpha,0}$ satisfies the properties in the statement.

    Since $D^\#_{\alpha,0}$ is finitely generated over $\mathbb{F}_q\bs X_\alpha\js$, $D^\#_\alpha$ is finitely generated over $E_\Delta^+[X_{\Delta\setminus\{\alpha\}}^{-1}]$. Also, for any $\lambda\in E_\Delta^+[X_{\Delta\setminus\{\alpha\}}^{-1}]$ we compute $\psi_\alpha(\varphi_\alpha(\lambda)x)=\lambda\psi_\alpha(x)$ showing that $\psi_\alpha\colon D^\#_\alpha\to D^\#_\alpha$ is onto as so is $\psi_\alpha\colon D^\#_{\alpha,0}\to D^\#_{\alpha,0}$. Further, we have $D^\#_\alpha[X_\alpha^{-1}]=E_\Delta D^\#_{\alpha,0}=E_\Delta D_{\alpha,0}=D$. Now any $x\in D$ can be written as $\lambda y$ for some $\lambda\in E^+_\Delta[X_{\Delta\setminus\{\alpha\}}^{-1}]$ and $y\in D_{0,\alpha}$. Using \cite[Proposition 3.2.5]{PalZabradi} again we obtain that there exists an integer $k>0$ such that $\psi_\alpha^k((1+X_\alpha)^jy)\in D^\#_{\alpha,0}$ for all $j=0,\dots,p^k-1$ since there is a treillis containing $(1+X_\alpha)^jy$ for all $j$. Then we write $\lambda=\sum_{j=0}^{p^k-1}(1+X_\alpha)^j\varphi_\alpha^k(\lambda_j)$ for some $\lambda_0,\dots,\lambda_{p^k-1}\in E^+_\Delta[X_{\Delta\setminus\{\alpha\}}^{-1}]$ so that $$\psi_\alpha^k(x)=\psi_\alpha^k\left(\sum_{j=0}^{p^k-1}(1+X_\alpha)^j\varphi_\alpha^k(\lambda_j)y\right)=\sum_{j=0}^{p^k-1}\lambda_j\psi_\alpha^k((1+X_\alpha)^jy)\in E_\Delta^+[X_{\Delta\setminus\{\alpha\}}^{-1}]D^\#_{\alpha,0}=D^\#_\alpha\ .$$
    Finally if $D_0$ is a finitely generated $E_\Delta^+[X_{\Delta\setminus\{\alpha\}}^{-1}]$-submodule of $D$ with generators $x_1,\dots,x_r$ such that $\psi_\alpha(D_0)=D_0$ then by the previous discussion there is an integer $k\gg 0$ such that $\psi_\alpha^k((1+X_\alpha)^jx_i)\in D^\#_\alpha$ for all $j=0,\dots,p^k-1$ and $i=1,\dots,r$. Hence we deduce $\psi_\alpha^k(\sum_{i=1}^r \lambda_i x_i)\in D^\#_\alpha$ for any $\lambda_1,\dots,\lambda_r\in E_\Delta^+[X_{\Delta\setminus\{\alpha\}}^{-1}]$ whence we obtain $D_0=\psi_\alpha^k(D_0)\subseteq D^\#_\alpha$. Finally, the $E_\Delta^+[X_{\Delta\setminus\{\alpha\}}^{-1}]$-submodules $D^?:=\gamma(D^\#_\alpha), \psi_\beta(D^\#_\alpha), \sum_{j=0}^{p-1}(1+X_\beta)^j\varphi_\beta(D^\#_\alpha)$ are all finitely generated with $\psi_\alpha(D^?)=D^?$ showing that $D^?\subseteq D^\#_\alpha$, ie.\ $D^\#_\alpha$ is stable under the action of $\gamma\in\Gamma_\Delta$, $\varphi_\beta$, and $\psi_\beta$ for all $\beta\neq \alpha\in\Delta$. The last statement follows from the stability under $\varphi_\beta$ and $\psi_\beta$ and the étale property of the action $\varphi_\beta$ on $D$.
\end{proof}

We put $D^\#:=\bigcap_{\alpha\in\Delta}D^\#_\alpha$.

\begin{pro}\label{Dhashprop}
$D^\#\subset D$ is a finitely generated $E_\Delta^+$-submodule stable under $\Gamma_\Delta$ and $\psi_\alpha^D$ for all $\alpha\in\Delta$ such that $D=D^\#[X_\Delta^{-1}]$.
\end{pro}
\begin{proof}
By Proposition \ref{Dhashalpha}, $D^\#_\beta$ is stable under $\psi_\alpha^D$ and $\Gamma_\Delta$ for all $\alpha,\beta\in \Delta$ whence so is their intersection. Further, if $x\in D$ is arbitrary then we find an integer $k_\alpha\geq 0$ such that $X_\alpha^{k_\alpha}x\in D^\#_\alpha$ so that $X_\Delta^{\max_\alpha k_\alpha}x\in D^\#$. Finally, pick any basis $e_1,\dots,e_N$ of $D$. Now $D^\#_\alpha$ is finitely generated over $E_\Delta^+[X_{\Delta\setminus\{\alpha\}}^{-1}]$, say by $x_1,\dots,x_r$. Then we may write $x_j=\sum_{i=1}^N \lambda_{i,j}e_i$ for each $j\in \{1,\dots,r\}$ and $\lambda_{i,j}\in E_\Delta=E_\Delta^+[X_\Delta^{-1}]$. So there is an integer $k_\alpha\geq 0$ such that $X_\alpha^{k_\alpha}\lambda_{i,j}\in E_\Delta^+[X_{\Delta\setminus\{\alpha\}}^{-1}]$ for all $i$ and $j$. So if $x\in D^\#$ then writing $x=\sum_{i=1}^N\lambda_i e_i$ we have $X_\alpha^{k_\alpha}\lambda_i\in E_\Delta^+[X_{\Delta\setminus\{\alpha\}}^{-1}]$ whence $X_\Delta^{\max_\alpha k_\alpha}\lambda_i\in \bigcap_{\alpha\in\Delta} E_\Delta^+[X_{\Delta\setminus\{\alpha\}}^{-1}]=E_\Delta^+$. This shows $D^\#$ is contained in the $E_\Delta^+$-submodule generated by $\frac{1}{X_\Delta^{\max_\alpha k_\alpha}}e_1,\dots,\frac{1}{X_\Delta^{\max_\alpha k_\alpha}}e_N$ whence $D^\#$ is finitely generated as $E_\Delta^+$ is noetherian.
\end{proof}

\begin{rem}
The relationship between our construction of $D^\#$ and that of \cite{EGK} is not clear.
\end{rem}

We also need the analogue of Colmez' \cite[section II.2]{Colmez2} $D^{\nr}$ in this multivariable setting. Let $S\subseteq \Delta$ be any subset and put $D_S^{\nr}:=\bigcap_{k\geq 1}\varphi_S^k(D)$.

The multivariable analogue of $E^{sep}$ is defined \cite{MultVarGal} as
\begin{align*}
E^{sep}_\Delta:=E^{sep}_{\alpha_1}\otimes_{E_{\alpha_1}}\left(E^{sep}_{\alpha_2}\otimes_{E_{\alpha_2}}\left(\cdots(E^{sep}_{\alpha_{n-1}}\otimes_{E_{\alpha_{n-1}}}E_\Delta)\right)\right) \cong \varinjlim_{E_\alpha\leq E'_\alpha\leq E^{sep}_\alpha,\forall\alpha\in\Delta}E'_\Delta
\end{align*}
where we put (see \cite[Lemma 3.2]{MultVarGal}) 
\begin{align*}
E'_\Delta\cong E'_{\alpha_1}\otimes_{E_{\alpha_1}}\left(E'_{\alpha_2}\otimes_{E_{\alpha_2}}\left(\cdots(E'_{\alpha_{n-1}}\otimes_{E_{\alpha_{n-1}}}E_\Delta)\right)\right)
\end{align*}
and $\Delta=\{\alpha_1,\dots,\alpha_{n-1}\}$. This does not depend on the ordering of the finite set $\Delta$. The quasi-inverse equivalences of categories $\mathbb{D}_\Delta$ and $\mathbb{V}_\Delta$ \cite{MultVarGal,CKZ} between representations of $\GQpD$ over $\Fp$ and multivariable $(\varphi,\Gamma)$-modules over $E_\Delta$ are given as
\begin{align*}
\mathbb{D}_\Delta(V):=(E^{sep}_\Delta\otimes_{\mathbb{F}_p}V)^{\HQpD}\qquad\text{and}\qquad
\mathbb{V}_\Delta(D):=\bigcap_{\alpha\in\Delta}\left(E_\Delta^{sep}\otimes_{E_\Delta}D\right)^{\varphi_\alpha=\id}\ .
\end{align*}

\begin{lem}\label{intersectphiEsep}
For any subset $S\subseteq \Delta$ we have $$\bigcap_{k\geq 1}\varphi_S^k(E^{sep}_\Delta)=F^{sep}_{\alpha_1}\otimes_{F_{\alpha_1}}\left(F^{sep}_{\alpha_2}\otimes_{F_{\alpha_2}}\left(\cdots(F^{sep}_{\alpha_{n-1}}\otimes_{F_{\alpha_{n-1}}}E_{\Delta\setminus S})\right)\right)$$
where we put $F_{\alpha_j}:=\Fp$ (hence $F^{sep}_{\alpha_j}=\overline{\Fp}$) if $\alpha_j$ is in $S$ and $F_{\alpha_j}:=E_{\alpha_j}$ if $\alpha_j\in\Delta\setminus S$ ($j=1,\dots,n-1$).
\end{lem}
\begin{proof}
This follows from the concrete description (see \cite[section 3.1]{MultVarGal}) of $E'_\Delta$ as
\begin{align*}
E'_\Delta\cong (\bigotimes_{\alpha\in\Delta,\mathbb{F}_p}\mathbb{F}_{q_\alpha})\bs X'_\alpha,\alpha\in \Delta\js[X_\Delta^{-1}]
\end{align*}
by applying the one-variable analogue $\bigcap_{k\geq 1}\varphi^k(E^{sep})=\overline{\Fp}$ (see \cite[Prop.\ II.2.2(i)]{Colmez2}) for the coefficient of the monomial $\prod_{\alpha\in \Delta\setminus S}{X'_\alpha}^{j_\alpha}$ as an element in $E'_{S}$ for all indices $(j_\alpha)_{\alpha\in\Delta\setminus S}\in\mathbb{Z}^{\Delta\setminus S}$.
\end{proof}

\begin{pro}\label{Dnrmultvar}
$D_S^{\nr}$ is a free module over $E_{\Delta\setminus S}$ of rank $\dim_{\Fq}\mathbb{V}_\Delta(D)^{G'_{\Qp,S}}$ for any subset $S\subseteq\Delta$ where $G'_{\Qp,S}\leq G_{\Qp,S}$ is the commutator subgroup. Moreover, the map $E_\Delta\otimes_{E_{\Delta\setminus S}}D^{\nr}_S\to D$ induced by the inclusion $D^{\nr}_S\hookrightarrow D$ is always injective and it is bijective if and only if $G_{\Qp,S}$ acts on $\mathbb{V}_\Delta(D)$ via its maximal abelian quotient $G_{\Qp,S}^{ab}$.
\end{pro}
\begin{proof}
This is completely analogous to the one-variable case \cite[Prop.\ II.2.2(i)]{Colmez2}. By the local Kronecker--Weber Theorem $G'_{\Qp,\alpha}$ is a normal subgroup in $H_{\Qp,\alpha}$ with quotient $\Gamma_{\nr,\alpha}:=H_{\Qp,\alpha}/G'_{\Qp,\alpha}\cong \Gal(\overline{\Fp}/\Fp)$. Put $\Gamma_{\nr,S}:=\prod_{\alpha\in S}\Gamma_{\nr,\alpha}$, $V:=\mathbb{V}_\Delta(D)$ and compute
\begin{align*}
\bigcap_{k\geq 1}\varphi_S^k(D)=\bigcap_{k\geq 1}\varphi_S^k(E^{sep}_\Delta\otimes_{\mathbb{F}_p}V)^{\HQpD}=\left(\bigcap_{k\geq 1}\varphi_S^k(E^{sep}_\Delta)\otimes_{\mathbb{F}_p}V\right)^{\HQpD}=\\
=\left(E^{sep}_{\Delta\setminus S} \otimes_{\mathbb{F}_p}\left((\bigotimes_{\alpha\in S,\Fp}\overline{\Fp})\otimes_{\Fp}V^{G'_{\Qp,S}}\right)^{\Gamma_{\nr,S}}\right)^{H_{\Qp,\Delta\setminus S}}=\mathbb{D}_{\Delta\setminus S}\left(\left((\bigotimes_{\alpha\in S,\Fp}\overline{\Fp})\otimes_{\Fp}V^{G'_{\Qp,S}}\right)^{\Gamma_{\nr,S}}\right)
\end{align*}
using Lemma \ref{intersectphiEsep}. Finally, 
\begin{align*}
\left((\bigotimes_{\alpha\in S,\Fp}\overline{\Fp})\otimes_{\Fp}V^{G'_{\Qp,S}}\right)^{\Gamma_{\nr,S}}
\end{align*}
is a representation of $G_{\Qp,\Delta\setminus S}$ of dimension $\dim_{\Fp}V^{G'_{\Qp,S}}$ over $\Fp$ by Hilbert Theorem 90. The first part follows since $\mathbb{D}_{\Delta\setminus S}$ is an equivalence of categories preserving rank \cite[Thm.\ 3.15]{MultVarGal}. The injectivity of $E_\Delta\otimes_{E_{\Delta\setminus S}}D^{\nr}_S\to D$ also follows from the isomorphism
\begin{align*}
(\bigotimes_{\alpha\in S,\Fp}\overline{\Fp})\otimes_{\Fp}\left((\bigotimes_{\alpha\in S,\Fp}\overline{\Fp})\otimes_{\Fp}V^{G'_{\Qp,S}}\right)^{\Gamma_{\nr,S}}\cong (\bigotimes_{\alpha\in S,\Fp}\overline{\Fp})\otimes_{\Fp}V^{G'_{\Qp,S}}
\end{align*}
obtained also by Hilbert Theorem 90.

Finally, if $G_{\Qp,S}$ acts on $\mathbb{V}_\Delta(D)$ via its quotient $G_{\Qp,S}^{ab}$, we have $V^{G_{\Qp,S}'}=V$. Hence the rank of $D$ equals the rank of $E_\Delta\otimes_{E_{\Delta\setminus S}}D^{\nr}_S$ as free modules over $E_\Delta$ by Proposition \ref{Dnrmultvar}. Hence the cokernel $C$ of the map $E_\Delta\otimes_{E_{\Delta\setminus S}}D^{\nr}_S\hookrightarrow D$ is a torsion $E_\Delta$-module. Therefore the global annihilator of $C$ is a nonzero $\Gamma_\Delta$-invariant ideal in $E_\Delta$, so $C=0$ by \cite[Prop.\ 2.1]{MultVar}.
\end{proof}

\subsection{Lattices for noncommutative $(\varphi_\Delta,\Gamma_\Delta)$-modules}

Recall that $N_{\Delta,0}=N_0/H_{\Delta,0}$ decomposes as a direct product of subgroups $N_0^{(\alpha)}$ ($\alpha\in \Delta$). The set of simple roots $\Delta$ of $G=\GL_n(\Qp)$ can be described as $\Delta=\{\alpha_1,\dots,\alpha_{n-1}\}$ where $\alpha_j(\diag(a_1,\dots,a_n))=\frac{a_i}{a_{i+1}}$. Therefore its Iwasawa-algebra $\Fq\bs N_{\Delta,0}\js$ with coefficients in the finite field $\Fq$ is isomorphic to the power series ring $\Fq\bs X_{\alpha_1},\dots,X_{\alpha_{n-1}}\js=E_\Delta^+$. In particular, we may identify $E_\Delta=E_\Delta^+[X_\Delta^{-1}]$ with a localization $\mathbb{F}_q\bs N_{\Delta,0}\js [X_\Delta^{-1}]=:\mathbb{F}_q\bg N_{\Delta,0}\jg$ of the Iwasawa algebra. On the other hand, $T_+$ consists diagonal matrices $t\in T$ with $\alpha_j(t)\in\Zp\setminus\{0\}$ for all $j=1,\dots,n-1$. The elements $t_j:=\diag(\underbrace{p,\dots,p}_j,\underbrace{1,\dots,1}_{n-j})$ satisfy $\alpha_j(t_j)=p$ and $\alpha_i(t_j)=1$ for $i\neq j\in\{1,\dots,n-1\}$. So under the identification $E_\Delta\cong \Fq\bg N_{\Delta,0}\jg$, the operator $\varphi_{\alpha_j}$ corresponds to the conjugation by $t_j$. Further, for any element $\gamma=(\gamma_1,\dots,\gamma_{n-1})\in \Gamma_\Delta=\prod_{\alpha\in\Delta}\Gamma_\alpha$ we put $t_\gamma:=\diag(\chi_{cyc}(\gamma_1\dots\gamma_{n-1}),\chi_{cyc}(\gamma_2\dots\gamma_{n-1}),\dots,\chi_{cyc}(\gamma_{n-1}),1)$ satisfies $\alpha_j(t_\gamma)=\chi_{cyc}(\gamma_j)$, so the action of $\gamma$ corresponds to the conjugation by $t_\gamma$. Here $\chi_{cyc}\colon \Gamma_\alpha\cong \Gal(\Qp(\mu_{p^\infty})/\Qp)\to\Zp^\times$ denotes the cyclotomic character. In particular, $T_+$ is identified with the direct product $Z(G)\times\Gamma_\Delta\times\prod_{j\in \{1,\dots,n-1\}}t_j^{\mathbb{N}}$. We often use these identifications in the sequel without any further notice. Under these identifications (together with $Z(G)\cong \Qp^\times$) we have
\begin{thm}[Thm.\ 3.15 in \cite{MultVarGal}, Thm.\ 1.1 in \cite{CKZ}]\label{galequivcat}
The functors $\mathbb{V}$ and $\mathbb{D}$ are quasi-inverse equivalences of categories between the category $\mathcal{D}^{et}(T_+,\Fq\bg N_{\Delta,0}\jg)$ of étale $T_+$-modules over $\Fq\bg N_{\Delta,0}\jg$ and the category $\Rep_{\Fq}(\GQpD\times \Qp^\times)$ of continuous finite dimensional representations of the group $\GQpD\times \Qp^\times$ over $\Fq$.
\end{thm}

Let $D$ be an étale $T_+$-module over $\mathbb{F}_q\bg N_{\Delta,0}\jg$. Assume that $D_0\subset D$ is a finitely generated $\mathbb{F}_q\bs N_{\Delta,0}\js$-submodule such that $D_0[X_\Delta^{-1}]=D$ and $\psi_t^{(D)}(D_0)\subseteq D_0$ for all $t\in T_+$. We call such a submodule $D_0\subset D$ a \emph{$\psi^{(D)}$-invariant lattice in $D$}. Note that the étale property for the action of $t\in T_+$ on $D$ means that the map
\begin{eqnarray*}
\id\otimes\varphi_t\colon \mathbb{F}_q\bs N_{\Delta,0}\js\otimes_{\mathbb{F}_q\bs N_{\Delta,0}\js,\varphi_t}D &\to& D\\
\sum_{u\in J(N_0/tN_0t^{-1}H_{\Delta,0})}u\otimes x_{u,t}&\mapsto& \sum_{u\in J(N_0/tN_0t^{-1}H_{\Delta,0})}u\varphi_t(x_{u,t})=:x
\end{eqnarray*}
is a bijection. Here $x_{u,t}=\psi_t^{(D)}(u^{-1}x)$. So (once $D_0$ is stable under the action of the group $N_{\Delta,0}$) the condition that $\psi_t^{(D)}(D_0)\subseteq D_0$ is equivalent to the condition that $$D_0\subseteq \id\otimes\varphi_t(\mathbb{F}_q\bs N_{\Delta,0}\js\otimes_{\mathbb{F}_q\bs N_{\Delta,0}\js,\varphi_t}D_0)=\mathbb{F}_q\bs N_{\Delta,0}\js\varphi_t(D_0)\ .$$ Our goal is to show that there is a maximal compact $\mathbb{F}_q\bs N_0\js$-submodule $\mathbb{M}_\infty^{bd}(D_0)\subset \mathbb{M}_{\infty,0}(D)$ with the properties $(i)$ $\mathbb{M}_\infty^{bd}(D_0)$ is stable under $\psi_t$ for all $t\in T_+$ and $(ii)$ $\mathbb{M}_\infty^{bd}(D_0)$ maps into $D_0$ under the quotient map $\pr_{0,\infty}\colon\mathbb{M}_{\infty,0}(D)\twoheadrightarrow D$. It is not always possible to obtain a lift such that the map $\pr_{0,\infty}\colon \mathbb{M}_\infty^{bd}(D_0)\to D_0$ is onto.

Let $\mathbb{M}_k(D_0):=\pr_{0,k}^{-1}(D_0)\subset \mathbb{M}_{k,0}(D)$ be the full preimage of $D_0\subset D$ under the quotient map $\pr_{0,k}\colon\mathbb{M}_{k,0}(D)\twoheadrightarrow D$. Note that for $k>0$ the set $\mathbb{M}_k(D_0)$ is not compact (nor finitely generated over $\mathbb{F}_q\bs N_{\Delta,k}\js$ as it contains the kernel of $\pr_{0,k}$). Define
\begin{equation*}
\mathbb{M}_k^{bd}(D_0):=\bigcap_{t\in T_+}\mathbb{F}_q\bs N_{\Delta,k}\js\varphi_t(\mathbb{M}_k(D_0))\ .
\end{equation*}

\begin{lem}\label{bdk_compact}
$\mathbb{M}_k^{bd}(D_0)$ is a finitely generated $\mathbb{F}_q\bs N_{\Delta,k}\js$-module. In particular, it is compact.
\end{lem}
\begin{proof}
Since $H_{\Delta,k}\lhd H_{\Delta,0}$ has finite index, there exists a $t_0\in T_+$ such that $t_0H_{\Delta,0}t_0^{-1}\leq H_{\Delta,k}$ so that $\varphi_{t_0}(\mathbb{M}_{k}(D_0))\cong \varphi_{t_0}(D_0)$ is finitely generated as a module over $\mathbb{F}_q\bs t_0N_{\Delta,0}t_0^{-1}\js\cong \mathbb{F}_q\bs t_0N_{\Delta,k}t_0^{-1}\js$ whence $\mathbb{F}_q\bs N_{\Delta,k}\js\varphi_{t_0}(\mathbb{M}_k(D_0))$ is finitely generated as a module over $\mathbb{F}_q\bs N_{\Delta,k}\js$. Since $\mathbb{F}_q\bs N_{\Delta,k}\js$ is noetherian, the lemma follows.
\end{proof}

Now for any $\psi^{(D)}$-invariant lattice $D_0\subset D$ we put $\mathbb{M}_\infty^{bd}(D_0):=\varprojlim_k \mathbb{M}_{k}^{bd}(D_0)$. By Lemma \ref{bdk_compact} we deduce that $\mathbb{M}_\infty^{bd}(D_0)$ is compact. However, it is unclear at this point whether $\mathbb{M}_\infty^{bd}(D_0)\neq 0$.

\begin{pro}\label{M_infty_D0_psi_descr}
For any $\psi^{(D)}$-invariant lattice $D_0\subset D$ we have 
$$\mathbb{M}_\infty^{bd}(D_0)=\{x\in \mathbb{M}_{\infty,0}(D)\mid \pr_{0,\infty}(\psi_t(u^{-1}x))\in D_0\text{ for all }t\in T_+,u\in N_0\}\ .$$
In particular, $\psi_t(\mathbb{M}_\infty^{bd}(D_0))\subseteq \mathbb{M}_\infty^{bd}(D_0)$ for all $t\in T_+$.
\end{pro}
\begin{proof}
Let $x\in \mathbb{M}_\infty^{bd}(D_0)$, $t\in T_+$, and $k\geq 0$ be arbitrary. By putting $x_u:=\psi_t(u^{-1}x)$ for $u\in N_0$ we may write
$$x=\sum_{u\in J(N_0/tN_0t^{-1})}u\varphi_t(x_u)$$
for any set $J(N_0/tN_0t^{-1})\subset N_0$ of representatives of $N_0/tN_0t^{-1}$. Let $t'\in T_+$ be arbitrary and choose $k_1\geq k$ so that $H_{\Delta,k_1}\subseteq tt'H_{\Delta,0}(tt')^{-1}$. This means that $|N_{\Delta,k_1}:tt'N_{\Delta,k_1}(tt')^{-1}|=|N_0:tt'N_0(tt')^{-1}|$. Hence $J(N_{\Delta,k_1}/tt'N_{\Delta,k_1}(tt')^{-1}):=\{\pr_{k_1,\infty}(utvt^{-1})\mid u\in J(N_0/tN_0t^{-1}),v\in J(N_0/t'N_0{t'}^{-1})\}$ is a set of representatives of $N_{\Delta,k_1}/tt'N_{\Delta,k_1}(tt')^{-1}$ if $J(N_0/t'N_0{t'}^{-1})$ is a set of representatives of $N_0/t'N_0{t'}^{-1}$. We put $x_{u,v}:=\psi_{t'}(v^{-1}x_u)$ and compute
$$x_u=\sum_{v\in J(N_0/t'N_0{t'}^{-1})}v\varphi_{t'}(x_{u,v})
$$
whence
\begin{align*}
\pr_{k_1,\infty}(x)=\sum_{u\in J(N_0/tN_0t^{-1})}\pr_{k_1,\infty}(u)\varphi_t\left(\sum_{v\in J(N_0/t'N_0{t'}^{-1})}\pr_{k_1,\infty}(v)\varphi_{t'}(\pr_{k_1,\infty}(x_{u,v}))\right)=\\
=\sum_{\substack{u\in J(N_0/tN_0t^{-1})\\ v\in J(N_0/t'N_0{t'}^{-1})}}\pr_{k_1,\infty}(utvt^{-1})\varphi_{tt'}(\pr_{k_1,\infty}(x_{u,v}))\ .
\end{align*}
On the other hand, since $x\in \mathbb{M}_\infty^{bd}(D_0)$, we have 
\begin{align*}
\pr_{k_1,\infty}(x)\in\mathbb{M}_{k_1}^{bd}(D_0)\subseteq \id\otimes\varphi_{tt'}(\mathbb{F}_q\bs N_{\Delta,k_1}\js\otimes_{\mathbb{F}_q\bs N_{\Delta,k_1}\js,\varphi_{tt'}}\mathbb{M}_{k_1}(D_0))=\\
=\bigoplus_{w\in J(N_{\Delta,k_1}/tt'N_{\Delta,k_1}(tt')^{-1})}w\varphi_{tt'}(\mathbb{M}_{k_1}(D_0))\ ,
\end{align*}
so we may write 
\begin{align*}
\pr_{k_1,\infty}(x)=\sum_{\substack{u\in J(N_0/tN_0t^{-1})\\ v\in J(N_0/t'N_0{t'}^{-1})}}\pr_{k_1,\infty}(utvt^{-1})\varphi_{tt'}(\widetilde{x_{u,v}})
\end{align*}
for some $\widetilde{x_{u,v}}\in \mathbb{M}_{k_1}(D_0)$ ($u\in J(N_0/tN_0t^{-1}), v\in J(N_0/t'N_0{t'}^{-1})$). By the étale property of $\mathbb{M}_{k_1,0}(D)$, we have $\pr_{k_1,\infty}(x_{u,v})-\widetilde{x_{u,v}}\in \Ker(\varphi_{tt'})\subseteq \Ker(\pr_{0,k_1})$ showing $\pr_{0,\infty}(x_{u,v})=\pr_{0,k_1}(\pr_{k_1,\infty}(x_{u,v}))=\pr_{0,k_1}(\widetilde{x_{u,v}})\in D_0$, ie.\ $\pr_{k_1,\infty}(x_{u,v})\in \mathbb{M}_{k_1}(D_0)$ for all $u\in J(N_0/tN_0t^{-1}), v\in J(N_0/t'N_0{t'}^{-1})$. Therefore we deduce
\begin{align*}
\pr_{k,\infty}(x_u)=\sum_{v\in J(N_0/t'N_0{t'}^{-1})}\pr_{k,\infty}(v)\varphi_{t'}(\pr_{k,k_1}(\pr_{k_1,\infty}(x_{u,v})))\in\\
\in\mathbb{F}_q\bs N_{\Delta,k}\js\varphi_{t'}(\pr_{k,k_1}(\mathbb{M}_{k_1}(D_0)))\subseteq \mathbb{F}_q\bs N_{\Delta,k}\js\varphi_{t'}(\mathbb{M}_k(D_0))\ .
\end{align*}
Since $t'\in T_+$ was arbitrary, we obtain $\pr_{k,\infty}(x_u)\in \mathbb{M}_k^{bd}(D_0)$ for all $k\geq 0$ whence we also have $\psi_t(u^{-1}x)=x_u\in \mathbb{M}_\infty^{bd}(D_0)$.

For the other containment assume $x\in \mathbb{M}_{\infty,0}(D)$ satisfies $\pr_{0,\infty}(\psi_t(u^{-1}x))\in D_0$ for all $t\in T_+$ and $u\in N_0$. Put $x_u:=\psi_t(u^{-1}x)$ and compute
\begin{align*}
\pr_{k,\infty}(x)=\sum_{u\in J(N_0/tN_0t^{-1})}\pr_{k,\infty}(u)\varphi_t(\pr_{k,\infty}(x_u))\in \mathbb{F}_q\bs N_{\Delta,k}\js\varphi_{t}(\mathbb{M}_k(D_0))
\end{align*}
showing $\pr_{k,\infty}(x)\in \mathbb{M}_k^{bd}(D_0)$ for all $k\geq 0$ and therefore $x\in \mathbb{M}_\infty^{bd}(D_0)$.
\end{proof}

The map $\psi_t\colon \mathbb{M}_\infty^{bd}(D_0)\to \mathbb{M}_\infty^{bd}(D_0)$ may not be onto for all $t\in T_+$ even if we assume $\psi_t^{(D)}\colon D_0\to D_0$ to be onto. However, one can take $\bigcap_{t\in T_+}\psi_t(\mathbb{M}_\infty^{bd}(D_0))$ for which this property is automatic. 

\begin{lem}\label{nilpDeltaminusalpha}
For any $\alpha\in\Delta$ there is a system of neighbourhoods of $1$ in $H_{\Delta,0}$ consisting of subgroups of the form $s_{\Delta\setminus\{\alpha\}}^kH_{\Delta,0}s_{\Delta\setminus\{\alpha\}}^{-k}$.
\end{lem}
\begin{proof}
For $\GL_n(\mathbb{Q}_p)$ this can be seen directly, but here is a more conceptual proof: since $N_0$ is totally decomposed, so is $H_{\Delta,0}$ whence we may write $H_{\Delta,0}=\prod_{\delta\in\Phi^+\setminus\Delta}(N_0\cap N_\delta)$. Any nonsimple root $\delta$ is the sum of certain simple roots, but they are not multiples of $\alpha$. So $\delta$ has a constituent $\beta\neq\alpha\in\Delta$ in which case $s_\beta(N_0\cap N_\delta)s_\beta^{-1}\subseteq (N_0\cap N_\delta)^p$. Is particular, we have $s_{\Delta\setminus\{\alpha\}}^kH_{\Delta,0}s_{\Delta\setminus\{\alpha\}}^{-k}\subseteq \prod_{\delta\in\Phi^+\setminus\Delta}(N_0\cap N_\delta)^{p^k}$.
\end{proof}

\begin{pro}\label{prcontainedinDhash}
Assume that $M\subset \mathbb{M}_\infty(D)$ is a compact $\mathbb{F}_q\bs N_0\js$-submodule stable under $\psi_t$ for all $t\in T_+$ such that $\psi_t\colon M\to M$ is onto for all $t\in T_+$. Then $M$ is contained in $\bigcap_{t\in T_+}\psi_t(\mathbb{M}_\infty^{bd}(D^\#))$. In particular, $\bigcap_{t\in T_+}\psi_t(\mathbb{M}_\infty^{bd}(D^\#))$ is the largest compact $\mathbb{F}_q\bs N_0\js$-submodule on which each operator $\psi_t$ ($t\in T_+$) acts surjectively.
\end{pro}
\begin{proof}
Let $M$ be as in the statement and put $D_0:=\pr_{0,\infty}(M)$. Then $D_0$ is a compact $E_\Delta^+$-submodule (being the continuous image of the compact $M$) of $D$ hence it is finitely generated over $E_\Delta^+$. Further, by \eqref{psireduction} we have $\psi^{(D)}_t(D_0)\subseteq D_0$ for all $t\in T_+$. Moreover, for any $x\in M$, we also have $\psi_t(u^{-1}x)\in M$ for all $u\in N_0$ and $t\in T_+$. Therefore we obtain $x\in \mathbb{M}^{bd}_\infty(D_0)$ using Proposition \ref{M_infty_D0_psi_descr}. Since $\psi_t\colon M\to M$ is onto, we deduce $M\subseteq \bigcap_{t\in T_+}\psi_t(\mathbb{M}_\infty^{bd}(D_0))$. It remains to show $D_0\subseteq D^\#$.

Let $\alpha\in\Delta$ be arbitrary. Since $D_0$ is finitely generated, there exists an integer $m>0$ such that $(\psi^{(D)}_\alpha)^m(D_0)\subseteq D^\#_\alpha$. Pick an element $\overline{x}\in D_0$, so there is an $x\in M$ such that $\pr_{0,\infty}(x)=\overline{x}$. Further, $\psi_\alpha^m\colon M\to M$ is onto, so we may choose an element $y\in M$ with $\psi_\alpha^m(y)=x$. On the other hand, there is an integer $k>0$ such that $s_{\Delta\setminus\{\alpha\}}^kH_{\Delta,0}s_{\Delta\setminus\{\alpha\}}^{-k}\subseteq s_\alpha^m H_{\Delta,0}s_\alpha^{-m}$ by Lemma \ref{nilpDeltaminusalpha}. Put $\overline{t}:=s_{\Delta\setminus\{\alpha\}}^k$, so we compute 
\begin{align*}
\psi_{\overline{t}}^{(D)}(\overline{x})=\psi_{\overline{t}}^{(D)}(\pr_{0,\infty}(x))=\psi_{\overline{t}}^{(D)}(\pr_{0,\infty}(\psi_\alpha^m(y))\overset{\eqref{psireduction}}{=}\\
=\sum_{u\in J(H_{\Delta,0}/\overline{t}H_{\Delta,0}\overline{t}^{-1})}\pr_{0,\infty}\left(\psi_{\overline{t}}(u^{-1}\psi_\alpha^m(y))\right)=\sum_{u\in J(H_{\Delta,0}/\overline{t}H_{\Delta,0}\overline{t}^{-1})}\pr_{0,\infty}\left(\psi_{\overline{t}}(\psi_\alpha^m(s_\alpha^mu^{-1}s_\alpha^{-m}y))\right)\overset{u_1=s_\alpha^mus_\alpha^{-m}}{=}\\
=\sum_{u_1\in J(s_\alpha^mH_{\Delta,0}s_\alpha^{-m}/s_\alpha^m\overline{t}H_{\Delta,0}\overline{t}^{-1}s_\alpha^{-m})}\pr_{0,\infty}\left(\psi_{\overline{t}}(\psi_\alpha^m(u_1^{-1}y))\right)\overset{u_1=vw}{=}\\
=\sum_{v\in J(s_\alpha^mH_{\Delta,0}s_\alpha^{-m}/\overline{t}H_{\Delta,0}\overline{t}^{-1})}
\sum_{w\in J(\overline{t}H_{\Delta,0}\overline{t}^{-1}/s_\alpha^m\overline{t}H_{\Delta,0}\overline{t}^{-1}s_\alpha^{-m})}
\pr_{0,\infty}\left(\psi_\alpha^m(\psi_{\overline{t}}(w^{-1}v^{-1}y))\right)\overset{w_1=\overline{t}^{-1}w\overline{t}}{=}\\
=\sum_{v\in J(s_\alpha^mH_{\Delta,0}s_\alpha^{-m}/\overline{t}H_{\Delta,0}\overline{t}^{-1})}
\sum_{w_1\in J(H_{\Delta,0}/s_\alpha^mH_{\Delta,0}s_\alpha^{-m})}
\pr_{0,\infty}\left(\psi_\alpha^m(w_1^{-1}\psi_{\overline{t}}(v^{-1}y))\right)\overset{\eqref{psireduction}}{=}\\
=\sum_{v\in J(s_\alpha^mH_{\Delta,0}s_\alpha^{-m}/\overline{t}H_{\Delta,0}\overline{t}^{-1})}(\psi_\alpha^{(D)})^m(\pr_{0,\infty}(\psi_{\overline{t}}(v^{-1}y))\in (\psi^{(D)}_\alpha)^m(D_0)\subseteq D^\#_\alpha\ .
\end{align*}

Since $\overline{x}\in D_0$ was arbitrary, we may apply it to $\left(\prod_{\beta\in\Delta\setminus\{\alpha\}}(1+X_\beta)^{j_\beta}\right)\overline{x}$ in place of $\overline{x}$ ($j_\beta=0,\dots,p^k-1$, $\beta\in\Delta\setminus\{\alpha\}$), too so that $$x\in \bigoplus_{(j_\beta)_{\beta\in\Delta\setminus\{\alpha\}}\in \{0,1\dots,p^m-1\}^{\Delta\setminus\{\alpha\}}}\left(\prod_{\beta\in\Delta\setminus\{\alpha\}}(1+X_\beta)^{j_\beta}\right)\varphi_{\overline{t}}(D^\#_\alpha)=D^\#_\alpha\ .$$ Since $\alpha\in\Delta$ was arbitrary, we deduce $D_0\subseteq D^\#=\bigcap_{\alpha\in\Delta}D^\#_\alpha$ as claimed.
\end{proof}

\begin{rem}
In view of Proposition \ref{prcontainedinDhash} and \cite[Lemma 4.16]{MultVar}, if we have $\bigcap_{t\in T_+}\psi_t(\mathbb{M}_\infty^{bd}(D^\#))=0$ for some étale $T_+$-module $D$ over $E_\Delta$ then there does not exist a smooth representation $\pi$ of $\GL_n(\mathbb{Q}_p)$ and a surjective morphism $D^\vee_\Delta(\pi)\twoheadrightarrow D$ of étale $T_+$-modules. In particular, $D$ cannot be in the image of the functor $D^\vee_\Delta$. At this point we cannot rule out this possibility for general $D$.
\end{rem}

\begin{que}
Let $D$ be an étale $T_+$-module  over $E_\Delta$. Is the map $\pr_{0,\infty}\colon \bigcap_{t\in T_+}\psi_t(\mathbb{M}_\infty^{bd}(D^\#))\to D^\#$ onto?
\end{que}

\subsection{Finiteness properties of noncommutative lattices}

Our main result in this section is Theorem \ref{Mbdfingen} and its corollaries.

 The group $N_0$ has an open uniform subgroup $U$ which contains $s^kN_0s^{-k}$ for $k\gg 0$ as the latter form a basis of neighbourhoods of $1$. However, the conjugation by $s^k$ induces an isomorphism $N_0\cong s^kN_0s^{-k}$ along which one can pull back any $p$-valuation on $U$ (in the sense of \cite[Chapter V]{MR2810332}) obtaining a $p$-valuation on $N_0$ (see also \cite[Lemma 1]{exact}\footnote{There is a misprint in the definition of the $p$-valuation $\omega$ on $N_0$ in \cite{exact}. The correct formula should read $\omega(g):=\min_{\alpha\in\Phi^+}(v_p(g_\alpha)+m_\alpha)$.} for a direct construction). Moreover, we may restrict this $p$-valuation to any subgroup in $N_0$. If we choose the $p$-valuation this way, it induces a filtration on the Iwasawa algebra $\mathbb{F}_q\bs H_{\Delta,0}\js$ such that the graded ring is a polynomial ring in $d:=|\Phi^+\setminus\Delta|=\frac{(n-2)(n-1)}{2}$ variables over $\mathbb{F}_q$. However, this way the grading is by the total degree but for technical reasons we need to equip $\mathbb{F}_q\bs H_{\Delta,0}\js$ with a filtration such that the graded pieces are $1$-dimensional over $\mathbb{F}_q$ generated by monomials, ie.\ such that the grading is by the lexicographic order. We achieve this as follows using that $N_0$ (hence $H_{\Delta,0}$) is nilpotent. Fix a total ordering $\leq$ of the set $\Phi^+$ refining the partial order by the degree: in particular, $\beta\leq \gamma$ if $\gamma-\beta$ is in $\Phi^+$ ($\beta,\gamma\in \Phi^+$), so we write $\Phi^+\setminus\Delta=\{\beta_1<\dots<\beta_d\}$. Pick a topological generator $u_i\in N_0\cap N_{\beta_i}$ ($i=1,\dots,d$) and put ${\bf b}^{\bf j}:=(u_1-1)^{j_1}\cdots(u_d-1)^{j_d}\in \mathbb{F}_q\bs H_{\Delta,0}\js$ for ${\bf j}=(j_1,\dots,j_d)\in (\mathbb{Z}^{\geq 0})^d$. Since $H_{\Delta,0}$ is homeomorphic to $\mathbb{Z}_p^d$, the elements of the Iwasawa algebra $\mathbb{F}_q\bs H_{\Delta,0}\js$ can uniquely be written as a noncommutative formal power series $$\sum_{{\bf j}\in(\mathbb{Z}^{\geq 0})^{d}}a_{\bf j}{\bf b}^{\bf j}$$
where $a_{\bf j}\in\mathbb{F}_q$ (see eg.\ \cite[Chapter IV]{MR2810332}). 
\begin{lem}\label{subgroupfilt}
The subgroups $H_r:=\prod_{i=r+1}^d(N_0\cap N_{\beta_i})$ are normal in $H_{\Delta,0}$ for all $1\leq r\leq d-1$ and are normalized by $N_0T_0$, too. The image of $N_0\cap N_{\beta_r}$ is contained in the center of the quotient group $H_{\Delta,0}/H_r$.
\end{lem}
\begin{proof}
This follows from our choice of ordering refining the partial order by the degree of positive roots using the commutator formula \cite[Proposition 8.2.3]{MR1642713} (see also the proof of \cite[Lemma 1]{exact}).
\end{proof}

In particular, the central elements $(u_r-1)^{j}$ ($j\geq 0$) generate two-sided ideals $I_r^{j}$ in the Iwasawa algebra $\mathbb{F}_q\bs H_{\Delta,0}/H_r\js$ such that the associated graded ring 
$$\gr\mathbb{F}_q\bs H_{\Delta,0}/H_r\js:=\bigoplus_{j=0}^\infty I_r^j/I_r^{j+1} $$
is isomorphic to the polynomial ring $\mathbb{F}_q\bs H_{\Delta,0}/H_{r-1}\js[Y_r]$ (putting $H_0:=H_{\Delta,0}$). We pull back all these filtrations to $\mathbb{F}_q\bs H_{\Delta,0}\js$ via the quotient map $\mathbb{F}_q\bs H_{\Delta,0}\js\to \mathbb{F}_q\bs H_{\Delta,0}/H_r\js$. More precisely, equip $(\mathbb{Z}^{\geq 0})^d$ with the lexicographic order, ie.\ ${\bf j}< {\bf j'}$ if ${\bf j}\neq {\bf j'}$ and $j_i<j'_i$ for the largest $i\in \{1,\dots,d\}$ such that $j_i\neq j'_i$ and define the filtration by two-sided ideals
\begin{align}
\Fil^{\bf j}\mathbb{F}_q\bs H_{\Delta,0}\js:=\left\{\sum_{{\bf j'}\in(\mathbb{Z}^{\geq 0})^{d}}a_{\bf j'}{\bf b}^{\bf j'}\mid a_{\bf j'}=0\text{ for all }{\bf j'}<{\bf j}\right\}\ ,\notag\\
\Fil^{{\bf j}+}\mathbb{F}_q\bs H_{\Delta,0}\js:=\left\{\sum_{{\bf j'}\in(\mathbb{Z}^{\geq 0})^{d}}a_{\bf j'}{\bf b}^{\bf j'}\mid a_{\bf j'}=0\text{ for all }{\bf j'}\leq {\bf j}\right\}\ .\label{multifilt}
\end{align}
\begin{pro}\label{multigraded}
The graded ring $\bigoplus_{{\bf j}\in (\mathbb{Z}^{\geq 0})^d}\Fil^{\bf j}/\Fil^{{\bf j}+}$ associated to the filtration \eqref{multifilt} is isomorphic to the polynomial ring $\mathbb{F}_q[Y_1,\dots,Y_d]$.
\end{pro}
\begin{proof}
The above filtration induces a filtration on the Iwasawa algebra $\mathbb{F}_q\bs H_{\Delta,0}/H_r\js$ for any $1\leq r\leq d$. We claim that the graded ring $\bigoplus_{{\bf j}\in (\mathbb{Z}^{\geq 0})^d}\Fil^{\bf j}\mathbb{F}_q\bs H_{\Delta,0}/H_r\js/\Fil^{{\bf j}+}\mathbb{F}_q\bs H_{\Delta,0}/H_r\js$ is isomorphic to $\mathbb{F}_q[Y_1,\dots,Y_r]$ for all $1\leq r\leq d$ (hence the statement in case $r=d$). We prove this by induction on $r$. If $r=1$ then the group $H_{\Delta,0}/H_1\cong N_0\cap N_{\beta_1}\cong \mathbb{Z}_p$ is commutative so there is nothing to prove. For any integer $r$ put ${\bf j}^{(r)}:=(j_1,\dots,j_r,0,\dots,0)$ for the $r$th truncation of ${\bf j}=(j_1,\dots,j_d)\in(\mathbb{Z}^{\geq 0})^d$. Let $r>1$. By Lemma \ref{subgroupfilt}, $u_r-1$ lies in the center of $H_{\Delta,0}/H_r$, so we have 
$${\bf b}^{{\bf j}^{(r)}}\cdot {\bf b}^{{\bf j'}^{(r)}}={\bf b}^{{\bf j}^{(r-1)}}\cdot {\bf b}^{{\bf j'}^{(r-1)}}(u_r-1)^{j_r+j'_r}$$
for all ${\bf j},{\bf j'}\in (\mathbb{Z}^{\geq 0})^d$. The claim follows from the inductional statement for $r-1$.
\end{proof}

The filtration \eqref{multifilt} induces a filtration 
\begin{align}
\Fil^{\bf j}\mathbb{M}_{\infty,0}(D):=(\Fil^{\bf j}\mathbb{F}_q\bs H_{\Delta,0}\js)\mathbb{M}_{\infty,0}(D) \qquad\text{and}\qquad 
\Fil^{{\bf j}+}\mathbb{M}_{\infty,0}(D):=(\Fil^{{\bf j}+}\mathbb{F}_q\bs H_{\Delta,0}\js)\mathbb{M}_{\infty,0}(D) \label{multfiltmod}
\end{align}
on $\mathbb{M}_{\infty,0}(D)$ by closed $\Fq\bs N_0\js$-submodules for any étale $T_+$-module $D$ over $E_\Delta$. In case ${\bf j}={\bf 0}=(0,\dots,0)$ we have $\Fil^{{\bf 0}+}\mathbb{M}_{\infty,0}(D)=\Ker(\pr_{0,\infty})$ by construction. Further, we have the following comparison of this filtration with the kernels of the projection maps $\pr_{k,\infty}\colon \mathbb{M}_{\infty,0}(D)\to \mathbb{M}_{k,0}(D)$.
\begin{lem}\label{comparefilpr}
For an integer $k\geq 0$ let $k_i\geq 0$ be the unique integer such that $H_{\Delta,k}\cap N_{\beta_i}=(H_{\Delta,0}\cap N_{\beta_i})^{p^{k_i}}$ and put ${\bf \tilde{k}}:=(p^{k_1}-1,\dots,p^{k_d}-1)$. We have $\Fil^{{\bf \tilde{k}}+}\mathbb{M}_{\infty,0}(D)\subset \Ker(\pr_{k,\infty})$, but $\Fil^{\bf \tilde{k}}\mathbb{M}_{\infty,0}(D)\not\subset \Ker(\pr_{k,\infty})$.
\end{lem}
\begin{proof}
If $j'_i\geq p^{k_i}$ for some $i\in \{1,\dots,d\}$ then ${\bf b}^{\bf j'}$ maps to $0$ in $\mathbb{F}_q[H_{\Delta,0}/H_{\Delta,k}]$ since $(u_{\beta_i}-1)^{j'_i}=(u_{\beta_i}-1)^{j'_i-p^{k_i}}(u_{\beta_i}^{p^{k_i}}-1)$ and $u_{\beta_i}^{p^{k_i}}$ is in $H_{\Delta,k}$. Hence we have $\Fil^{{\bf \tilde{k}}+}\mathbb{M}_{\infty,0}(D)\subset \Ker(\pr_{k,\infty})$. For the other statement note that $(u_{\beta_i}-1)^{p^{k_i}-1}=\frac{(u_{\beta_i}^{p^{k_i}}-1)}{u_{\beta_i}-1}=\sum_{r=0}^{p^{k_i}-1}u_{\beta_i}^r$ whence ${\bf b}^{\bf \tilde{k}}=\sum_{u\in J(H_{\Delta,0}/H_{\Delta,k})}u$ for some system of representatives $J(H_{\Delta,0}/H_{\Delta,k})$. In particular ${\bf b}^{\bf \tilde{k}}\mathbb{M}_{\infty,0}(D)\subseteq \Fil^{\bf \tilde{k}}\mathbb{M}_{\infty,0}(D)$, but ${\bf b}^{\bf \tilde{k}}\mathbb{M}_{\infty,0}(D)\not\subset \Ker(\pr_{k,\infty})$.
\end{proof}
We will also need the following continuity result.
\begin{lem}\label{intersect+}
For any compact $\Fq\bs N_0\js$-submodule $M\leq \mathbb{M}_{\infty,0}(D)$ and index ${\bf j}^{(0)}\in (\mathbb{Z}^{\geq 0})^d$ we have $$M+\Fil^{{\bf j}^{(0)}}\mathbb{M}_{\infty,0}(D)= \bigcap_{{\bf j'}<{\bf j}^{(0)}}(M+\Fil^{{\bf j'}+}\mathbb{M}_{\infty,0}(D))\ .$$
\end{lem}
\begin{proof}
The containment $\subseteq$ is trivial, let us prove the other containment. Moreover, if ${\bf j}^{(0)}$ is a successor of some ${\bf j'}$ then we have $\Fil^{{\bf j'}+}\mathbb{M}_{\infty,0}(D)=\Fil^{{\bf j}^{(0)}}\mathbb{M}_{\infty,0}(D)$, so there is nothing to prove. Otherwise we have $j^{(0)}_1=\dots=j^{(0)}_{r}=0$, but $j^{(0)}_{r+1}\neq 0$ for some $1\leq r< d$. Let $h^{(n)}_1=\dots=h^{(n)}_{r-1}=0$, $h^{(n)}_r=n$, $h^{(0)}_{r+1}=j^{(0)}_{r+1}-1,$ and $h^{(n)}_i=j^{(0)}_i$ for all $r+1<i\leq d$ so that ${\bf h}^{(1)}<{\bf h}^{(2)}<\cdots<{\bf h}^{(n)}<\cdots<{\bf j}^{(0)} \in (\mathbb{Z}^{\geq 0})^d$. Further, write $x=y_n+z_n$ with $y_n\in M$ and $z_n\in \Fil^{{\bf h}^{(n)}+}\mathbb{M}_{\infty,0}(D)$. Note that for any index ${\bf h}^{(n)}<{\bf j}\in (\mathbb{Z}^{\geq 0})^d$ we either have ${\bf j}^{(0)}\leq {\bf j}$ or $j_{i}=h^{(n)}_i$ for all $r\leq i\leq d$. In the latter case ${\bf b}^{\bf j}$ is a multiple of ${\bf b}^{{\bf h}^{(n)}}$ (as we have $h^{(n)}_i=0$ for all $1\leq i<r$). Therefore we may write $z_n={\bf b}^{{\bf h}^{(n)}}z_n'+w_n$ where $z_n'\in \mathbb{M}_{\infty,0}(D)$ and $w_n\in \Fil^{{\bf j}^{(0)}}\mathbb{M}_{\infty,0}(D)$. Since $M$ is compact, we may pass to a subsequence to achieve that $y_n$ is convergent. On the other hand we have ${\bf b}^{{\bf h}^{(n)}}z_n'\to 0$ regardless what $z_n'$ is since for all $k\geq 0$ we have $\pr_{k,\infty}({\bf b}^{{\bf h}^{(n)}}\mathbb{M}_{\infty,0}(D))=0$ for $n\geq p^{k_r}$ where $u_r^{p^{k_r}}\in H_{\Delta,k}$. Therefore $w_n=x-y_n-{\bf b}^{{\bf h}^{(n)}}z_n'$ is convergent, too. However, $\Fil^{{\bf j}^{(0)}}\mathbb{M}_{\infty,0}(D)$ is closed in $\mathbb{M}_{\infty,0}(D)$ whence $x=\lim y_n+\lim w_n\in M+\Fil^{{\bf j}^{(0)}}\mathbb{M}_{\infty,0}(D)$.
\end{proof}

By Proposition \ref{multigraded}, the multiplication by ${\bf b}^{\bf j}$ induces an isomorphism
\begin{align*}
{\bf b}^{\bf j}\colon D\cong \Fil^{{\bf 0}}\mathbb{M}_{\infty,0}(D)/\Fil^{{\bf 0}+}\mathbb{M}_{\infty,0}(D)\overset{\sim}{\to} \Fil^{{\bf j}}\mathbb{M}_{\infty,0}(D)/\Fil^{{\bf j}+}\mathbb{M}_{\infty,0}(D)
\end{align*}
for any ${\bf j}\in (\mathbb{Z}^{\geq 0})^d$. So under this identification we put 
\begin{align}
\Fil^{\bf j}\mathfrak{M}:=\mathfrak{M}\cap \Fil^{\bf j}\mathbb{M}_{\infty,0}(D),\qquad \Fil^{{\bf j}+}\mathfrak{M}:=\mathfrak{M}\cap \Fil^{{\bf j}+}\mathbb{M}_{\infty,0}(D)\ ,\notag\\
D^{\bf j}(\mathfrak{M}):=\Fil^{\bf j}\mathfrak{M}/\Fil^{{\bf j}+}\mathfrak{M}
\cong (\mathfrak{M}\cap \Fil^{\bf j}\mathbb{M}_{\infty,0}(D)+ \Fil^{{\bf j}+}\mathbb{M}_{\infty,0}(D))/\Fil^{{\bf j}+}\mathbb{M}_{\infty,0}(D) \subset D \label{filtonMbd}
\end{align}
for any compact $\Fq\bs N_0\js$-submodule $\mathfrak{M}\leq \mathbb{M}_{\infty,0}(D)$ and let $D^{\bf j}:=D^{\bf j}(\mathbb{M}^{bd}_\infty(D^\#))$.

\begin{lem}\label{strictlybiggerjbiggerDj}
Let $\mathfrak{M}\leq \mathbb{M}_{\infty,0}(D)$ be a compact $\Fq\bs N_0\js$-submodule stable under the action of $T_0$. Then $D^{\bf j}(\mathfrak{M})$ is an $E_\Delta^+$-submodule in $D$ stable under the action of $T_0$ for all ${\bf j}\in (\mathbb{Z}^{\geq 0})^d$. Further, we have $D^{\bf j}(\mathfrak{M})\subseteq D^{{\bf j}+{\bf j'}}(\mathfrak{M})$ for all ${\bf j'}\in (\mathbb{Z}^{\geq 0})^d$.
\end{lem}
\begin{proof}
By lemma \ref{subgroupfilt} the filtration $\Fil^\bullet$ on $\mathbb{F}_q\bs H_{\Delta,0}\js$ is stable under conjugation by $N_0T_0$. Hence $\Fil^{\bf j}\mathbb{M}_{\infty,0}(D)$ and $\Fil^{{\bf j}+}\mathbb{M}_{\infty,0}(D)$ are $\mathbb{F}_q\bs N_0\js$-submodules in $\mathbb{M}_{\infty,0}(D)$ stable under the action of $T_0$. The first part of the statement follows from this since $\mathfrak{M}$ is also a $\mathbb{F}_q\bs N_0\js$-submodule stable under the action of $T_0$. The multiplication by ${\bf b}^{\bf j'}$ induces an injective homomorphism $D^{\bf j}(\mathfrak{M})\hookrightarrow D^{{\bf j}+{\bf j'}}(\mathfrak{M})$ of $E_\Delta^+$-modules.
\end{proof}

We call a subset $Y\subseteq D$ \emph{bounded} if it is contained in some finitely generated $E_\Delta^+$-submodule $D_0\leq D$.

\begin{pro}\label{fingen+kertors}
Let $\mathfrak{M}\leq \mathbb{M}_{\infty,0}(D)$ be a compact $\Fq\bs N_0\js$-submodule stable under the action of $T_0$. Assume that $\bigcup_{{\bf j}\in (\mathbb{Z}^{\geq 0})^d}D^{\bf j}(\mathfrak{M})$ is bounded in $D$. Then we have \begin{enumerate}[$(a)$]
\item $\mathfrak{M}$ is finitely generated as a module over $\Fq\bs N_0\js$. 
\item Assume, in addition, that we have $D^{\bf 0}(\mathfrak{M})[X_\Delta^{-1}]=D$. Then the kernel of the map $\Fq\bs N_{\Delta,0}\js \otimes_{\Fq\bs N_0\js}\mathfrak{M}\to D^{\bf 0}(\mathfrak{M})$ is killed by some power of $X_\Delta$. In particular, we have $\Fq\bg N_{\Delta,0}\jg \otimes_{\Fq\bs N_0\js}\mathfrak{M}\cong D$.
\end{enumerate}
\end{pro}
\begin{proof}
Let $\mathfrak{M}\leq \mathbb{M}_{\infty,0}(D)$ be a compact $\Fq\bs N_0\js$-submodule stable under the action of $T_0$ and assume that $D^{\bf j}(\mathfrak{M})\leq D_0$ for all ${\bf j}\in (\mathbb{Z}^{\geq 0})^d$ and for some finitely generated $E_\Delta^+$-submodule $D_0$ in $D$.

Consider the directed graph on the set $(\mathbb{Z}^{\geq 0})^d$ of vertices where we have an edge from ${\bf j}$ to ${\bf j'}\in (\mathbb{Z}^{\geq 0})^d$ if and only if all the following $3$ conditions are satisfied:
\begin{enumerate}[$(i)$]
\item $1+\sum_{i=1}^dj_i=\sum_{i=1}^dj'_i$;
\item $j_i\leq j'_i$ for all $i\in\{1,\dots,d\}$;
\item There exists an index ${\bf j''}\in (\mathbb{Z}^{\leq 0})^d$ such that $j'_i\leq j''_i$ for all $i\in\{1,\dots,d\}$ and $D^{\bf j}(\mathfrak{M})\neq D^{\bf j''}(\mathfrak{M})$.
\end{enumerate}
Note that $j_i\leq j'_i\leq j''_i$ for all $i\in\{1,\dots,d\}$ implies $D^{\bf j}(\mathfrak{M})\subseteq D^{\bf j'}(\mathfrak{M})\subseteq D^{\bf j''}(\mathfrak{M})\subseteq D_0$ by Lemma \ref{strictlybiggerjbiggerDj}. Hence by the noetherian property of $E_\Delta^+$ this graph does not have any infinite path. Therefore by K\"onig's Lemma this graph can only contain finitely many edges. In particular, there exists an integer $r\geq 0$ such that for all ${\bf j},{\bf j'}\in (\mathbb{Z}^{\geq 0})^d$ with $j_i\leq j'_i$ and $\sum_{i=1}^dj_i=r$ we have $D^{\bf j}(\mathfrak{M})=D^{\bf j'}(\mathfrak{M})$. Pick a finite set $\{x_1^{({\bf j})},\dots,x_{N_{\bf j}}^{({\bf j})}\}$ of generators of $D^{\bf j}(\mathfrak{M})$ as $E_\Delta^+$-modules for all ${\bf j}\in (\mathbb{Z}^{\geq 0})^d$ with $\sum_{i=1}^dj_i\leq r$. Further, lift these elements to $\{y_1^{({\bf j})},\dots,y_{N_{\bf j}}^{({\bf j})}\}\subset \Fil^{\bf j}\mathfrak{M}$. We claim that the finite set 
\begin{align}\label{setofgen}
\bigcup_{\substack{{\bf j}\in (\mathbb{Z}^{\geq 0})^d\\ \sum_{i=1}^dj_i\leq r}}\{y_1^{({\bf j})},\dots,y_{N_{\bf j}}^{({\bf j})}\}
\end{align}
generates $\mathfrak{M}$. Let $M\leq \mathfrak{M}$ be the $\mathbb{F}_q\bs N_0\js$-submodule generated by \eqref{setofgen}. Assume that $\mathfrak{M}\not\subseteq M$. By compactness of $M$ we have $M=\varprojlim_k \pr_{k,\infty}(M)=\varprojlim_k (M+\Ker(\pr_{k,\infty})/\Ker(\pr_{k,\infty}))$. So there is a $k\geq 0$ such that $\mathfrak{M}\not\subseteq M+\Ker(\pr_{k,\infty})$. By Lemma \ref{comparefilpr} we also obtain $\mathfrak{M}\not\subseteq M+\Fil^{\bf \tilde{k}+}\mathbb{M}_{\infty,0}(D)$. Since the set $(\mathbb{Z}^{\geq 0})^d$ is well-ordered, there is a smallest ${\bf j}^{(0)}\in (\mathbb{Z}^{\geq 0})^d$ such that $\mathfrak{M}\not\subseteq M+\Fil^{{\bf j}^{(0)}+}\mathbb{M}_{\infty,0}(D)$. Pick an element $x\in \mathfrak{M}$ such that $x\notin M+\Fil^{{\bf j}^{(0)}+}\mathbb{M}_{\infty,0}(D)$. By the minimality of ${\bf j}^{(0)}$, we have $$x\in \bigcap_{{\bf j'}<{\bf j}^{(0)}}(M+\Fil^{{\bf j'}+}\mathbb{M}_{\infty,0}(D))\ .$$

By Lemma \ref{intersect+} we may write $x=y+z$ with $y\in M$ and $z\in \Fil^{{\bf j}^{(0)}}\mathbb{M}_{\infty,0}(D)$. Since $M\subseteq \mathfrak{M}$, we also have $z\in \mathfrak{M}$. There is a ${\bf j}^{(1)}$ with $j^{(1)}_i\leq j^{(0)}_i$ for all $i=1,\dots,d$ and $\sum_{i=1}^dj^{(1)}_i\leq r$ such that $D^{{\bf j}^{(1)}}=D^{{\bf j}^{(0)}}$ so that we obtain elements $\lambda_1,\dots,\lambda_{N_{{\bf j}^{(1)}}}\in\mathbb{F}_q\bs N_0\js$ such that $$z-\sum_{\ell=1}^{N_{{\bf j}^{(1)}}}\lambda_\ell{\bf b}^{{\bf j}^{(0)}-{\bf j}^{(1)}}y_\ell^{{\bf j}^{(1)}}\in \Fil^{{\bf j}^{(0)}+}\mathbb{M}_{\infty,0}(D)$$
which contradicts to our assumption $x=y+z\notin M+\Fil^{{\bf j}^{(0)}+}\mathbb{M}_{\infty,0}(D)$ since $y+\sum_{\ell=1}^{N_{{\bf j}^{(1)}}}\lambda_\ell{\bf b}^{{\bf j}^{(0)}-{\bf j}^{(1)}}y_\ell^{{\bf j}^{(1)}}$ lies in $M$. This proves part $(a)$. Further, the argument above also shows $\Fil^{{\bf j}^{(0)}}\mathfrak{M}\subseteq (H_{\Delta,0}-1)\mathfrak{M}+\Fil^{{\bf j}^{(0)}+}\mathfrak{M}$ whence
\begin{align}\label{bigjequal}
(H_{\Delta,0}-1)\mathfrak{M}+\Fil^{{\bf j}^{(0)}}\mathfrak{M}=(H_{\Delta,0}-1)\mathfrak{M}+\Fil^{{\bf j}^{(0)}+}\mathfrak{M}
\end{align}
for all ${\bf j}^{(0)}\in (\mathbb{Z}^{\geq 0})^d$ with $j_1^{(0)}+\dots+j_d^{(0)}>r$ since we have ${\bf b}^{{\bf j}^{(0)}-{\bf j}^{(1)}}\in \Fil^{{\bf 0}+}\Fq\bs H_{\Delta,0}\js=(H_{\Delta,0}-1)\Fq\bs H_{\Delta,0}\js$ in this case as ${\bf j}^{(0)}-{\bf j}^{(1)}\neq {\bf 0}$.

For part $(b)$ assume, in addition, that we have $D^{\bf 0}(\mathfrak{M})[X_\Delta^{-1}]=D$. In particular, for all ${\bf j}\in(\mathbb{Z}^{\geq 0})^d$ there exists an integer $r_{\bf j}\in\mathbb{Z}^{\geq 0}$ such that $X_\Delta^{r_{\bf j}}D^{\bf j}(\mathfrak{M})\subseteq D^{\bf 0}(\mathfrak{M})$. We deduce
\begin{align}\label{smalljmultip}
X_\Delta^{r_{\bf j}}\left((H_{\Delta,0}-1)\mathfrak{M}+\Fil^{\bf j}\mathfrak{M}\right)\subseteq(H_{\Delta,0}-1)\mathfrak{M}+\Fil^{{\bf j}+}\mathfrak{M}
\end{align}
for all ${\bf j}\neq {\bf 0}$ as we have ${\bf b}^{\bf j}\in (H_{\Delta,0}-1)\Fq\bs H_{\Delta,0}\js $. On the other hand, we have $\Fq\bs N_{\Delta,0}\js \otimes_{\Fq\bs N_0\js}\mathfrak{M}=\mathfrak{M}/(H_{\Delta,0}-1)\mathfrak{M}$ and $D^{\bf 0}(\mathfrak{M})=\mathfrak{M}/\Fil^{{\bf 0}+}\mathfrak{M}$, so the kernel of the map $\Fq\bs N_{\Delta,0}\js \otimes_{\Fq\bs N_0\js}\mathfrak{M}\to D^{\bf 0}(\mathfrak{M})$ can be identified with $$\Fil^{{\bf 0}+}\mathfrak{M}/(H_{\Delta,0}-1)\mathfrak{M}\ .$$ There are only finitely many indices ${\bf j}\in(\mathbb{Z}^{\geq 0})^d$ with $j_1+\dots+j_d\leq r$, so combining \eqref{bigjequal} and \eqref{smalljmultip} with Lemma \ref{intersect+} we deduce 
\begin{align}\label{almostkilled}
\left(\prod_{\substack{{\bf j}\in (\mathbb{Z}^{\geq 0})^d,\\ j_1+\dots+j_d\leq r}}X_\Delta^{r_{\bf j}}\right)\left(\Fil^{{\bf 0}+}\mathfrak{M}/(H_{\Delta,0}-1)\mathfrak{M} \right)\subseteq \left((H_{\Delta,0}-1)\mathfrak{M}+\Fil^{{\bf j}^{(0)}+}\mathfrak{M}\right)/(H_{\Delta,0}-1)\mathfrak{M}
\end{align}
for all indices ${\bf j}^{(0)}\in (\mathbb{Z}^{\geq 0})^d$. Finally, part $(b)$ follows noting that the intersection of the right hand side of \eqref{almostkilled} for varying ${\bf j}^{(0)}\in (\mathbb{Z}^{\geq 0})^d$ is $0$ by Lemma \ref{comparefilpr} and the compactness of $(H_{\Delta,0}-1)\mathfrak{M}$.
\end{proof}

Now our goal is to show that $\mathbb{M}^{bd}_\infty(D^{\#})$ satisfies the assumptions of Prop.\ \ref{fingen+kertors}.

\begin{lem}\label{tracepsipr}
Let $x\in \mathbb{M}_{\infty,0}(D)$ and $t\in T_+$. Assume $\pr_{k,\infty}(x)=0$ and $tH_{\Delta,0}t^{-1}\subseteq H_{\Delta,k}$. Then we have $$\sum_{u\in J(H_{\Delta,k}/tH_{\Delta,0}t^{-1})}\pr_{0,\infty}(\psi_t(u^{-1}x))=0\ .$$
\end{lem}
\begin{proof}
We have $\mathbb{M}_{k,0}(D)=H_0(H_{\Delta,k},\mathbb{M}_{\infty,0}(D))$ by the discussion before \cite[Proposition 4.7]{MultVar}. So $\pr_{k,\infty}(x)=0$ implies that $x$ can be written as a finite sum $x=\sum_j(v_j-1)x_j$ with $v_j\in H_{\Delta,k}$. Hence we have $\sum_{u\in J(H_{\Delta,k}/tH_{\Delta,0}t^{-1})}u^{-1}(v_j-1)\in (tH_{\Delta,0}t^{-1}-1)\mathbb{F}_q\bs H_{\Delta,0}\js$ for all $j\in J$ since both $u^{-1}$ and $u^{-1}v_j$ run on a set of representatives of the right cosets $tH_{\Delta,0}t^{-1}\backslash H_{\Delta,k}$. In particular, we have $\sum_{u\in J(H_{\Delta,k}/tH_{\Delta,0}t^{-1})}u^{-1}(v_j-1)=\sum_l (tw_{l,j}t^{-1}-1)\lambda_{l,j}$ for some $w_{l,j}\in H_{\Delta,0}$ and $\lambda_{l,j}\in \mathbb{F}_q\bs H_{\Delta,0}\js$ ($l\in L$ for some finite set $L$). So we compute
\begin{align*}
\sum_{u\in J(H_{\Delta,k}/tH_{\Delta,0}t^{-1})}\pr_{0,\infty}(\psi_t(u^{-1}x))=\sum_j\sum_{u\in J(H_{\Delta,k}/tH_{\Delta,0}t^{-1})}\pr_{0,\infty}(\psi_t(u^{-1}(v_j-1)x_j))=\\ 
=\sum_l\sum_j\sum_{u\in J(H_{\Delta,k}/tH_{\Delta,0}t^{-1})}\pr_{0,\infty}(\psi_t((tw_{l,j}t^{-1}-1)\lambda_{l,j}x_j))=\\
=\sum_l\sum_j\sum_{u\in J(H_{\Delta,k}/tH_{\Delta,0}t^{-1})}\pr_{0,\infty}((w_{l,j}-1)\psi_t(\lambda_{l,j}x_j))=0\ .
\end{align*}
\end{proof}
\begin{lem}\label{onlyonetermnotinkerpr}
Suppose that there is an element $x_0\in \mathbb{M}_{\infty,0}(D)$ and $\lambda_0\in \mathbb{F}_q\bs H_{\Delta,0}\js$ such that $\lambda_0 x_0+x_1\in \mathbb{M}^{bd}_\infty(D^\#)$ for some $x_1\in \mathbb{M}_{\infty,0}(D)$ with $\pr_{k,\infty}(x_1)=0\neq \pr_{k,\infty}(\lambda_0 x_0)$ for some integer $k\geq 0$. Then we have $\pr_{0,\infty}(x_0)\in D^\#$.
\end{lem}
\begin{proof}
Assume for contradiction that $\pr_{0,\infty}(x_0)\notin D^\#$. In particular, we have $\pr_{0,\infty}(x_0)\neq 0$. Since $D^\#=\bigcap_{\alpha\in\Delta}D^\#_\alpha$, there is an $\alpha\in \Delta$ such that $\pr_{0,\infty}(x_0)\notin D^\#_\alpha$. By Lemma \ref{nilpDeltaminusalpha} there is an integer $r\geq 1$ such that $s_{\Delta\setminus\{\alpha\}}^r H_{\Delta,0} s_{\Delta\setminus\{\alpha\}}^{-r}\subseteq H_{\Delta,k}$. Further, since $D^\#_\alpha$ is an étale $\varphi_{s_{\Delta\setminus\{\alpha\}}}$-module by Proposition \ref{Dhashalpha}, there is an element $\overline{v}\in N_{\Delta,0}$ such that $(\psi_{\Delta\setminus\{\alpha\}}^{D})^r(\overline{v}^{-1}\pr_{0,\infty}(x_0))\notin D^\#_\alpha$. Pick any lift $v\in N_0$ of $\overline{v}\in N_{\Delta,0}=N_0/H_{\Delta,0}$. Since both $H_{\Delta,0}$ and $H_{\Delta,k}$ are normal subgroups in $N_0$, $v^{-1}\lambda_0 v$ lies in $\mathbb{F}_q\bs H_{\Delta,0}\js$ and its image in $\mathbb{F}_q[H_{\Delta,0}/H_{\Delta,k}]$ is nonzero by our assumption $\pr_{k,\infty}(\lambda_0 x_0)\neq 0$. Now note that the socle of the left $\mathbb{F}_q\bs H_{\Delta,0}\js$-module $\mathbb{F}_q[H_{\Delta,0}/H_{\Delta,k}]$ is one dimensional, generated by the element $\sum_{w\in J(H_{\Delta,0}/H_{\Delta,k})}w^{-1}$. Since any nonzero submodule intersects the socle nontrivially, the nonzero submodule $\mathbb{F}_q\bs H_{\Delta,0}\js(v^{-1}\lambda_0v)\leq \mathbb{F}_q[H_{\Delta,0}/H_{\Delta,k}]$ must contain the socle, ie.\ there is an element $\lambda_1\in \mathbb{F}_q\bs H_{\Delta,0}\js$ such that 
\begin{align}\label{tracemultiple}
\lambda_1(v^{-1}\lambda_0v)=\lambda_2+\sum_{w\in J(H_{\Delta,0}/H_{\Delta,k})}w^{-1} 
\end{align}
for some $\lambda_2\in \Ker(\mathbb{F}_q\bs H_{\Delta,0}\js\to \mathbb{F}_q[H_{\Delta,0}/H_{\Delta,k}])=(H_{\Delta,k}-1)\mathbb{F}_q\bs H_{\Delta,0}\js$. Hence we put $t:=s_{\Delta\setminus\{\alpha\}}^r$ and compute
\begin{align*}
D^\#\supseteq\pr_{0,\infty}(\mathbb{M}^{bd}_\infty(D^\#) )\ni \sum_{u\in J(H_{\Delta,k}/t H_{\Delta,0} t^{-1})} \pr_{0,\infty}\left( \psi_t\left( u^{-1}\lambda_1v^{-1}(\lambda_0 x_0+x_1)\right)\right)\overset{\eqref{tracemultiple}}{=}\\
=\sum_{u\in J(H_{\Delta,k}/t H_{\Delta,0} t^{-1})} \pr_{0,\infty}\left( \psi_t\left( u^{-1}(\lambda_2v^{-1}x_0+\lambda_1v^{-1}x_1+ \sum_{w\in J(H_{\Delta,0}/H_{\Delta,k})}w^{-1}v^{-1}x_0)\right)\right)\overset{\text{Lemma }\ref{tracepsipr}}{=}\\
=\sum_{w\in J(H_{\Delta,0}/H_{\Delta,k})}\sum_{u\in J(H_{\Delta,k}/t H_{\Delta,0} t^{-1})} \pr_{0,\infty}\left( \psi_t\left( u^{-1}w^{-1}v^{-1}x_0\right)\right)\overset{w_1=wu}{=}\\
=\sum_{w_1\in J(H_{\Delta,0}/t H_{\Delta,0} t^{-1})} \pr_{0,\infty}\left( \psi_t\left( w_1^{-1}v^{-1}x_0\right)\right)\overset{\eqref{psireduction}}{=}\psi_t^{D}\left(\pr_{0,\infty}\left(v^{-1}x_0\right)\right)=\psi_t^{D}\left(\overline{v}^{-1}\pr_{0,\infty}\left(x_0\right)\right)\ . 
\end{align*}
This contradicts to our assumption that $(\psi_{\Delta\setminus\{\alpha\}}^{D})^r(\overline{v}^{-1}\pr_{0,\infty}(x_0))\notin D^\#_\alpha$, so the statement follows.
\end{proof}

\begin{pro}\label{kerintersect}
We have $D^{\bf j}\subseteq D^\#$ for all ${\bf j}\in (\mathbb{Z}^{\geq 0})^d$. In particular, $\mathbb{M}^{bd}_\infty(D^\#)$ satisfies the assumptions of Proposition \ref{fingen+kertors}.
\end{pro}
\begin{proof}
Let ${\bf j}=(j_1,\dots,j_d)\in (\mathbb{Z}^{\geq 0})^d$ be fixed. Let $k$ be large enough such that $p^{k_i}>j_i$ for all $i=1,\dots,d$ (see Lemma \ref{comparefilpr} for the definition of $k_i$, $i=1,\dots,d$). There exists such a $k$ since the subgroups $H_{\Delta,k}$ form a system of open neighbourhoods of $1$ in $H_{\Delta,0}$.

Pick an element $x\in D^{\bf j}$, so there is an element $x_2\in \Fil^{{\bf j}+}\mathbb{M}_{\infty,0}(D)$ and $x_0\in \mathbb{M}_{\infty,0}(D)$ such that $\pr_{0,\infty}(x_0)=x$ and ${\bf b}^{\bf j}x_0+x_2\in \mathbb{M}^{bd}_\infty(D^\#)$. By Lemma \ref{comparefilpr} and Proposition \ref{multigraded} we have $\pr_{k,\infty}({\bf b}^{{\bf \tilde{k}}-{\bf j}}x_2)=0$ and $\pr_{k,\infty}({\bf b}^{{\bf \tilde{k}}-{\bf j}}{\bf b}^{\bf j}x_0)\neq 0$. We deduce $x\in D^\#$ from Lemma \ref{onlyonetermnotinkerpr} where $\lambda_0:={\bf b}^{{\bf \tilde{k}}-{\bf j}}{\bf b}^{\bf j}$ and $x_1:={\bf b}^{{\bf \tilde{k}}-{\bf j}}x_2$.
\end{proof}

\begin{thm}\label{Mbdfingen}
$\mathbb{M}_\infty^{bd}(D^\#)$ is finitely generated over $\mathbb{F}_q\bs N_0\js$. 
\end{thm}
\begin{proof}
By Propositions \ref{Dhashprop} and \ref{kerintersect} we may apply Proposition \ref{fingen+kertors}$(a)$ to deduce the result.
\end{proof}

\begin{cor}\label{compact+psi-fingen}
Let $\mathfrak{M}\subset \mathbb{M}_{\infty,0}(D)$ be a compact $\mathbb{F}_q\bs N_0\js$-submodule stable under $\psi_t$ for all $t\in T_+$ such that $\psi_t\colon \mathfrak{M}\to \mathfrak{M}$ is onto for all $t\in T_+$. Then $\mathfrak{M}$ is finitely generated over $\mathbb{F}_q\bs N_0\js$.
\end{cor}
\begin{proof}
Since $\mathbb{F}_q\bs N_0\js$ is noetherian, $\mathfrak{M}$ is finitely generated by Theorem \ref{Mbdfingen} and Proposition \ref{prcontainedinDhash}.
\end{proof}

\begin{cor}\label{Minftyadmissible}
Let $\pi$ be a smooth representation of $\GL_n(\mathbb{Q}_p)$ over $\mathbb{F}_q$ and let $M_0\in \mathcal{M}^0_\Delta(\pi^{H_{\Delta,0}})$. Then $B_+M_0$ is admissible as a representation of $N_0$. 
\end{cor}
\begin{proof}
By \cite[Lemma 4.16]{MultVar}, the image of $\pi^{\vee}$ in $\mathbb{M}_{\infty,0}(D)$ is $(B_+M_0)^{\vee}$. The map $\pi^\vee\twoheadrightarrow (B_+M_0)^\vee\hookrightarrow \mathbb{M}_{\infty,0}(D)$ is $\psi_t$ equivariant for all $t\in T_+$ where $\psi_t\colon (B_+M_0)^\vee\to (B_+M_0)^\vee$ is the Pontryagin dual of the multiplication by $t$ map $B_+M_0\overset{t\cdot}{\rightarrow}B_+M_0$. Since the multiplication by $t$ is injective on $B_+M_0$ (as it is invertible on $\pi$), $\psi_t\colon (B_+M_0)^\vee\to (B_+M_0)^\vee$ is onto for all $t\in T_+$. By Theorem \ref{Mbdfingen} and Proposition \ref{prcontainedinDhash} we deduce that $(B_+M_0)^\vee$ is finitely generated over $\mathbb{F}_q\bs N_0\js$ whence its Pontryagin dual $B_+M_0$ is $N_0$-admissible.
\end{proof}

\begin{cor}\label{kerneltorsion}
Let $\pi$ be a smooth representation of $\GL_n(\mathbb{Q}_p)$ over $\mathbb{F}_q$, $M_0\in \mathcal{M}^0_\Delta(\pi^{H_{\Delta,0}})$, and put $M_\infty:=B_+M_0$. Then the kernel of the map $\Fq\bs N_{\Delta,0}\js\otimes_{\Fq\bs N_0\js}M_\infty^\vee\to M_0^\vee$ induced by the inclusion $M_0\leq M_\infty^{H_{\Delta,0}}$ is killed by some power of $X_\Delta$.
\end{cor}
\begin{proof}
This follows from Proposition \ref{fingen+kertors}$(b)$ noting the identification $M_0^\vee=D^{\bf 0}(M_\infty^\vee)$.
\end{proof}

\subsection{$G$-equivariant sheaves on $G/B$}

Following \cite[end of section 4]{MultVar} we choose a total ordering $\geq$ of the Weyl group $N_G(T)/T$ refining the Bruhat order $\succeq$. This gives a decreasing filtration of $G/B$ by open $B$-invariant subsets $\Fil^{w}(G/B):=\bigcup_{w_1\geq w}\mathcal{C}^{w_1}\subseteq G/B$ and $\Fil^{>w}(G/B):=\bigcup_{w_1> w}\mathcal{C}^{w_1}\subseteq G/B$ for $w\in N_G(T)/T$ where $\mathcal{C}^{w_1}:=Bw_1B/B=Nw_1B/B\subset G/B$. Put $\mathcal{C}_0^{w_1}:=N_0w_1B/B\subset\mathcal{C}^{w_1}\subset G/B$. Since $N=\bigcup_{k\geq 0}s^{-k}N_0s^k$, we obtain $\mathcal{C}^w=\bigcup_{k\geq 0}s^{-k}\mathcal{C}_0^w$ for all $w\in N_G(T)/T$. The bottom term of the filtration is $\Fil^{w_0}(G/B)=\mathcal{C}^{w_0}$ for the element $w_0\in N_G(T)/T$ of maximal length and we have $\Fil^1(G/B)=G/B$. Following \cite{SchVZ} put $\mathcal{C}_0:=\mathcal{C}_0^{w_0}:=N_0w_0B/B\subset G/B$, and $\mathcal{C}:=\mathcal{C}^{w_0}:=Nw_0B/B\subset G/B$. Both $\mathcal{C}$ and $\mathcal{C}_0$ are open in $G/B$. Moreover, $\mathcal{C}\subset G/B$ is dense and $\mathcal{C}_0$ is compact.

Recall the following multiplication rule of Bruhat cells: $$Bw_\alpha B\cdot Bw_1B=\begin{cases}Bw_\alpha w_1B&\text{if }\ell(w_\alpha w_1)=\ell(w_1)+1\\
Bw_1B\cup Bw_\alpha w_1 B&\text{if }\ell(w_\alpha w_1)=\ell(w_1)-1\end{cases}$$
where $w_\alpha\in N_G(T)/T$ denotes the simple reflection corresponding to the simple root $\alpha\in\Delta$, $w_1\in N_G(T)/T$ and $\ell$ is the length. Write $w$ as a minimal product of $\ell(w)$ simple reflections $w=w_{\alpha_1}\cdots w_{\alpha_{\ell(w)}}$ for simple roots $\alpha_1,\dots,\alpha_{\ell(w)}\in \Delta$. Then $w^{-1}=w_{\alpha_{\ell(w)}}\cdots w_{\alpha_1}$ is also a minimal product and $\ell(ww_0)=\ell(w_0)-\ell(w)$. In particular, we deduce 
\begin{align*}
BwBw_0B=Bw_{\alpha_1}\cdots w_{\alpha_{\ell(w)}}w_0B\cup Bw_{\alpha_2}\cdots w_{\alpha_{\ell(w)}}w_0B\cup\cdots\cup Bw_{\alpha_{\ell(w)}}w_0B\cup Bw_0B\quad\text{whence}\\
w\mathcal{C}=wNw_0B/B\subseteq \Fil^{ww_0}(G/B)\ .
\end{align*}
\begin{lem}\label{Cwsubsetww_0C}
We have $w_0^{-1}w^{-1}\mathcal{C}_0^w=\mathcal{C}_0\cap w_0^{-1}w^{-1}\mathcal{C}^w=N_{0,w_0^{-1}w^{-1}}w_0B/B$.
\end{lem}
\begin{proof}
We have $N=(N\cap w\overline{N}w^{-1})\cdot (N\cap wNw^{-1})$, so we compute 
\begin{align*}
\mathcal{C}^{w}=NwB/B=(N\cap w \overline{N}w^{-1})\cdot (N\cap wNw^{-1})wB/B=\\
=ww_0(w_0^{-1}w^{-1}Nww_0\cap N)\cdot (w_0^{-1}w^{-1}Nww_0\cap \overline{N})w_0B/B=\\
=ww_0(w_0^{-1}w^{-1}Nww_0\cap N)w_0(w^{-1}Nw\cap N)B/B=ww_0(w_0^{-1}w^{-1}Nww_0\cap N)w_0B/B\subset ww_0\mathcal{C}\ .
\end{align*}
Similarly, since $N_0$ is totally decomposed, we also have $N_0=(N_0\cap w\overline{N}w^{-1})\cdot (N_0\cap wNw^{-1})$ and
\begin{align*}
\mathcal{C}^{w}_0=N_0wB/B=ww_0(w_0^{-1}w^{-1}Nww_0\cap N_0)w_0B/B=ww_0N_{0,w_0^{-1}w^{-1}}w_0B/B\subset ww_0\mathcal{C}_0\ .
\end{align*}
The statement follows noting that the map $N\to Nw_0B/B$ is a bijection whence $(w_0^{-1}w^{-1}Nww_0\cap N)w_0B/B\cap N_0w_0B/B=(w_0^{-1}w^{-1}Nww_0\cap N_0)w_0B/B$.
\end{proof}

\begin{lem}\label{kersheaf}
Assume that $\mathfrak{Y}$ is a $G$-equivariant sheaf of abelian groups on $G/B$ in the $p$-adic topology of $G/B$ and let $w\in N_G(T)/T$. Then $\mathfrak{Y}^w$ is a $B$-equivariant sheaf on $\mathcal{C}^{w}\subset G/B$ where we put $$\mathfrak{Y}^w(U\cap \mathcal{C}^{w}):=\Ker\left(\res_{\Fil^{>w}(G/B)}^{\Fil^w(G/B)}\colon \mathfrak{Y}(U\cap \Fil^w(G/B))\to \mathfrak{Y}(U\cap \Fil^{>w}(G/B))\right)$$
for any open subset $U\subseteq G/B$.
\end{lem}
\begin{proof}
Since $\Fil^w(G/B)$ (resp.\ $\Fil^{>w}(G/B)$) is open in $G/B$, $\mathfrak{Y}_{\mid \Fil^w}\colon U\mapsto \mathfrak{Y}(U\cap \Fil^w(G/B))$ (resp.\ $\mathfrak{Y}_{\mid \Fil^{>w}}\colon U\mapsto \mathfrak{Y}(U\cap \Fil^{>w}(G/B))$) is $\iota(w)_*\iota(w)^{-1}\mathfrak{Y}$ (resp.\ $\iota(>w)_*\iota(>w)^{-1}\mathfrak{Y}$) where $\iota(w)\colon \Fil^w(G/B)\hookrightarrow G/B$ (resp.\ $\iota(>w)\colon \Fil^{>w}(G/B)\hookrightarrow G/B$) is the inclusion map. So $\res_{\Fil^{>w}(G/B)}^{\Fil^w(G/B)}\colon \mathfrak{Y}_{\mid \Fil^w}\to \mathfrak{Y}_{\mid \Fil^{>w}}$ is a $B$-equivariant map of $B$-equivariant sheaves on $G/B$ whose kernel is a sheaf supported on $\mathcal{C}^w$.
\end{proof}

Let $D_0$ be an étale $T_+$-module over $\Fq\bs N_0\js$, ie.\ $D_0$ is a (not necessarily finitely generated) left $\Fq\bs N_0\js$-module together with a semilinear action of the monoid $T_+$ given by maps $\varphi_t\colon D_0\to D_0$ for $t\in T_+$ such that $\id\otimes\varphi_t\colon \Fq\bs N_0\js\otimes_{\Fq\bs N_0\js,\varphi_t}D_0\to D_0$ is an isomorphism for all $t\in T_+$. The $2$ particular cases of interest for such an étale $T_+$-module $D_0$ over $\Fq\bs N_0\js$ are the following: On one hand, one can take $D_0$ to be $D_\infty=\mathbb{M}_{\infty,0}(D)$ for some étale $T_+$-module $D$ over $\Fq\bg N_{\Delta,0}\jg$. In this case $D_\infty$ is an étale $T_+$-module over $\Fq\bg N_{\Delta,\infty}\jg$ hence also over $\Fq\bs N_0\js$ via the injective ring homomorphism $\Fq\bs N_0\js\hookrightarrow \Fq\bg N_{\Delta,\infty}\jg$. On the other hand, one could take $D_0$ to be the étale hull $D_0=\widetilde{D_1}$ \cite[section 2.3]{links} of a compact $\Fq\bs N_0\js$-module $D_1$ together with a compatible action of the operators $\psi_t$ for all $t\in T_+$ such that $\psi_t\colon D_1\to D_1$ is onto for all $t\in T_+$. 

Then $D_0$ is an étale $\Fq[B_+]$-module in the sense of \cite[section 3]{SchVZ} by \cite[Prop.\ 3.29]{SchVZ}. So by \cite[Theorem 3.32]{SchVZ} we can associate to $D_0$ a $B$-equivariant sheaf $\mathfrak{Y}_{D_0}$ on $\mathcal{C}$ with sections $\mathfrak{Y}_{D_0}(\mathcal{C}_0)=D_0$ on the compact open $\mathcal{C}_0$. Further, for any $u\in N_0$ and $t\in T_+$ the restriction maps (see \cite[Proposition 3.16]{SchVZ}) are given by $\res^{\mathcal{C}_0}_{ut\mathcal{C}_0}:=u\circ\varphi_t\circ\psi_t\circ u^{-1}\colon D_0\to u\varphi_t(D_0)=:\mathfrak{Y}_D(ut\mathcal{C}_0)\subseteq D_0$. Here $\psi_t$ is the canonical left inverse of $\varphi_t$. Any compact open subset $U$ in $\mathcal{C}_0$ is the disjoint union $U=\coprod_{i\in I} u_it_i\mathcal{C}_0$ of finitely many open subsets of the form $u_it_i\mathcal{C}_0$ ($u_i\in N_0$, $t_i\in T_+$, $i\in I$), so we put $\mathfrak{Y}_{D_0}(\coprod_{i\in I} u_it_i\mathcal{C}_0):=\bigoplus_{i\in I}\mathfrak{Y}_{D_0}(u_it_i\mathcal{C}_0))$. In particular, whenever $U\subseteq \mathcal{C}_0$ then $\mathfrak{Y}_{D_0}(U)$ can be identified as a subspace in $D$. Finally, if $U\subseteq \mathcal{C}_0$ is any open subset then $U=\bigcup_{k\geq 1}U_k$ is the ascending union of compact open subsets $U_k\subseteq U$ and we have $\mathfrak{Y}_{D_0}(U)=\varprojlim_k \mathfrak{Y}_{D_0}(U_k)$. For any open subsets $U\subseteq V\subseteq \mathcal{C}_0$ we put  $\res^{V}_{U}\colon \mathfrak{Y}_{D_0}(V)\to \mathfrak{Y}_{D_0}(U)$. By \cite[Theorem 3.32]{SchVZ} one can extend this sheaf to a $B$-equivariant sheaf on $\mathcal{C}$. Since $\mathcal{C}=\bigcup_{k\geq 0}s^{-k}\mathcal{C}_0$, the global sections are $\mathfrak{Y}_{D_0}(\mathcal{C})=\varprojlim_{\psi_s}D_0=\{(x_k)_{k\geq 0}\subset D_0\mid \psi_s(x_{k+1})=x_k\text{ for all }k\geq 0\}$. 

\begin{lem}\label{varphibequivres}
Assume that $U\cap b\mathcal{C}_0= bU$ for some $b\in B_+$ and open subset $U\subseteq \mathcal{C}_0$. Then $\res^{\mathcal{C}_0}_U$ is $\varphi_b$-equivariant.
\end{lem}
\begin{proof}
By multiplying the identity $U\cap b\mathcal{C}_0=bU$ with $b^{-1}$ we obtain $\mathcal{C}_0\cap b^{-1}U=U$. So we compute $$\res^{\mathcal{C}_0}_U(\varphi_b(x))=\varphi_b(\res^{\mathcal{C}_0}_{\mathcal{C}_0\cap b^{-1}U}(x))=\varphi_b(\res^{\mathcal{C}_0}_U(x))\ .$$
\end{proof}

Put $U_w:=\mathcal{C}_0\setminus w_0^{-1}w^{-1}\mathcal{C}^w\subseteq \mathcal{C}_0$. Further, write $D_w\subseteq D_0$ for the kernel of $$\res^{\mathcal{C}_0}_{U_w} \colon D_0=\mathfrak{Y}_{D_0}(\mathcal{C}_0)\to \mathfrak{Y}_{D_0}(U_w) \ .$$

\begin{lem}\label{phipsiinvariantDw}
 We have $U_w\cap b\mathcal{C}_0= bU_w$ for all $b\in N_{0,w_0^{-1}w^{-1}}T_+$. In particular, $D_w$ is invariant under the action of the group $N_{0,w_0^{-1}w^{-1}}$ and of the operators $\varphi_t$ and $\psi_t$ for all $t\in T_+$. Further, we have $\psi_t(u^{-1}D_w)=0$ whenever $utN_0t^{-1}\cap N_{0,w_0^{-1}w^{-1}}=\emptyset$ for some $u\in N_0$ and $t\in T_+$.
\end{lem}
\begin{proof}
Let $b=ut\in N_{0,w_0^{-1}w^{-1}}T_+$ and $vw_0B\in U_w$ for some $v\in N_0$. By Lemma \ref{Cwsubsetww_0C} we have $v\notin N_{0,w_0^{-1}w^{-1}}$. By construction, $N_{0,w_0^{-1}w^{-1}}$ is the intersection of $N_0$ with the product of certain root subgroups, so we obtain $tvt^{-1}\notin N_{0,w_0^{-1}w^{-1}}$ and $utvt^{-1}\notin N_{0,w_0^{-1}w^{-1}}$, too. This shows $bvw_0B\in U_w$ hence $bU_w\subseteq U_w\cap b\mathcal{C}_0$. On the other hand, if $vw_0B\in N_{0,w_0^{-1}w^{-1}}w_0B/B=\mathcal{C}_0\setminus U_w$ then we have $v\in N_{0,w_0^{-1}w^{-1}}$ whence $utvt^{-1}\in N_{0,w_0^{-1}w^{-1}}$ showing $bvw_0B\notin U_w$. We deduce $U_w\cap b\mathcal{C}_0=bU_w$ as claimed. The $N_{0,w_0^{-1}w^{-1}}$- and $\varphi_t$-invariance of $\Ker(\res^{\mathcal{C}_0}_{U_w})$ follows directly from Lemma \ref{varphibequivres}. For the statement on the $\psi$-invariance let $t\in T_+$. Note that whenever $utN_0t^{-1}\cap N_{0,w_0^{-1}w^{-1}}$ is nonempty, we may indeed choose a representative $u'\in utN_0t^{-1}\cap N_{0,w_0^{-1}w^{-1}}$ of the left coset $utN_0t^{-1}\in N_0/tN_0t^{-1}$, so there is a set $J(N_0/tN_0t^{-1})$ of representatives consisting of such elements whenever $utN_0t^{-1}\cap N_{0,w_0^{-1}w^{-1}}\neq\emptyset$ and of arbitrary representatives when $utN_0t^{-1}\cap N_{0,w_0^{-1}w^{-1}}=\emptyset$. Further, we may, as well, assume $1\in J(N_0/tN_0t^{-1})$. We put $J'(N_0/tN_0t^{-1}):= J(N_0/tN_0t^{-1})\cap N_{0,w_0^{-1}w^{-1}}$. For any $x\in D_0$ with $\res^{\mathcal{C}_0}_{U_w}(x)=0$ $u\in J(N_0/tN_0t^{-1})\setminus J'(N_0/tN_0t^{-1})$ we have $ut\mathcal{C}_0\subseteq U_w$, so we obtain 
\begin{align*}
u\varphi_t(\psi_t(u^{-1}x))=\res^{\mathcal{C}_0}_{utN_0t^{-1}}(x)=\res^{U_w}_{utN_0t^{-1}}\circ\res^{\mathcal{C}_0}_{U_w}(x)=0
\end{align*}
showing $\psi_t(u^{-1}x)=0$ as $\varphi_t\colon D_0\to D_0$ is injective. Further, we compute
\begin{align*}
0=\res^{\mathcal{C}_0}_{U_w}(x)=\res^{\mathcal{C}_0}_{U_w}(\sum_{u\in J(N_0/tN_0t^{-1})}u\varphi_t(\psi_t(u^{-1}x)))=\res^{\mathcal{C}_0}_{U_w}(\sum_{u\in J'(N_0/tN_0t^{-1})}u\varphi_t(\psi_t(u^{-1}x)))=\\
=\sum_{u\in J'(N_0/tN_0t^{-1})}u\varphi_t(\res^{\mathcal{C}_0}_{U_w}(\psi_t(u^{-1}x)))\ .
\end{align*}
Therefore we get $0=\psi_t(u^{-1}0)=\res^{\mathcal{C}_0}_{U_w}(\psi_t(u^{-1}x))$ which means $\psi_t(u^{-1}x)$ lies in the kernel of $\res^{\mathcal{C}_0}_{U_w}$ as desired.
\end{proof}

Now assume we are in the case $D_0=D_\infty=\mathbb{M}_{\infty,0}(D)$ for some finitely generated étale $T_+$-module $D$ over $\Fq\bg N_{\Delta,0}\jg$ so that $D_\infty$ is an étale $T_+$-module over $\Fq\bg N_{\Delta,\infty}\jg$.

\begin{pro}\label{Dwprojlim}
Assume $D_\infty=\mathbb{M}_{\infty,0}(D)$ is an étale $T_+$-module over $\Fq\bg N_{\Delta,\infty}\jg$. Then we have
\begin{align*}
D_w:=\Ker(\res^{\mathcal{C}_0}_{U_w} \colon D_\infty=\mathfrak{Y}_{D_\infty}(\mathcal{C}_0)\to \mathfrak{Y}_{D_\infty}(U_w))=\varprojlim_k \Fq\bg N_{\Delta,k,w_0^{-1}w^{-1}}\jg\otimes_{\varphi_{s^{kn}},\Fq\bg N_{\Delta\setminus S,0}\jg}D_S^{\nr}\ .
\end{align*} 
In particular, $D_w$ is a free module of rank $\dim_{\Fq}\mathbb{V}_\Delta(D)^{G'_{\Qp,S}}$ over $\Fq\bg N_{\Delta,\infty,w_0^{-1}w^{-1}}\jg$ for all $w\in N_G(T)/T$ where $S:=\Delta\setminus\Delta_{w_0^{-1}w^{-1}}$.
\end{pro}
\begin{proof}
Let $x$ be an element in $D_w$. 
\begin{lem}\label{prinnr}
We have $\pr_{0,\infty}(x)\in D_S^{\nr}$.
\end{lem}
\begin{proof}
Put $t_S:=\prod_{\alpha\in S}t_\alpha$. Then $t_S$ commutes with elements in $N_{\Delta,0,w_0^{-1}w^{-1}}=\prod_{\alpha\in \Delta\cap w_0^{-1}w^{-1}(\Phi^+)}(N_\alpha\cap N_0)=\prod_{\alpha\in\Delta\setminus S}(N_\alpha\cap N_0)$, so we have $N_{\Delta,0,w_0^{-1}w^{-1}}=t^k_SN_{\Delta,0,w_0^{-1}w^{-1}}t_S^{-k}\leq t_S^kN_{\Delta,0}t_S^{-k}$ for all $k\geq 1$. Assume that $u\in N_0$ is such an element that the image $\overline{u}$ of $u$ in $N_{\Delta,0}=N_0/H_{\Delta,0}$ does not lie in the subgroup $t_S^k N_{\Delta,0}t_S^{-k}$. Hence the image of $ut_S^kN_0t_S^{-k}$ in $N_{\Delta,0}$ is disjoint from the subgroup $N_{\Delta,0,w_0^{-1}w^{-1}}$. Therefore $ut^k_SN_0t_S^{-k}$ is disjoint from $N_{0,w_0^{-1}w^{-1}}$ showing $ut_S^k\mathcal{C}_0\cap w_0^{-1}w^{-1}\mathcal{C}^w=\emptyset$ by Lemma \ref{Cwsubsetww_0C}. So we obtain $\psi_S^k(u^{-1}x)=0$ by Lemma \ref{phipsiinvariantDw}. Using \eqref{psireduction} we compute
\begin{align*}
x_{\overline{u}}:=\psi_{t^k_S}^{(D)}(\overline{u}^{-1}\pr_{0,\infty}(x))=\psi_{t^k_S}^{(D)}(\pr_{0,\infty}(u^{-1}x))=\pr_{0,\infty}\left(\sum_{v\in J(H_{\Delta,0}/t_S^kH_{\Delta,0}t_S^{-k})}\psi_S^k(v^{-1}u^{-1}x)\right)=0
\end{align*}
since $\overline{uv}=\overline{u}$ does not lie in the subgroup $t_S^kN_{\Delta,0}t_S^{-k}$. Finally, by the étale property of the action of $\varphi_S^k$ on $D$ we may write
\begin{align*}
\pr_{0,\infty}(x)=\sum_{\overline{u}\in N_{\Delta,0}/t_S^kN_{\Delta,0}t_S^{-k}}\overline{u}\varphi_S^k(x_{\overline{u}})=\varphi_S^k(x_1)\in \varphi_S^k(D)\ .
\end{align*}
Since $k\geq 1$ was arbitrary, we deduce $\pr_{0,\infty}(x)\in \bigcap_{k\geq 1}\varphi_S^k(D)=D_S^{\nr}$.
\end{proof}
By the discussion of \cite[section 4.1]{MultVar} and noting $n>\deg(\beta\circ\xi)$ for all $\beta\in \Phi^+$, $s^{kn}N_{\Delta,k}s^{-kn}$ lies in the center of the group $N_{\Delta,k}$ and we have
\begin{align*}
D_\infty=\varprojlim_k\mathbb{M}_{k,0}(D)=\varprojlim_k\Fq\bg N_{\Delta,k}\jg\otimes_{\varphi_{s^{kn_0}},\Fq\bg N_{\Delta,0}\jg}D
\end{align*}
by construction \cite[section 4.2]{MultVar}. Combining Lemma \ref{phipsiinvariantDw} with Lemma \ref{prinnr} we can write $\pr_{k,\infty}(x)$ as
\begin{align*}
\pr_{k,\infty}(x)=\sum_{u\in J(N_{\Delta,k,w_0^{-1}w^{-1}}/s^{kn}N_{\Delta,k,w_0^{-1}w^{-1}}s^{-k})}u\varphi_{s^{kn}}(x_{u,k})
\end{align*}
where $x_{u,k}=\pr_{0,\infty}(\psi_{s^{kn}}(u^{-1}x))\in D^{\nr}_S$ is uniquely determined. Conversely, assume that $y\in D_\infty$ is such an element that 
\begin{align}\label{prkyeq}
\pr_{k,\infty}(y)=\sum_{u\in J(N_{\Delta,k,w_0^{-1}w^{-1}}/s^{kn}N_{\Delta,k,w_0^{-1}w^{-1}}s^{-k})}u\varphi_{s^{kn}}(y_{u,k})
\end{align}
for some elements $y_{u,k}=\pr_{0,\infty}(\psi_{s^{kn}}(u^{-1}y))\in D$ and every $k\geq 1$. We claim that $y$ lies in $D_w\leq D_\infty$. Indeed, let $v\in N_0$ be such that $vs^r\mathcal{C}_0\subseteq U_w$ for some $r\geq 1$. Then we have $vs^rv_1s^{kn-r}\mathcal{C}_0\subseteq vs^r\mathcal{C}_0\subseteq U_w$ for all $v_1\in N_0$ and $k\geq r$. In particular, the image of $vs^rv_1s^{-r}$ under the quotient map $N_0\twoheadrightarrow N_{\Delta,k}$ is not contained in $N_{\Delta,k,w_0^{-1}w^{-1}}s^{kn}N_{\Delta,k}s^{-kn}$. From \eqref{prkyeq} we deduce $\pr_{0,\infty}(\psi_{s^{kn}}(s^rv_1^{-1}s^{-r}v^{-1}y))=0$ for all $v_1\in N_0$. Hence we compute
\begin{align*}
\psi_{s^r}(v^{-1}y)=\sum_{v_1\in J(N_0/s^{kn-r}N_0s^{-kn+r})}v_1\varphi_{s^{kn-r}}\circ\psi_{s^{kn}}(s^rv_1^{-1}s^{-r}v^{-1}y)\in \\
\in N_0\varphi_{s^{kn-r}}(\Ker(\pr_{0,\infty}))\subseteq \Ker(\pr_{kn-r,\infty})
\end{align*}
showing $\pr_{kn-r,\infty}(\psi_{s^r}(v^{-1}y))=0$. Letting $k\to\infty$ we deduce $\psi_{s^r}(v^{-1}y)=0$ whence $\res^{\mathcal{C}_0}_{vs^r\mathcal{C}_0}(y)=v\varphi_{s^r}\circ\psi_{s^r}(v^{-1}y)=0$. Since $U_w$ can be written as the union of open subsets of the form $vs^r\mathcal{C}_0$, we obtain $\res^{\mathcal{C}_0}_{U_w}(y)=0$, too, showing $y\in D_w$. Therefore we have
\begin{align*}
\pr_{k,\infty}(D_w)=\bigoplus_{u\in J(N_{\Delta,k,w_0^{-1}w^{-1}}/s^{kn}N_{\Delta,k,w_0^{-1}w^{-1}}s^{-k})}u\varphi_{s^{kn}}(D^{\nr}_S)=\Fq\bg N_{\Delta,k,w_0^{-1}w^{-1}}\jg\otimes_{\varphi_{s^{kn}},\Fq\bg N_{\Delta\setminus S,0}\jg}D_S^{\nr}\ .
\end{align*}
The first statement follows by taking the limit with respect to $k\to\infty$. Further, $\pr_{k,\infty}(D_w)$ is a free module of rank $\dim_{\Fq}\mathbb{V}_\Delta(D)^{G'_{\Qp,S}}$ over $\Fq\bg N_{\Delta,k,w_0^{-1}w^{-1}}\jg$ by Prop.\ \ref{Dnrmultvar}. Finally, for any integers $k_1\leq k_2$ and finite set $x_1,\dots,x_\ell$ of free generators of $\pr_{k_1,\infty}(D_w)$ can be lifted to a set of free generators of $\pr_{k_2,\infty}(D_w)$ by \cite[Lemma 4.5]{MultVar}, so the last statement also follows by taking the limit.
\end{proof}

\begin{pro}\label{MbdcapDwfin}
Assume $D_\infty=\mathbb{M}_{\infty,0}(D)$ is an étale $T_+$-module over $\Fq\bg N_{\Delta,\infty}\jg$. Then $\mathbb{M}^{bd}_\infty(D^\#)\cap D_w$ is finitely generated as a module over $\Fq\bs N_{0,w_0^{-1}w^{-1}}\js$.
\end{pro}
\begin{proof}
This is completely analogous to Proposition \ref{fingen+kertors}, so we adapt its proof to this situation using Prop.\ \ref{Dwprojlim}. Recall we have a filtration \eqref{multfiltmod} indexed by the lexicographically ordered set $(\mathbb{Z}^{\geq 0})^d$ on $D_\infty$. This induces a filtration on both $D_w$ and $\mathbb{M}^{bd}_\infty(D^\#)\cap D_w$ (and also on $\mathbb{M}^{bd}_\infty(D^\#)$, see \eqref{filtonMbd}), so we put
\begin{align*}
D^{\bf j}_{w}:=(D_w\cap \Fil^{\bf j}\mathbb{M}_{\infty,0}(D))/(D_w\cap \Fil^{{\bf j}+}\mathbb{M}_{\infty,0}(D)) \cong\\
\cong (D_w\cap \Fil^{\bf j}\mathbb{M}_{\infty,0}(D)+ \Fil^{{\bf j}+}\mathbb{M}_{\infty,0}(D))/\Fil^{{\bf j}+}\mathbb{M}_{\infty,0}(D) \subseteq  D \qquad\text{and} \\
D^{\bf j}_{bd,w}:=(D_w\cap\mathbb{M}^{bd}_\infty(D^\#)\cap \Fil^{\bf j}\mathbb{M}_{\infty,0}(D))/(D_w\cap \mathbb{M}^{bd}_\infty(D^\#)\cap \Fil^{{\bf j}+}\mathbb{M}_{\infty,0}(D))\cong\\
\cong (D_w\cap \mathbb{M}^{bd}_\infty(D^\#)\cap \Fil^{\bf j}\mathbb{M}_{\infty,0}(D)+ \Fil^{{\bf j}+}\mathbb{M}_{\infty,0}(D))/\Fil^{{\bf j}+}\mathbb{M}_{\infty,0}(D) \subseteq D^{\bf j}\cap D_w^{\bf j}\subset D 
\end{align*}
for the graded pieces. Note that the monomial ${\bf b}^{\bf j}\in\Fq\bs H_{\Delta,0}\js$ lies in the subring $\Fq\bs N_{0,w_0^{-1}w^{-1}}\cap H_{\Delta,0}\js$ if and only if $j_r=0$ for all $r\in\{1,\dots,d\}$ with $\beta_r\in(\Phi^+\setminus\Delta)\cap w_0^{-1}w^{-1}(\Phi^-)=\Phi^+_{w_0^{-1}w^{-1}}\setminus\Delta$. We put $(\mathbb{Z}^{\geq 0})^{\Phi^+_{w_0^{-1}w^{-1}}\setminus\Delta}\subseteq (\mathbb{Z}^{\geq 0})^{\Phi^+\setminus\Delta}=(\mathbb{Z}^{\geq 0})^d$ for this set of indices. This observation leads to the following
\begin{lem}\label{gradedDw}
We have
\begin{align*}
D^{{\bf j}^{(0)}}_w=\begin{cases} D^{\nr}_S &\text{if }{\bf j}^{(0)}\in (\mathbb{Z}^{\geq 0})^{\Phi^+_{w_0^{-1}w^{-1}}\setminus\Delta} \\ 0&\text{otherwise, } \end{cases}
\end{align*}
where $S:=\Delta\cap w_0^{-1}w^{-1}(\Phi^-)$.
\end{lem}
\begin{proof}
Pick an index ${\bf j}^{(0)}\in(\mathbb{Z}^{\geq 0})^d$ and let $k$ be a positive integer. We write $k_r$ for the smallest integer such that $u_{\beta_r}^{p^k_r}\in H_{\Delta,k}$ ($r=1,\dots,d$). For large enough $k$, we have $p^{k_r}\geq j_r+1$ for all $r=1,\dots,d$, so we have $\Ker(\pr_{k,\infty})\subseteq \Fil^{{\bf j}+}D_\infty$ as $\Ker(\pr_{k,\infty})$ is generated by $(u_{\beta_r}^{p^{k_r}}-1)D_\infty$ ($r=1,\dots,d$). Since $D_k=\mathbb{M}_{k,0}(D)=\Fq\bg N_{\Delta,k}\jg\otimes_{\Fq\bg N_{\Delta,0}\jg,\varphi_{s^{kn}}}D$, we may decompose $D_k$ as a direct sum
\begin{align}
\pr_{k,\infty}(D_\infty)=D_k=\bigoplus_{{\bf j}\in \bigtimes_{r=1}^d\{0,\dots,p^{k_r}-1\}}\bigoplus_{i_1,\dots,i_{n-1}\in\{0,\dots,p^{kn}-1\}}{\bf b}^{\bf j}\prod_{j=1}^{n-1}u_{\alpha_j}^{i_j}\varphi_{s^{kn}}(D)\label{prkDinftydecompose}
\end{align}
where $u_{\alpha_j}\in N_0\cap N_{\alpha_j}$ is a topological generator for all $j=1,\dots,n-1$ with $\Delta=\{\alpha_1,\dots,\alpha_{n-1}\}$. Similarly, by noting $N_{0,w_0^{-1}w^{-1}}=\prod_{\gamma\in \Phi^+\cap w_0^{-1}w^{-1}(\Phi^+) }(N_0\cap N_\gamma)$ and Prop.\ \ref{Dwprojlim} we may decompose 
\begin{align}
\pr_{k,\infty}(D_w)=\bigoplus_{\substack{{\bf j}\in \bigtimes_{r=1}^d\{0,\dots,p^{k_r}-1\}\\ j_r=0\text{ if }\beta_r\notin w_0^{-1}w^{-1}(\Phi^+)}}\bigoplus_{\substack{i_1,\dots,i_{n-1}\in\{0,\dots,p^{kn}-1\}\\ i_j=0 \text{ if }\alpha_j\notin w_0^{-1}w^{-1}(\Phi^+)}}{\bf b}^{\bf j}\prod_{j=1}^{n-1}u_{\alpha_j}^{i_j}\varphi_{s^{kn}}(D^{\nr}_S)\ .\label{prkDwdecompose}
\end{align}
Assume that $x\in D_w\cap \Fil^{{\bf j}^{(0)}}D_\infty$. Then by \eqref{prkDinftydecompose} we may uniquely write
\begin{align*}
\pr_{k,\infty}(x)=\sum_{{\bf j}\in \bigtimes_{r=1}^d\{0,\dots,p^{k_r}-1\}}\sum_{i_1,\dots,i_{n-1}\in\{0,\dots,p^{kn}-1\}}{\bf b}^{\bf j}\prod_{j=1}^{n-1}u_{\alpha_j}^{i_j}\varphi_{s^{kn}}(x_{{\bf j},{\bf i}})
\end{align*}
and we have $x_{{\bf j},{\bf i}}=0$ unless ${\bf j}\geq {\bf j}^{(0)}$ as $x$ lies in $\Fil^{{\bf j}^{(0)}}D_\infty$. Further, $x\in D_w$ implies that $x_{{\bf j},{\bf i}}=0$ unless ${\bf j}\in (\mathbb{Z}^{\geq 0})^{\Phi^+_{w_0^{-1}w^{-1}}\setminus\Delta}$ and ${\bf i}\in (\mathbb{Z}^{\geq 0})^{\Delta_{w_0^{-1}w^{-1}}}$. Further, $x\in D_w$ also shows $x_{{\bf j},{\bf i}}\in D^{\nr}_S$ for all indices ${\bf j}$ and ${\bf i}$. Finally, the image of $x$ in $D_w^{{\bf j}^{(0)}}=D_w\cap \Fil^{{\bf j}^{(0)}}D_\infty+\Fil^{{\bf j}^{(0)}+}D_\infty/\Fil^{{\bf j}^{(0)}+}D_\infty$ is $0$ unless ${\bf j}\in (\mathbb{Z}^{\geq 0})^{\Phi^+_{w_0^{-1}w^{-1}}\setminus\Delta}$, and equals  $$\sum_{i_1,\dots,i_{n-1}\in\{0,\dots,p^{kn}-1\}}\prod_{j=1}^{n-1}u_{\alpha_j}^{i_j}\varphi_{s^{kn}}(x_{{\bf j}^{(0)},{\bf i}})\in D_S^{\nr}$$
otherwise. We deduce the containment $\subseteq$ in the statement.

Conversely, note that $D_w^{\bf 0}=\pr_{0,\infty}(D_w)=D^{\nr}_S$ by Prop.\ \ref{Dwprojlim} and the multiplication by ${\bf b}^{\bf j}$ induces an injective map $D_w^{\bf 0}\hookrightarrow D_w^{\bf j}$ whenever ${\bf b}^{\bf j}$ lies in $\Fq\bs N_{0,w_0^{-1}w^{-1}}\cap H_{\Delta,0}\js$ as $D_w$ is an $\Fq\bs N_{0,w_0^{-1}w^{-1}}\cap H_{\Delta,0}\js$-module.
\end{proof}

Combining Prop.\ \ref{kerintersect} and Lemma \ref{gradedDw} we deduce $D_{bd,w}^{\bf j}=0$ if ${\bf j}\notin (\mathbb{Z}^{\geq 0})^{\Phi^+_{w_0^{-1}w^{-1}}\setminus\Delta}$ and $D_{bd,w}^{\bf j}\subseteq D^{\nr}_S\cap D^\#$ if ${\bf j}\in (\mathbb{Z}^{\geq 0})^{\Phi^+_{w_0^{-1}w^{-1}}\setminus\Delta}$. As in the proof of Thm.\ \ref{Mbdfingen}, there exists an integer $r\geq 0$ such that for all ${\bf j},{\bf j'}\in (\mathbb{Z}^{\geq 0})^{\Phi^+_{w_0^{-1}w^{-1}}\setminus\Delta}$ with $j_i\leq j'_i$ and $\sum_{i=1}^dj_i=r$ we have $D_{bd,w}^{\bf j}=D_{bd,w}^{\bf j'}$. Pick a finite set $\{x_1^{({\bf j})},\dots,x_{N_{\bf j}}^{({\bf j})}\}$ of generators of $D_{bd,w}^{\bf j}$ as $E_{\Delta\setminus S}^+$-modules for all ${\bf j}\in (\mathbb{Z}^{\geq 0})^{\Phi^+_{w_0^{-1}w^{-1}}\setminus\Delta}$ with $\sum_{i=1}^dj_i\leq r$. Further, lift these elements to $\{y_1^{({\bf j})},\dots,y_{N_{\bf j}}^{({\bf j})}\}\subset \Fil^{\bf j}(D_w\cap\mathbb{M}_\infty^{bd}(D^\#))$. We claim that the finite set 
\begin{align}
\bigcup_{\substack{{\bf j}\in (\mathbb{Z}^{\geq 0})^{\Phi^+_{w_0^{-1}w^{-1}}\setminus\Delta}\\ \sum_{i=1}^dj_i\leq r}} \{y_1^{({\bf j})},\dots,y_{N_{\bf j}}^{({\bf j})}\} \label{setofgenw}
\end{align}
generates $D_w\cap\mathbb{M}^{bd}_\infty(D^\#)$. Let $M\leq D_w\cap\mathbb{M}^{bd}_\infty(D^\#)$ be the $\mathbb{F}_q\bs N_{0,w_0^{-1}w^{-1}}\js$-submodule generated by \eqref{setofgenw}. Assume that $D_w\cap\mathbb{M}^{bd}_\infty(D^\#)\not\subseteq M$. By compactness of $M$ we have $M=\varprojlim_k \pr_{k,\infty}(M)=\varprojlim_k M+\Ker(\pr_{k,\infty})/\Ker(\pr_{k,\infty})$. So there is a $k\geq 0$ such that $D_w\cap \mathbb{M}^{bd}_\infty(D^\#)\not\subseteq M+\Ker(\pr_{k,\infty})$. By Lemma \ref{comparefilpr} we also obtain $D_w\cap\mathbb{M}^{bd}_\infty(D^\#)\not\subseteq M+\Fil^{\bf \tilde{k}}\mathbb{M}_{\infty,0}(D)$ hence $D_w\cap\mathbb{M}^{bd}_\infty(D^\#)\not\subseteq (M+\Fil^{\bf \tilde{k}}\mathbb{M}_{\infty,0}(D))\cap D_w=M+\Fil^{\bf \tilde{k}}D_w$. Since the set $(\mathbb{Z}^{\geq 0})^{\Phi^+_{w_0^{-1}w^{-1}}\setminus\Delta}$ is well-ordered, there is a smallest ${\bf j}^{(0)}\in (\mathbb{Z}^{\geq 0})^d$ such that $D_w\cap\mathbb{M}^{bd}_\infty(D^\#)\not\subseteq M+\Fil^{{\bf j}^{(0)}+}D_w$. By Lemma \ref{gradedDw} we have ${\bf j}^{(0)}\in (\mathbb{Z}^{\geq 0})^{\Phi^+_{w_0^{-1}w^{-1}}\setminus\Delta}$. So from now on we restrict the filtration on $D_w$, on $D_w\cap \mathbb{M}^{bd}_\infty(D^\#)$, and on $M$ to the well-ordered subset $(\mathbb{Z}^{\geq 0})^{\Phi^+_{w_0^{-1}w^{-1}}\setminus\Delta}\subseteq (\mathbb{Z}^{\geq 0})^d$. Pick an element $x\in D_w\cap\mathbb{M}^{bd}_\infty(D^\#)$ such that $x\notin M+\Fil^{{\bf j}^{(0)}+}D_w$. By the minimality of ${\bf j}^{(0)}$, we have $$x\in \bigcap_{{\bf j}^{(0)}>{\bf j'}\in (\mathbb{Z}^{\geq 0})^{\Phi^+_{w_0^{-1}w^{-1}}\setminus\Delta}}(M+\Fil^{{\bf j'}+}D_w)\ .$$
\begin{lem}\label{intersect+w}
We have $$M+\Fil^{{\bf j}^{(0)}}D_w= \bigcap_{{\bf j}^{(0)}>{\bf j'}\in (\mathbb{Z}^{\geq 0})^{\Phi^+_{w_0^{-1}w^{-1}}\setminus\Delta}}(M+\Fil^{{\bf j'}+}D_w)\ .$$
\end{lem}
\begin{proof}
This is completely analogous to the proof of Lemma \ref{intersect+}, but for the convenience of the reader we give details. The containment $\subseteq$ is trivial, let us prove the other containment. Put $d_w:=|\Phi^+_{w_0^{-1}w^{-1}}\setminus\Delta|$ and renumber the indices of ${\bf j}$ by $1,\dots,d_w$ in an order preserving way, so we identify $(\mathbb{Z}^{\geq 0})^{\Phi^+_{w_0^{-1}w^{-1}}\setminus\Delta}$ with $(\mathbb{Z}^{\geq 0})^{d_w}$. If ${\bf j}^{(0)}\in $ is a successor of some ${\bf j'}\in (\mathbb{Z}^{\geq 0})^{d_w}$ then we have $\Fil^{{\bf j'}+}D_w=\Fil^{{\bf j}^{(0)}}D_w$, so there is nothing to prove. Otherwise we have $j^{(0)}_1=\dots=j^{(0)}_{r}=0$, but $j^{(0)}_{r+1}\neq 0$ for some $1\leq r< d_w$. Let $h^{(n)}_1=\dots=h^{(n)}_{r-1}=0$, $h^{(n)}_r=n$, $h^{(0)}_{r+1}=j^{(0)}_{r+1}-1,$ and $h^{(n)}_i=j^{(0)}_i$ for all $r+1<i\leq d$ so that ${\bf h}^{(1)}<{\bf h}^{(2)}<\cdots<{\bf h}^{(n)}<\cdots<{\bf j}^{(0)} \in (\mathbb{Z}^{\geq 0})^{d_w}$. Further, write $x=y_n+z_n$ with $y_n\in M$ and $z_n\in \Fil^{{\bf h}^{(n)}+}D_w$. Note that for any index ${\bf h}^{(n)}<{\bf j}\in (\mathbb{Z}^{\geq 0})^{d_w}$ we either have ${\bf j}^{(0)}\leq {\bf j}$ or $j_{i}=h^{(n)}_i$ for all $r\leq i\leq d$. In the latter case ${\bf b}^{\bf j}$ is a multiple of ${\bf b}^{{\bf h}^{(n)}}$ (as we have $h^{(n)}_i=0$ for all $1\leq i<r$). Therefore we may write $z_n={\bf b}^{{\bf h}^{(n)}}z_n'+w_n$ where $z_n'\in D_w$ and $w_n\in \Fil^{{\bf j}^{(0)}}D_w$. Since $M$ is compact, we may pass to a subsequence to achieve that $y_n$ is convergent. On the other hand we have ${\bf b}^{{\bf h}^{(n)}}z_n'\to 0$ regardless what $z_n'$ is since for all $k\geq 0$ we have $\pr_{k,\infty}({\bf b}^{{\bf h}^{(n)}}D_w)=0$ for $n\geq p^{k_r}$ where $u_r^{p^{k_r}}\in H_{\Delta,k}$. Therefore $w_n=x-y_n-{\bf b}^{{\bf h}^{(n)}}z_n'$ is convergent, too. However, $\Fil^{{\bf j}^{(0)}}D_w$ is closed in $D_w$ whence $x=\lim y_n+\lim w_n\in M+\Fil^{{\bf j}^{(0)}}D_w$.
\end{proof}
By Lemma \ref{intersect+w} we may write $x=y+z$ with $y\in M$ and $z\in \Fil^{{\bf j}^{(0)}}D_w$. Since $M\subseteq D_w\cap\mathbb{M}_\infty^{bd}(D^\#)$, we also have $z\in D_w\cap\mathbb{M}_\infty^{bd}(D^\#)$. There is a ${\bf j}^{(1)}\in (\mathbb{Z}^{\geq 0})^{d_w}$ with $j^{(1)}_i\leq j^{(0)}_i$ for all $i=1,\dots,d_w$ and $\sum_{i=1}^{d_w}j^{(1)}_i\leq r$ such that $D^{{\bf j}^{(1)}}_{bd,w}=D^{{\bf j}^{(0)}}_{bd,w}$ so that we obtain elements $\lambda_1,\dots,\lambda_{N_{{\bf j}^{(1)}}}\in\mathbb{F}_q\bs N_{0,w_0^{-1}w^{-1}}\js$ such that $$z-\sum_{\ell=1}^{N_{{\bf j}^{(1)}}}\lambda_\ell{\bf b}^{{\bf j}^{(0)}-{\bf j}^{(1)}}y_\ell^{{\bf j}^{(1)}}\in \Fil^{{\bf j}^{(0)}+}D_w$$
which contradicts to our assumption $x=y+z\notin M+\Fil^{{\bf j}^{(0)}+}D_w$ since $y+\sum_{\ell=1}^{N_{{\bf j}^{(1)}}}\lambda_\ell{\bf b}^{{\bf j}^{(0)}-{\bf j}^{(1)}}y_\ell^{{\bf j}^{(1)}}$ lies in $M$.
\end{proof}

Let now $\pi$ be a smooth representation of $G=\GL_n(\Qp)$ over $\Fq$ and $M\in \mathcal{M}^0_\Delta(\pi^{H_{\Delta,0}})$. Then $M_\infty:=B_+M$ is a $B_+$-subrepresentation of $\pi$ and its Pontryagin dual $M_\infty^\vee$ equals the image of $\pi^\vee$ in $\mathbb{M}_{\infty,0}(M^\vee[X_\Delta^{-1}])$. In \cite[Section 4.5]{MultVar} a $G$-equivariant sheaf $\mathfrak{Y}_{\pi,M}$ of topological $\Fq$-vectorspaces is defined on $G/B$ (in the $p$-adic topology of $G/B$) such that the space $\mathfrak{Y}_{\pi,M}(\mathcal{C}_0)$ of sections on the open compact cell $\mathcal{C}_0=N_0w_0B$ is the étale hull $\widetilde{M_\infty^\vee}$ of $M_\infty^\vee$ in the sense of \cite[section 2.3]{links}. Further, by \cite[Cor.\ 4.21]{MultVar} there is a $G$-equivariant continuous map $\beta_{G/B,M}\colon \pi^\vee\to \mathfrak{Y}_{\pi,M}(G/B)$ whose image is the Pontryagin dual of the $G$-subrepresentation of $\pi$ generated by $M$.

For each $w\in N_G(T)/T$ we define $$\Fil^w_M(\pi^\vee):=\beta_{G/B,M}^{-1}(\Ker(\res_{\Fil^w(G/B)}^{G/B}\colon \mathfrak{Y}_{\pi,M}(G/B)\to  \mathfrak{Y}_{\pi,M}(\Fil^w(G/B))))\ .$$
This is an ascending filtration of $\pi^\vee$ by closed $B$-subrepresentations. Taking Pontryagin duals we obtain a descending filtration $\Fil^w_M(\pi):=(\pi^\vee/\Fil^w_M(\pi^\vee))^\vee\leq \pi$ whose graded pieces $\mathrm{gr}^w_M(\pi)$ are smooth representations of $B$ over $\Fq$. In particular, we have  
\begin{align*}
\mathrm{gr}^w_M(\pi)^\vee\cong (\Fil^w_M(\pi)/\Fil^{>w}_M(\pi))^\vee\cong \left((\pi^\vee/\Fil^w_M(\pi^\vee))^\vee\big{/} (\pi^\vee/\Fil^{>w}_M(\pi^\vee))^\vee\right)^\vee\cong\\ 
\cong \left((\res_{\Fil^w(G/B)}\circ\beta_{G/B,M}(\pi^\vee))^\vee\big{/} (\res_{\Fil^{>w}(G/B)}\circ\beta_{G/B,M}(\pi^\vee))^\vee\right)^\vee\cong \beta_{\Fil^w(G/B),M}(\pi^\vee)\cap \Ker(\res_{\Fil^{>w}(G/B)}^{\Fil^{w}(G/B)})\ .
\end{align*}
We obtain a continuous and $B$-equivariant embedding $\beta_{\mathcal{C}^w,M}\colon \mathrm{gr}^w_M(\pi)^\vee\hookrightarrow \mathfrak{Y}_{\pi,M}^w(\mathcal{C}^w)$ into the global sections  of the sheaf $\mathfrak{Y}_{\pi,M}^w$ provided by Lemma \ref{kersheaf} in case $\mathfrak{Y}=\mathfrak{Y}_{\pi,M}$. Put $$M_{w,\infty}=\left(\res^{\mathcal{C}^w}_{\mathcal{C}_0^w}\circ\beta_{\mathcal{C}^w,M}(\mathrm{gr}_{M}^{w}(\pi)^\vee)\right)^\vee\leq \mathrm{gr}^w(\pi)\ .$$ This is a $B_+$-subrepresentation of $\mathrm{gr}_M^w(\pi)$ as we have $B_+\mathcal{C}_0^w\subseteq \mathcal{C}_0^w$. By Lemma \ref{Cwsubsetww_0C} we may also identify $M_{w,\infty}^\vee\cong \beta_{ww_0\mathcal{C}_0,M}(\pi^\vee)\cap \Ker(\res^{ww_0\mathcal{C}_0}_{ww_0\mathcal{C}_0\cap\Fil^{>w}(G/B)})$ as a subquotient of the global sections of $\mathfrak{Y}_{\pi,M}$. Further, in case $w=w_0$ we have $M_{w_0,\infty}=M_\infty$. We need the following weaker, but more general version of \cite[Lemma 4.23]{MultVar}.

\begin{lem}\label{M'containedinMw}
$M_{w,\infty}\leq \mathrm{gr}_{M}^{w}(\pi)$ is a generating $B_+$-subrepresentation for any $w\in N_G(T)/T$. In particular, for any $M'\in \mathcal{M}^0_\Delta(\mathrm{gr}_{M}^{w}(\pi)^{H_{\Delta,0}})$ we have $M'\subseteq M_{w,\infty}$.
\end{lem}
\begin{proof}
We write $\mathcal{C}^w=NwB/B=\bigcup_{n\geq 0} s^{-n}\mathcal{C}^w_0$. So we compute
\begin{align*}
\mathrm{gr}_{M}^{w}(\pi)^\vee\cong \beta_{\mathcal{C}^w,M}(\mathrm{gr}_{M}^{w}(\pi)^\vee) =\varprojlim_n\res^{\mathcal{C}^w}_{s^{-n}\mathcal{C}^w_0}\circ\beta_{\mathcal{C}^w,M}(\mathrm{gr}_{M}^{w}(\pi)^\vee)=\\
=\varprojlim_ns^{-n}\res^{\mathcal{C}^w}_{\mathcal{C}^w_0}\circ\beta_{\mathcal{C}^w,M}(\mathrm{gr}_{M}^{w}(\pi)^\vee)=\varprojlim_n s^{-n}\cdot M_{w,\infty}^\vee=(\bigcup_{n}s^{-n}M_{w,\infty})^\vee\ .
\end{align*}
Now let $M'\in \mathcal{M}^0_\Delta(\mathrm{gr}_{M}^{w}(\pi)^{H_{\Delta,0}})$ be arbitrary. Then $M'$ is generated by a finite set $\{m_1,\dots,m_r\}\subset M'$ as a module over $\Fq\bs N_{\Delta,0}\js [F_\Delta]$. So there exists an integer $k\geq 0$ such that $s^km_i\in M_{w,\infty}$ for all $i=1,\dots,r$. By \cite[Lemma 4.16]{MultVar} we have $M'\subseteq B_+s^kM'=B_+\{s^km_1,\dots,s^km_r\}\subseteq M_{w,\infty}$.
\end{proof}

\begin{lem}\label{ww0bij}
The map
\begin{eqnarray*}
ww_0\colon M_\infty^\vee\cap D_w=\beta_{\mathcal{C}_0,M}(\pi^\vee)\cap D_w& \to & M_{w,\infty}^\vee\\
\beta_{\mathcal{C}_0,M}(\mu)&\mapsto & \beta_{ww0\mathcal{C}_0,M}(ww_0\mu)
\end{eqnarray*}
is a bijection.
\end{lem}
\begin{proof}
Note that we have $\beta_{ww_0\mathcal{C}_0,M}(\pi^\vee)=ww_0\cdot \beta_{\mathcal{C}_0,M}(\pi^\vee)=ww_0\cdot M_\infty^\vee=(ww_0 M_\infty)^\vee$. Hence $M_{w,\infty}\cong ww_0M_\infty/K_{w,\infty}$ is the quotient of $ww_0M_\infty=\{ww_0x\mid x\in M_\infty\}$ by the subspace $$K_{w,\infty}:=\left\{ww_0x\in ww_0M_\infty \mid \beta_{ww_0\mathcal{C}_0,M}(\mu)(ww_0x)=0\text{ for all }\mu\in \pi^\vee\text{ with }\beta_{ww_0\mathcal{C}_0\cap\Fil^{>w}(G/B), M}(\mu)=0 \right\}\ .$$
The lemma follows by taking duals using Lemma \ref{Cwsubsetww_0C}.
\end{proof}

\begin{pro}\label{Mbdwkerfingen}
Assume that $\pi$ is a smooth representation of $\GL_n(\mathbb{Q}_p)$ over $\Fq$ and $M\in \mathcal{M}_\Delta^0(\pi^{H_{\Delta,0}})$. Then $M_{w,\infty}^\vee$ is finitely generated over $\Fq\bs N_{0,ww_0}\js$ for all $w\in N_G(T)/T$. In particular, if $w$ is different from $w_0$ then $M_{w,\infty}^\vee$ is a finitely generated torsion $\Fq\bs N_0\js$-module. 
\end{pro}
\begin{proof}
By Prop.\ \ref{MbdcapDwfin} and \ref{prcontainedinDhash} $M_\infty^\vee\cap D_w$ is finitely generated over $\Fq\bs N_{0,w_0^{-1}w^{-1}}\js$. By Lemma \ref{ww0bij} this implies that $M_{w,\infty}^\vee=ww_0\cdot (M_\infty^\vee\cap D_w)$ is finitely generated over $\Fq\bs N_{0,ww_0}\js=\Fq\bs ww_0 N_{0,w_0^{-1}w^{-1}}w_0^{-1}w^{-1}\js$. Finally, if $w\neq w_0$ then the index of $N_{0,ww_0}$ is infinite in $N_0$, so any $\Fq\bs N_0\js$-module is torsion if it is finitely generated over  $\Fq\bs N_{0,ww_0}\js$.
\end{proof}

\begin{thm}\label{DveeDelta0gr}
Assume that $\pi$ is a smooth representation of $\GL_n(\mathbb{Q}_p)$ over $\Fq$ and $M\in \mathcal{M}_\Delta^0(\pi^{H_{\Delta,0}})$. Then we have $D^\vee_\Delta(\mathrm{gr}_{M}^{w}(\pi))=0$ for all $w_0\neq w\in N_G(T)/T$.
\end{thm}
\begin{proof}
Assume that $M'$ is an object in $\mathcal{M}_\Delta^0(\mathrm{gr}_{M}^{w}(\pi)^{H_{\Delta,0}})$. Then $M'\subseteq M_{w,\infty}$ by Lemma \ref{M'containedinMw}. Since the elements of $M'$ are fixed by $H_{\Delta,0}$, we also deduce $M'\subseteq M_{w,\infty}^{N_{0,ww_0}\cap H_{\Delta,0}}$. So ${M'}^\vee\leq {M'}^\vee[X_\Delta^{-1}]$ arises as a quotient of $(M_{w,\infty}^\vee)_{N_{0,ww_0}\cap H_{\Delta,0}}$ which is finitely generated over $\Fq\bs N_{\Delta,0,ww_0}\js=\Fq\bs N_{0,ww_0}/(N_{0,ww_0}\cap H_{\Delta,0})\js$ by Proposition \ref{Mbdwkerfingen}. Whenever $w\neq w_0$, $N_{\Delta,0,ww_0}$ is a subgroup of $N_{\Delta,0}$ of infinite index whence ${M'}^\vee$ is a torsion $\Fq\bs N_{\Delta,0}\js$-module. Hence $M'=0$ as the étale $T_+$-module ${M'}^\vee[X_\Delta^{-1}]$ has no nonzero torsion $\Fq\bs N_{\Delta,0}\js$-submodule by \cite[Cor.\ 3.16]{MultVarGal}.
\end{proof}

\begin{thm}\label{DveeDeltagrw_0}
Let $\pi$ be a smooth representation of $\GL_n(\mathbb{Q}_p)$ over $\mathcal{O}_F/\varpi^h$ for some finite extension $F/\Qp$ with valuation ring $\mathcal{O}_F$ and uniformizer $\varpi$. Assume that there is an object $M\in \mathcal{M}_\Delta(\pi^{H_{\Delta,0}})$ such that the map $\beta_{G/B,M}\colon \pi^\vee\to \mathfrak{Y}_{\pi,M}(G/B)$ is injective. Then we have $D^\vee_\Delta(\pi)=M^\vee[X_\Delta^{-1}]$.
\end{thm}
\begin{proof}
At first assume that $h=1$ so that $\pi$ is a smooth representation over $\Fq:=\mathcal{O}_F/\varpi$. Note that the filtration $\Fil^\bullet_M$ by $B$-subrepresentations is exhaustive on $\pi$ as we have $\Fil^1_M(\pi)=(\pi^\vee/\Ker(\beta_{G/B,M}))^\vee=\pi$. We deduce $D^\vee_\Delta(\pi)=D^\vee_\Delta(\mathrm{gr}_M^{w_0}(\pi))$ by Theorem \ref{DveeDelta0gr} and the right exactness of $D^\vee_\Delta$ \cite[Theorem A]{MultVar}. Let $M'\in \mathcal{M}_\Delta(\mathrm{gr}_M^{w_0}(\pi)^{H_{\Delta,0}})$ be arbitrary. By \cite[Lemma 4.23]{MultVar}, we have $M'\leq M_\infty=B_+M$ whence $M'\leq M_\infty^{H_{\Delta,0}}=M_\infty\cap \pi^{H_{\Delta,0}}$. Put $M_1:=M+M'\in \mathcal{M}_\Delta(\mathrm{gr}_M^{w_0}(\pi)^{H_{\Delta,0}})$. So we obtain a surjective map $\Fq\bs N_{\Delta,0}\js\otimes_{\Fq\bs N_0\js}M_\infty^\vee\to M_1^\vee$ showing that the rank of $M_1^\vee[X_\Delta^{-1}]$ is bounded above by the rank of $\Fq\bs N_{\Delta,0}\js\otimes_{\Fq\bs N_0\js}M_\infty^\vee$ which is equal to the rank of $M^\vee[X_\Delta^{-1}]$ by Corollary \ref{kerneltorsion}. Hence the quotient map $M_1^\vee[X_\Delta^{-1}]\twoheadrightarrow M^\vee[X_\Delta^{-1}]$ is an isomorphism. Since $M'$ was arbitrary, the map $$D^\vee_\Delta(\pi)=\varprojlim_{M'\in \mathcal{M}_\Delta(\mathrm{gr}_M^{w_0}(\pi)^{H_{\Delta,0}})}{M'}^\vee[X_\Delta^{-1}]\twoheadrightarrow M^\vee[X_\Delta^{-1}]$$ is also an isomorphism.

The general case $h>1$ is deduced by a simple devissage argument as follows. We certainly have an epimorphism $f_M\colon D^\vee_\Delta(\pi)\twoheadrightarrow M^\vee[X_\Delta^{-1}]$. On the other hand, $\pi^\vee[\varpi]=(\pi/\varpi\pi)^\vee$ maps isomorphically under $\beta_{G/B,M}$ to $\beta_{G/B,M}(\pi^\vee)\cap \mathfrak{Y}_{\pi,M}(G/B)[\varpi]$, so applying the $h=1$ one case we obtain $D^\vee_\Delta(\pi/\varpi\pi)\cong M^\vee[X_\Delta^{-1}][\varpi]$ ($\varpi$-torsion part). Moreover, we also deduce that the map $(\varpi\pi)^\vee\cong\pi^\vee/(\pi^\vee[\varpi])\to \mathfrak{Y}_{\pi,M}(G/B)/\mathfrak{Y}_{\pi,M}(G/B)[\varpi]\cong \varpi\mathfrak{Y}_{\pi,M}(G/B)$ is also injective whence we get $D^\vee_\Delta(\varpi\pi)\cong \varpi M^\vee[X_\Delta^{-1}]$ by induction. By the right exactness of $D^\vee_\Delta$ we compute (see \cite[section 2.3]{MultVar} for the definition of generic length)
\begin{align*}
\genlength(M^\vee[X_\Delta^{-1}])=\genlength(M^\vee[X_\Delta^{-1}][\varpi])+\genlength(\varpi M^\vee[X_\Delta^{-1}])=\\
=\genlength(D^\vee_\Delta(\pi/\varpi\pi))+\genlength(D^\vee_\Delta(\varpi\pi))\geq \genlength(D^\vee_\Delta(\pi)\ , 
\end{align*}
so $f_M$ must be an isomorphism.
\end{proof}

\begin{cor}\label{finitenessofDveeDelta}
Assume that $\pi$ is a smooth representation of $\GL_n(\mathbb{Q}_p)$ over $\mathcal{O}_F/\varpi^h$ of finite length for some finite extension $F/\Qp$ with valuation ring $\mathcal{O}_F$ and uniformizer $\varpi$. Then $D_\Delta^\vee(\pi)$ is a finitely generated $\OED$-module.
\end{cor}
\begin{proof}
By the right exactness \cite[Theorem A]{MultVar} we may assume that $\pi$ is irreducible and $h=1$ so that the coefficient ring is $\mathcal{O}_F/\varpi\cong \Fq$. If $D_\Delta^\vee(\pi)$ is $0$ then there is nothing to prove. Otherwise there exists an object $M\neq 0$ in $\mathcal{M}^0_\Delta(\pi^{H_{\Delta,0}})$. The $G$-equivariant map $\beta_{G/B,M}\colon\pi^\vee\to \mathfrak{Y}_{\pi,M}(G/B)$ is nonzero hence injective since $\pi$ is irreducible. The statement follows from Thm.\ \ref{DveeDeltagrw_0}.
\end{proof}

\begin{cor}\label{irredtoirred}
Assume that $\pi$ is an irreducible smooth representation of $\GL_n(\mathbb{Q}_p)$ over $\Fq$ such that $D_\Delta^\vee(\pi)\neq 0$. Then $\pi$ is admissible and $\mathbb{V}^\vee(D_\Delta^\vee(\pi))$ is an irreducible representation of $\GQpD\times \Qp^\times$ over $\Fq$.
\end{cor}
\begin{proof}
At first we show that $\mathbb{V}^\vee(D_\Delta^\vee(\pi))$ is irreducible. Since $\mathbb{V}_\Delta$ is an equivalence of categories \cite{MultVarGal}, we are reduced to showing that $D:=D^\vee_\Delta(\pi)$ is irreducible. Assume for contradiction that we have a short exact sequence $0\to D'\to D\to D''\to 0$ of étale $T_+$-modules over $E_\Delta$ such that $D'\neq 0\neq D''$. Pick an object $M''\in \mathcal{M}_\Delta(\pi^{H_{\Delta,0}})$ such that ${M''}^\vee[X_\Delta^{-1}]=D''$. The image of $\beta_{\mathcal{C}_0,M''}=\res^{G/B}_{\mathcal{C}_0}\circ\beta_{G/B,M''}\colon\pi^\vee\to \mathfrak{Y}_{\pi,M''}(\mathcal{C}_0)$ can be identified with $(B_+M'')^\vee$ by \cite[Lemma 4.16]{MultVar}. In particular, $\beta_{M'',G/B}$ is not the zero map hence it is injective as $\pi^\vee$ is irreducible. By Theorem \ref{DveeDeltagrw_0} we deduce $D\cong D^\vee_\Delta(\pi)\cong {M''}^\vee[X_\Delta^{-1}]\cong D''$ which contradicts to $D'\neq 0$.

For the admissibility put $U^{(1)}:=\Ker(\GL_n(\Zp)\to \GL_n(\Fp))\leq \GL_n(\Qp)$. Note that $\mathcal{C}_0=N_0w_0B$ is $U^{(1)}$-invariant and $U^{(1)}\cap N_0$ has finite index in $N_0$ (see, for instance \cite[Lemma 4.18]{MultVar}). Therefore $M_\infty^\vee=\beta_{M,\mathcal{C}_0}(\pi^\vee)$ is finitely generated as a module over $\Fq\bs U^{(1)}\js$, since it is finitely generated over $\Fq\bs N_0\js$ by Corollary \ref{Minftyadmissible}. On the other hand, we have a covering $\bigcup_{g\in \GL_n(\Zp)/U^{(1)}}g\mathcal{C}_0 =G/B$ by the Iwasawa decomposition. Since $U^{(1)}$ is normal in $\GL_n(\Zp)$, $\beta_{M,g\mathcal{C}_0}(\pi^\vee)=g\beta_{\mathcal{C}_0,M}(\pi^\vee)$ is also a finitely generated $\Fq\bs U^{(1)}\js$-module. The statement follows from the sheaf property that the composite map
\begin{align*}
\pi^\vee\overset{\beta_{G/B,M}}{\hookrightarrow}\mathfrak{Y}_{\pi,M}(G/B)\hookrightarrow \bigoplus_{g\in \GL_n(\Zp)/U^{(1)}} \mathfrak{Y}_{\pi,M}(g\mathcal{C}_0)
\end{align*}
is injective noting that $\Fq\bs U^{(1)}\js$ is (left and right) noetherian.
\end{proof}

The following statement shows that the functor $D^\vee_\Delta$ detects isomorphisms of irreducible objects unless it vanishes.
\begin{cor}\label{detectsiso}
Assume that $\pi_1$ and $\pi_2$ are irreducible smooth representations of $\GL_n(\mathbb{Q}_p)$ over $\Fq$ such that $D_\Delta^\vee(\pi_1)\cong D_\Delta^\vee(\pi_2)\neq 0$. Then we have $\pi_1\cong \pi_2$.
\end{cor}
\begin{proof}
Put $D:=D_\Delta^\vee(\pi_1)$ so we have $D_\Delta^\vee(\pi_1\oplus\pi_2)\cong D\oplus D$. Let 
\begin{eqnarray*}
f\colon D^\vee_\Delta(\pi_1\oplus\pi_2)\cong D\oplus D&\to & D\\
(x,y)&\mapsto & x+y
\end{eqnarray*}
and pick $M\in \mathcal{M}_\Delta^0((\pi_1\oplus\pi_2)^{H_{\Delta,0}})$ be an object such that we have $D^\vee_\Delta(\pi_1\oplus\pi_2)\overset{f}{\twoheadrightarrow} M^\vee[X_\Delta^{-1}]\cong D$, so the composites of both embeddings $D^\vee_\Delta(\pi_j)\to D^\vee_\Delta(\pi_1\oplus\pi_2)$ ($j=1,2$) with $f$ are isomorphisms. In particular, the composite maps $$\pi_j^\vee\hookrightarrow \pi_1^\vee\oplus\pi_2^\vee\overset{\beta_{G/B,M}}{\to}\mathfrak{Y}_{\pi_1\oplus\pi_2,M}(G/B)$$ are nonzero hence injective for $j=1,2$ as $\pi_1$ and $\pi_2$ are irreducible. However, $\beta_{G/B,M}$ \emph{cannot} be injective by Theorem \ref{DveeDeltagrw_0} since $D^\vee_\Delta(\pi_1\oplus\pi_2)\not\cong D$. Therefore we have constructed an isomorphism from both $\pi_1^\vee$ and $\pi_2^\vee$ to the image of $\beta_{G/B,M}$.
\end{proof}

We end this section with the following conditional finiteness result for Breuil's functor \cite{Breuil}.
\begin{cor}\label{Breuilfin}
Assume that $\pi$ is an irreducible smooth representation of $\GL_n(\mathbb{Q}_p)$ over $\Fq$ such that $D_\Delta^\vee(\pi)\neq 0$. Then we have $D^\vee_\xi(\pi)\cong \Fq\bg X\jg\otimes_{\Fq\bg N_{\Delta,0}\jg,\ell}D^\vee_\Delta(\pi)$. In particular, $D_\xi^\vee(\pi)$ is a finitely generated $\Fq\bg X\jg$-module and the Galois representation $\mathbb{V}^\vee\circ D^\vee_\xi(\pi)$ attached by Fontaine's equivalence is finite dimensional and isomorphic to the restriction of $\mathbb{V}_\Delta^\vee\circ D_\Delta^\vee(\pi)$ to the diagonal embedding $G_{\Qp}\hookrightarrow \GQpD\times\Qp^{\times}$. Here $D_\xi^\vee$ denotes Breuil's functor \cite{Breuil}.
\end{cor}
\begin{proof}
Since $D_\Delta^\vee(\pi)\neq 0$, there is an object $M\in\mathcal{M}_\Delta(\pi^{H_{\Delta,0}})$. Since $\pi$ is irreducible, the map $\beta_{G/B,M}\colon \pi^\vee\to \mathfrak{Y}_{\pi,M}(G/B)$ is injective and the filtration $\Fil^\bullet_M$ is exhaustive on $\pi$. Using the notations of \cite{Breuil} let $M_0\in \mathcal{M}(\mathrm{gr}_M^w(\pi)^{N_1})$ be an object such that $M_0^\vee$ has no $X$-torsion (this can always be achieved by passing to a finite index subobject of $M_0$). Then by Lemma \ref{M'containedinMw} $M_0$ is contained in $M_{w,\infty}$ hence also in $M_{w,\infty}^{N_1}$, so we obtain a surjective map 
$$\Fq\bg X\jg\otimes_{\ell,\Fq\bg N_{\Delta,0}\jg} \Fq\bg N_{\Delta,0}\jg\otimes_{\Fq\bs N_0\js} M_{w,\infty}^\vee \cong\Fq\bg X\jg\otimes_{\ell,\Fq\bs N_0\js}M_{w,\infty}^\vee\twoheadrightarrow M_0^\vee[X^{-1}]\ .$$
Now if $w\neq w_0$ then $\Fq\bg N_{\Delta,0}\jg\otimes_{\Fq\bs N_0\js} M_{w,\infty}^\vee$ is a torsion $E_\Delta\cong \Fq\bg N_{\Delta,0}\jg$-module by Proposition \ref{Mbdwkerfingen} admitting a semilinear action of the group $T_0$. Therefore the global annihilator of $\Fq\bg N_{\Delta,0}\jg\otimes_{\Fq\bs N_0\js} M_{w,\infty}^\vee$ is a nonzero $T_0$-invariant ideal in $E_\Delta$ whence it contains $1$ by \cite[Prop.\ 2.1]{MultVar}. We deduce $D^\vee_\xi(\mathrm{gr}_M^w(\pi))$ for all $w_0\neq w\in N_G(T)$ since $M_0$ was arbitrary. On the other hand, we compute $$\Fq\bg N_{\Delta,0}\jg\otimes_{\Fq\bs N_0\js} M_{w_0,\infty}^\vee=\Fq\bg N_{\Delta,0}\jg\otimes_{\Fq\bs N_0\js} M_{\infty}^\vee\overset{\text{Cor.\ \ref{kerneltorsion}}}{\cong} M^\vee[X_\Delta^{-1}]\overset{\text{Thm.\ \ref{DveeDeltagrw_0}}}{\cong}D^\vee_\Delta(\pi)\ . $$
The first statement follows from the right exactness of $D^\vee_\xi$ \cite[Prop.\ 3.2]{Breuil} using \cite[Remark 2 after Prop.\ 2.6]{MultVar}. The statement on the corresponding Galois representations is deduced from \cite[Cor.\ 3.10]{MultVarGal} (see also \cite[Remark 1 after Prop.\ 2.6]{MultVar}).
\end{proof}

\begin{rems}\begin{enumerate}
\item At the moment we cannot rule out the possibility that $D^\vee_\Delta(\pi)=0$, but $D^\vee_\xi(\pi)\neq 0$ for some irreducible admissible smooth representation $\pi$ of $\GL_n(\Qp)$ over $\Fq$. In this case it may happen that $D^\vee_\xi(\pi)$ is not even finitely generated. However, we conjecture the isomorphism $D^\vee_\xi(\pi)\cong \Fq\bg X\jg\otimes_{\Fq\bg N_{\Delta,0}\jg,\ell}D^\vee_\Delta(\pi)$ for all admissible smooth representations $\pi$. The conjecture is true in case of successive extensions of subquotients of principal series by \cite[Cor.\ 3.20]{MultVar}.
\item One striking question is to characterize the irreducible smooth representations of $\GL_n(\Qp)$ on which $D^\vee_\Delta$ has nonzero value. For subquotients of principal series we answer this question in the next section. By compatibility with tensor products and parabolic induction \cite[Prop.\ 3.2, Thm.\ B]{MultVar} and the nonvanishing of Colmez' Montréal functor \cite[Thm.\ 0.10]{Colmez1} we also know the nonvanishing of $D^\vee_\Delta$ on representations of the form
\begin{align*}
\Ind_{\overline{P}}^G\sigma_1\otimes\dots\otimes\sigma_r
\end{align*}
where $P=LU\leq G$ is a standard parabolic subgroup with opposite parabolic $\overline{P}$ and Levi component $L\cong \GL_{n_1}\times \dots \times \GL_{n_r}$ such that $1\leq n_j\leq 2$ and $\sigma_j$ is irreducible for all $j=1,\dots,r$ and $\sigma_j$ is, in addition, infinite dimensional in case $n_j=2$.
\item Note that we do not assume admissibility of the representation $\pi$ above. The finiteness condition is encoded, for instance in Theorem \ref{DveeDeltagrw_0}, in the existence of an object $M\in \mathcal{M}_\Delta(\pi^{H_{\Delta,0}})$ such that $\beta_{G/B,M}$ is injective (see Corollary \ref{irredtoirred}).
\end{enumerate}
\end{rems}

\section{Applications to higher extensions of principal series}\label{sec:princser}

Let $F/\Qp$ be a finite extension of $\Qp$ with ring of integers $\mathcal{O}_F$, uniformizer $\varpi\in\mathcal{O}_F$ and residue field $\mathcal{O}_F/\varpi\cong\Fq$. Fix $\alpha\in \Delta$ and a character $\xi\colon \GQpDa\times \mathbb{Q}_p^\times\to \mathbb{F}_q^\times$. Denote by $\Rep^{\alpha,\xi}_{\mathbb{F}_q}(\GQpD\times \mathbb{Q}_p^\times)$ (resp.\ by $\Rep^{\alpha,ab}_{\mathbb{F}_q}(\GQpD\times \mathbb{Q}_p^\times)$) the full subcategory of the category $\Rep_{\mathbb{F}_q}(\GQpD\times \mathbb{Q}_p^\times)$ of finite dimensional continuous representations of $\GQpD\times \mathbb{Q}_p^\times$ over $\mathbb{F}_q$ consisting of objects $V$ on which the subgroup $\GQpDa\times \mathbb{Q}_p^\times\leq \GQpD\times \mathbb{Q}_p^\times$ acts via the character $\xi$ (resp.\ acts via its maximal abelian quotient $\GQpDa^{ab}\times \mathbb{Q}_p^\times$). Moreover, for any positive integer $h$ we denote by $\Rep_h(\GQpD\times \mathbb{Q}_p^\times):=\Rep_{\mathcal{O}_F/\varpi^h}(\GQpD\times \mathbb{Q}_p^\times)$ the category of continuous representations of $\GQpD\times \mathbb{Q}_p^\times$ on finitely generated $\mathcal{O}_F/\varpi^h$-modules. Further, put $\Rep^{\ord}_{\mathbb{F}_q}(\GQpD\times \mathbb{Q}_p^\times)$ (resp.\ $\Rep^{\ord}_h(\GQpD\times \mathbb{Q}_p^\times)$, resp.\ $\Rep^{\alpha,\xi,\ord}_{\mathbb{F}_q}(\GQpD\times \mathbb{Q}_p^\times)$, resp.\ $\Rep^{\alpha,ab,\ord}_{\mathbb{F}_q}(\GQpD\times \mathbb{Q}_p^\times)$) for the full subcategory of $\Rep_{\mathbb{F}_q}(\GQpD\times \mathbb{Q}_p^\times)$ (resp.\ of $\Rep_h(\GQpD\times \mathbb{Q}_p^\times)$, resp.\ of $\Rep^{\alpha,\xi}_{\mathbb{F}_q}(\GQpD\times \mathbb{Q}_p^\times)$, resp.\ of $\Rep^{\alpha,ab}_{\mathbb{F}_q}(\GQpD\times \mathbb{Q}_p^\times)$) whose objects are successive extensions of $1$-dimensional representations over $\Fq$. 

For any positive integer $h$ put $\SP_h:=\SP_h(G):=\SP_{\mathcal{O}_F/\varpi^h}(G)$ for the category of finite length smooth representations of $G=\GL_n(\Qp)$ on $\mathcal{O}_F/\varpi^h$-modules with all Jordan-Hölder factors isomorphic to subquotients of principal series representations (over $\Fq$). Recall \cite[Corollary 3.4 and Theorem 3.19]{MultVar}, \cite[Theorem 3.15]{MultVarGal} that the functor $\mathbb{V}^\vee\circ D^\vee_\Delta$ is exact on $\SP_h(G)$ and lands in the subcategory $\Rep^{\ord}_h(\GQpD\times \mathbb{Q}_p^\times)$. Let $\SP_{\mathbb{F}_q}^{\alpha,\xi}(G)$ (resp.\ $\SP_{\mathbb{F}_q}^{\alpha,ab}(G)$) be the full subcategory of $\SP_{\mathbb{F}_q}(G)$ whose objects are representations $\pi$ in $\SP_{\mathbb{F}_q}(G)$ such that $\mathbb{V}\circ D^\vee_\Delta(\pi)$ lies in $\Rep^{\alpha,\xi,\ord}_{\mathbb{F}_q}(\GQpD\times \mathbb{Q}_p^\times)$ (resp.\ in $\Rep^{\alpha,ab,\ord}_{\mathbb{F}_q}(\GQpD\times \mathbb{Q}_p^\times)$).

\subsection{The kernel of $D^\vee_\Delta$ on $\SP_{\mathcal{O}_F/\varpi^h}$}

For any parabolic subgroup $\overline{B}\leq \overline{P}=L\overline{U}\leq G$ with unipotent radical $\overline{U}\leq \overline{N}$ and Levi component $L\geq T$ we write $\SP_h(L)$ (and $\SP_{\Fq}(L):=\SP_1(L)$) for the category of smooth representations of $L$ over $\mathcal{O}_F/\varpi^h$ of finite length with Jordan--H\"older factors isomorphic to subquotients of principal series, ie.\ of representations parabolically induced from $\overline{B}\cap L$. We may view representations of $L$ as representations of $\overline{P}$ via the quotient map $\overline{P}\twoheadrightarrow \overline{P}/\overline{U}\cong L$. Recall \cite[Theorem 3.5]{MultVar} that $D^\vee_\Delta(\Ind_{\overline{P}}^G\pi_L)\cong E_\Delta\otimes_{E_{\Delta_P}}D^\vee_{\Delta_P}(\pi_L)$ for any object $\pi_L$ in $\SP_{\mathbb{F}_q}(L)$.

Let $\Fin_{\mathbb{F}_q}(L)$ (resp.\ $\Fin_h(L)$) be the full subcategory of $\SP_{\mathbb{F}_q}(L)$ (resp.\ of $\SP_h(L)$) consisting of representations with finite cardinality (ie.\ finite dimensional over $\Fq$, resp.\ finitely generated $\mathcal{O}_F/\varpi^h$-modules). This is a Serre subcategory as any extension of finite dimensional representations is finite dimensional. We define $\Fin^{>B}_{\mathbb{F}_q}(G)$ (resp.\ $\Fin^{>B}_h(G)$) as the full subcategory of $\SP_{\mathbb{F}_q}(G)$ (resp.\ of $\SP_h(G)$) whose objects are representations with all Jordan-Hölder factors isomorphic to subquotients of representations of the form $\Ind_{\overline{P}}^G(\pi_L)$ for some parabolic $\overline{B}<\overline{P}=L\overline{U}\leq G$ and $\pi_L$ in $\Fin_{\mathbb{F}_q}(L)$. This is a Serre subcategory of $\SP_{\mathbb{F}_q}$ (resp.\ of $\SP_h(G)$). We put $\overline{SP}_{\mathbb{F}_q}:=\overline{SP}_{\mathbb{F}_q}(G)$ (resp.\ $\overline{SP}_h:=\overline{SP}_h(G)$) for the quotient category of $\SP_{\mathbb{F}_q}(G)$ by $\Fin^{>B}_{\mathbb{F}_q}(G)$ (resp.\ for the quotient category of $\SP_h(G)$ by $\Fin^{>B}_h(G)$) and put $Q$ for the quotient functor $\SP_{\mathbb{F}_q}(G)\to \overline{SP}_{\mathbb{F}_q}(G)$ (resp.\ $\SP_h(G)\to \overline{SP}_h(G)$). Further denote by $\overline{\SP}_{\mathbb{F}_q}^{\alpha,\xi}$ the image of the category $\SP_{\mathbb{F}_q}^{\alpha,\xi}$ under $Q$.  

In order to describe the irreducible objects in $\SP_{\Fq}$ (hence in $\SP_h$ as any irreducible object is killed by $\varpi$) recall that the generalized Steinberg representation corresponding to a standard parabolic subgroup $\overline{P}$ (ie.\ $\overline{B}\leq \overline{P}$) is given by
\begin{align*}
\Sp_{\overline{P}}=\frac{\Ind_{\overline{P}}^G 1}{\sum_{\overline{Q}\supsetneq \overline{P}}\Ind_{\overline{Q}}^G 1}\ .
\end{align*}
These representations are irreducible \cite{EGKSteinberg}. Further note that if $\overline{P}\supsetneq \overline{B}$ then $\Sp_{\overline{P}}$ is in the subcategory $\Fin^{>B}_{\Fq}(G)$. By the classification \cite{Herzigclass} of irreducible representations of $G=\GL_n(\Qp)$ in characteristic $p$ (or simply by \cite{EGKSteinberg} since we are only considering subquotients of principal series), the irreducible objects in $\SP_{\mathbb{F}_q}$ are of the form $\Ind_{\overline{P}}^G\sigma_1\otimes\cdots\otimes\sigma_r$ where $\overline{P}$ is a standard parabolic subgroup with Levi $\prod_{i=1}^r\GL_{n_i}(\Qp)$ and $\sigma_i\cong \Sp_{\overline{Q_i}}\otimes (\eta_i\circ\det)$ for some smooth character $\eta_i\colon \Qp^\times\to\Fq^\times$ and standard parabolic $\overline{Q_i}\leq \GL_{n_i}(\Qp)$ ($i=1,\dots,r$) with $\eta_i\neq \eta_{i+1}$ ($i=1,\dots,r-1$).

\begin{lem}\label{KerDveeDelta}
Let $\Ind_{\overline{P}}^G\sigma_1\otimes\cdots\otimes\sigma_r$ be an irreducible representation in $\SP_{\Fq}$ as above. The following are equivalent:
\begin{enumerate}[$(i)$]
\item We have $D^\vee_\Delta(\Ind_{\overline{P}}^G\sigma_1\otimes\cdots\otimes\sigma_r)=0$. 
\item There exists an index $i\in \{1,\dots,r\}$ such that $n_i>1$ and $\overline{Q_i}$ is \emph{not} a Borel subgroup in $\GL_{n_i}(\Qp)$.
\item $\Ind_{\overline{P}}^G\sigma_1\otimes\cdots\otimes\sigma_r$ lies in the subcategory $\Fin^{>B}_{\Fq}(G)$.
\end{enumerate}
In particular, $D^\vee_\Delta$ is identically $0$ on the subcategory $\Fin^{>B}_{\Fq}(G)$. Assume $D^\vee_\Delta(\Ind_{\overline{P}}^G\sigma_1\otimes\cdots\otimes\sigma_r)\neq 0$. Then $D^\vee_\Delta(\Ind_{\overline{P}}^G\sigma_1\otimes\cdots\otimes\sigma_r)$ is a free module of rank $1$ over $E_\Delta$ on which the action of $T_+$ is given by the dual of the character $\eta_1^{\otimes n_1}\otimes\cdots\otimes\eta_r^{\otimes n_r}$.
\end{lem}
\begin{proof}
For the direction $(ii)\Rightarrow (iii)$ we use the identification $(\Ind_{\overline{P}}^G1)\otimes (\eta\circ\det)\cong \Ind_{\overline{P}}^G((\eta\circ\det)\otimes\cdots(\eta\circ\det))$ to deduce that $\Ind_{\overline{P}}^G\sigma_1\otimes\cdots\otimes\sigma_r$ is a quotient of the parabolically induced representation $\Ind_{\overline{Q}}^G\chi$ for some character $\chi$ of $\overline{Q}$ where $\overline{Q}$ is the standard Borel subgroup containing $\overline{P}$ with Levi component $\prod_{i=1}^rL_i$ and $L_i$ is the Levi component of $\overline{Q_i}$. Therefore $\Ind_{\overline{P}}^G\sigma_1\otimes\cdots\otimes\sigma_r$ lies in $\Fin^{>B}_{\Fq}(G)$ if $L_i$ is not a torus for some $i\in\{1,\dots,r\}$.

For $(iii)\Rightarrow (i)$ note that $\Fin_{\mathbb{F}_q}(L)$ is in the kernel of $D^\vee_{\Delta_P}$ whenever $\overline{P}$ is strictly bigger than $\overline{B}$ by the definiton\footnote{Breuil's functor \cite{Breuil} is defined in a different way for tori than for other reductive groups. The definition of the functor $D^\vee_\Delta$ is not made explicit in \cite{MultVar} in the special case $G=T$. However, in this case we have $\Delta=\emptyset$, $H_{\Delta,0}=\{1\}=N_{\Delta,0}$ whence $X_\Delta=1$. Hence any finite dimensional subrepresentation of $\pi$ lies in $\mathcal{M}_\Delta(\pi^{H_{\Delta,0}})$, so $D^\vee_\emptyset(\pi)=\pi^\vee$ together with the action of the monoid $T=T_+$ if $\pi$ is locally finite.} of $D^\vee_\Delta$. Indeed, in case $L$ is not a torus then $\Delta_P$ is nonempty whence $X_{\Delta_P}\neq 1$. So for any $M_0\in \mathcal{M}_{\Delta_P}(\pi_L^{H_{\Delta_P,0}})$ and finite dimensional representation $\pi_L$ of $L$ the dual $M_0^\vee$ is killed by some power of $X_\Delta$ as $M_0^\vee$ is an $\Fq\bs X_{\Delta_P}\js$-module with $\dim_{\Fq}M_0^\vee\leq \dim_{\Fq}\pi_L<\infty$. This means we have $M_0^\vee[X_{\Delta_P}^{-1}]=0$ for all $M_0\in \mathcal{M}_{\Delta_P}(\pi_L^{H_{\Delta_P,0}})$ whence $D^\vee_{\Delta_P}(\pi_L)=\varprojlim_{M_0\in \mathcal{M}_{\Delta_P}(\pi_L^{H_{\Delta_P,0}})}M_0^\vee[X_{\Delta_P}^{-1}]=0$. We deduce $D^\vee_\Delta(\Ind_{\overline{P}}^G\pi_L)=0$ by \cite[Theorem 3.5]{MultVar}. Since $D^\vee_\Delta$ is exact on $\SP_{\Fq}$ \cite[Theorem 3.19]{MultVar}, this means $D^\vee_\Delta$ is identically $0$ on $\Fin^{>B}_{\Fq}(G)$.

For $(i)\Rightarrow (ii)$ assume $\overline{Q_i}$ is a Borel subgroup in $\GL_{n_i}(\Qp)$ for all $i\in\{1,\dots,r\}$ with $n_i>1$. In this case $\Ind_{\overline{P}}^G\sigma_1\otimes\dots\sigma_r$ is the only irreducible subquotient of the principal series representation $\Ind_{\overline{B}}^G\eta_1^{\otimes n_1}\otimes\cdots\otimes\eta_r^{\otimes n_r}$ which does \emph{not} lie in the subcategory $\Fin^{>B}_{\Fq}(G)$. So we have 
\begin{align*}
D^\vee_\Delta(\Ind_{\overline{P}}^G\sigma_1\otimes\dots\sigma_r)\cong D^\vee_\Delta(\Ind_{\overline{B}}^G\eta_1^{\otimes n_1}\otimes\cdots\otimes\eta_r^{\otimes n_r})\cong E_\Delta\otimes_{\Fq}(\eta_1^{\otimes n_1}\otimes\cdots\otimes\eta_r^{\otimes n_r})^\vee
\end{align*}
as claimed.
\end{proof}

In view of Lemma \ref{KerDveeDelta} the functor $D^\vee_\Delta$ factors through $\overline{SP}_h$ and by an abuse of notation we still denote by $D^\vee_\Delta$ the induced functor on $\overline{\SP}_h$. For an object $\pi$ in $\SP_h$ let $\ell(\pi)$ be the generic rank of $D^\vee_\Delta(\pi)$ as a module over $\OED$. By Lemma \ref{KerDveeDelta} and \cite[Theorem 3.19]{MultVar} $\ell(\pi)$ equals the number of Jordan-Hölder factors of $\pi$ not belonging to $\Fin^{>B}_{\Fq}(G)$.

Our main result in this section is the following refinement of \cite[Theorem 4.26]{MultVar}. The proof is similar to that of \cite[Theorem 4.26]{MultVar}. However, \cite[Lemmata 4.24 and 4.25]{MultVar} rely on the work of Hauseux \cite{Hauseux1,Hauseux} (hence on Emerton's ordinary parts functor \cite{EmertonOrd1}) which we cannot use here as the objects are not just successive extensions of principal series but equivalence classes in the quotient category $\overline{\SP}_h$. So we shall use Theorem \ref{DveeDeltagrw_0} instead.
\begin{thm}\label{fullyfaithful}
The functor $D^\vee_\Delta$ is fully faithful on $\overline{\SP}_h$ for all $h\geq 1$.
\end{thm}
\begin{proof}
The faithfulness is clear by exactness using Lemma \ref{KerDveeDelta}.

Let $\pi_1,\pi_2$ be objects in $\SP_h$ and assume we have a morphism $f\colon D^\vee_\Delta(\pi_1)\to D^\vee_\Delta(\pi_2)$ of étale $T_+$-modules and let $D$ be the image. Pick an object $M_1\in \mathcal{M}_\Delta^{0}(\pi_1^{H_{\Delta,0}})$ with $M_1^\vee[X_\Delta^{-1}]\cong D$. We have a $G$-equivariant continuous map $\beta_{G/B,M_1}\colon \pi_1^\vee\to \mathfrak{Y}_{\pi_1,M_1}(G/B)$. We denote by $\pi_3:=\beta_{G/B,M_1}(\pi_1^\vee)^\vee$ the Pontryagin dual of the image. So $\pi_3$ is a subrepresentation in $\pi_1$ and we have $D^\vee_\Delta(\pi_3)\cong M_1^\vee[X_\Delta^{-1}]\cong D$ by Theorem \ref{DveeDeltagrw_0}. On the other hand, pick an object $M_2\in \mathcal{M}_\Delta^0(\pi_2^{H_{\Delta,0}})$ such that $M_2^\vee[X_\Delta^{-1}]\cong D^\vee_\Delta(\pi_2)/D$. Put $\pi_5:=\beta_{G/B,M_2}(\pi_2^\vee)^\vee\leq \pi_2$ and $\pi_4:=\pi_2/\pi_5$. By Theorem \ref{DveeDeltagrw_0} we have $D^\vee_\Delta(\pi_5)\cong M_2^\vee[X_\Delta^{-1}]\cong D^\vee_\Delta(\pi_2)/D$, so by \cite[Theorem C]{MultVar} we obtain $D^\vee_\Delta(\pi_4)\cong D$. Finally, put
\begin{eqnarray*}
f_1\colon D^\vee_\Delta(\pi_3\oplus\pi_4)\cong D\oplus D&\to & D\\
(x,y)&\mapsto & x+y
\end{eqnarray*}
and let $M_3\in \mathcal{M}_\Delta^0((\pi_3\oplus\pi_4)^{H_{\Delta,0}})$ be an object such that we have $D^\vee_\Delta(\pi_3\oplus\pi_4)\overset{f_1}{\twoheadrightarrow} M_3^\vee[X_\Delta^{-1}]\cong D$, so the composites of both embeddings $D^\vee_\Delta(\pi_j)\to D^\vee_\Delta(\pi_3\oplus\pi_4)$ ($j=3,4$) with $f_1$ are isomorphisms. We put $\pi_6:=\beta_{G/B,M_3}(\pi_3^\vee\oplus\pi_4^\vee)^\vee$ so that the composite map $$\pi_j^\vee\hookrightarrow \pi_3^\vee\oplus\pi_4^\vee\overset{\beta_{G/B,M_3}}{\twoheadrightarrow}\pi_6^\vee$$ induces an isomorphism in the quotient category $\overline{\SP}_h$ after taking Pontryagin duals for each $j=3,4$ by Theorem \ref{DveeDeltagrw_0} and Lemma \ref{KerDveeDelta}. In particular, we obtain a morphism $\widetilde{f}\colon Q(\pi_2)\twoheadrightarrow Q(\pi_4)\cong Q(\pi_3)\hookrightarrow Q(\pi_1)$ such that $D^\vee_\Delta(\widetilde{f})=f$. This proves $D^\vee_\Delta$ is full as a functor on $\overline{\SP}_h$.
\end{proof}

\begin{cor}\label{Qisom}
Let $\pi_1,\pi_2$ be two objects in $\SP_h$ such that $D^\vee_\Delta(\pi_1)\cong D^\vee_\Delta(\pi_2)$. Then we have $Q(\pi_1)\cong Q(\pi_2)$ in the quotient category $\overline{\SP}_h$.
\end{cor}
\begin{proof}
This is an immediate consequence of Theorem \ref{fullyfaithful}.
\end{proof}

\begin{rem}
Corollary \ref{Qisom} gives a new proof of \cite[Lemma 4.24]{MultVar} not relying on results of Hauseux \cite{Hauseux1,Hauseux}. Indeed, let $\chi,\chi'\colon T\to \Fq^\times$ be two not necessarily distinct characters such that both principal series representations $\Ind_{\overline{B}}^G\chi$ and $\Ind_{\overline{B}}^G\chi'$ are irreducible. Then the natural map $$\Ext^1(\Ind_{\overline{B}}^G\chi',\Ind_{\overline{B}}^G\chi)\to \Ext^1(D^\vee_\Delta(\Ind_{\overline{B}}\chi),D^\vee_\Delta(\Ind_{\overline{B}}\chi'))$$ is injective.
\end{rem}
\begin{proof}
Assume we have a nonsplit extension $0\to \Ind_{\overline{B}}^G\chi\to \pi_1\to \Ind_{\overline{B}}^G\chi'\to 0$ such that $D^\vee_\Delta(\pi_1)\cong D^\vee_\Delta(\pi_2)$ where we put $\pi_2:=\Ind_{\overline{B}}^G\chi\oplus \Ind_{\overline{B}}^G\chi'$. Then by the proof of Theorem \ref{fullyfaithful} we have another object $\pi_3$ in $\SP_{\Fq}$ and morphisms $f_j\colon \pi_3\to\pi_j$ ($j=1,2$) that become isomorphisms in the quotient category $\overline{\SP}_{\Fq}$. Since $\pi_1$ and $\pi_2$ have no subquotients in the category $\Fin^{>B}_{\Fq}(G)$, we must have $\Ker(f_1)=\Ker(f_2)$ by Lemma \ref{KerDveeDelta} whence $\pi_1\cong \pi_2$, contradiction.
\end{proof}

\subsection{Characterizing parabolically induced mod $p$ representations}\label{sec:charparind}

Let $\overline{P}=L\overline{U}\leq G$ be a parabolic subgroup and let $\pi$ be an object in $\Rep^{\mathrm{sm}}_h(G)$. There are (at least) $2$ possibilities to generalize Colmez' Jacquet-functor $J$ \cite[sections IV.3, VII.1]{Colmez1}. Putting $J_{\overline{P}}(\pi):=\pi_{\overline{U}}=((\pi^\vee)^{\overline{U}})^\vee$ we obtain a functor $J_{\overline{P}}\colon \Rep^{\mathrm{sm}}_h(G)\to \Rep^{\mathrm{sm}}_h(L)$. On the other hand, we also propose the following construction. Assume that $\pi$ has finite length. By Corollary \ref{finitenessofDveeDelta}, $D^\vee_\Delta(\pi)$ is finitely generated over $E_\Delta$, so we have an object $M\in \mathcal{M}_\Delta(\pi^{H_{\Delta,0}})$ such that $D^\vee_\Delta(\pi)=M^\vee[X_\Delta^{-1}]$. So we have a $G$-equivariant continuous map $\beta_{M,G/B}\colon \pi^\vee\to \mathfrak{Y}_{M,\pi}(G/B)$. Assume for simplicity that $\overline{P}$ contains $\overline{B}$ (there is always such a conjugate of $\overline{P}$). Put $\overline{B}_L:=L\cap \overline{B}$. We have the embedding $$\iota_{\overline{P}}\colon L/\overline{B}_L\cong \overline{P}/\overline{B}\hookrightarrow G/\overline{B}\cong G/B$$
where the last identification is given by $g\overline{B}\mapsto gw_0B$.

For an open subset $U\subseteq L/\overline{B}_L$ there exists an open subset $V\subseteq G/B$ such that $U=\iota_{\overline{P}}^{-1}(V)$ since $\iota_{\overline{P}}$ is a closed embedding. We define $$\mathfrak{Y}^{\overline{P}}_{\pi,M}(U):=\Ker\left(\res^V_{V\setminus\iota_{\overline{P}}(U)}\colon \mathfrak{Y}_{\pi,M}(V)\to \mathfrak{Y}_{\pi,M}(V\setminus\iota_{\overline{P}}(U))\right)\ .$$ 
\begin{lem}
The abelian group $\mathfrak{Y}^{\overline{P}}_{\pi,M}(U)$ does not depend on the choice of $V$. Moreover, the assignment $\mathfrak{Y}^{\overline{P}}_{\pi,M}$ is a $\overline{P}$-equivariant sheaf on $L/\overline{B}_L\cong \overline{P}/\overline{B}$.
\end{lem}
\begin{proof}
As in the proof of Lemma \ref{kersheaf}, $\mathfrak{Y}^{\overline{P}}_{\pi,M}$ is the kernel of a morphism from the ($G$-equivariant hence) $\overline{P}$-equivariant sheaf $\mathfrak{Y}_{\pi,M}$ to $\overline{\iota_{\overline{P}}}_*\overline{\iota_{\overline{P}}}^{-1}\mathfrak{Y}_{\pi,M}$ where $\overline{\iota_{\overline{P}}}$ is the inclusion map of the open subset $(G/B)\setminus (\overline{P}w_0B/B)$ in $G/B$. This sheaf is supported on $\overline{P}w_0B/B$ which is homeomorphic to $L/\overline{B}_L$.
\end{proof}
We put
\begin{align}\label{jacquet}
J_{\overline{P}}^{\mathrm{sheaf}}(\pi)^\vee:=\beta_{M,G/B}(\pi^\vee)\cap \Ker\left(\res^{G/B}_{(G/B) \setminus (\overline{P}w_0B/B)}\colon \mathfrak{Y}_{\pi,M}(G/B)\to \mathfrak{Y}_{\pi,M}( (G/B) \setminus (\overline{P}w_0B/B) ))\right)\ .
\end{align}
Since $\res^{G/B}_{(G/B) \setminus (\overline{P}w_0B/B)}$ is $\overline{P}$-equivariant, so is the quotient map $\pi\twoheadrightarrow J_{\overline{P}}^{\mathrm{sheaf}}(\pi)$. In particular, $J_{\overline{P}}^{\mathrm{sheaf}}(\pi)$ is a smooth representation of $\overline{P}$ and also of $L$ via the inclusion $L\hookrightarrow \overline{P}$. By construction, we have a $\overline{P}$-equivariant inclusion $J_{\overline{P}}^{\mathrm{sheaf}}(\pi)^\vee\hookrightarrow \mathfrak{Y}^{\overline{P}}_{\pi,M}(L/\overline{B}_L)$.

The following results are generalizations of \cite[Lemme III.1.4]{Colmez1} and \cite[Prop.\ IV.3.2(ii)]{Colmez1} that we will need later.

\begin{pro}\label{uniptriv}
Let $P=LU$ be a parabolic subgroup of $G=\GL_n(\Qp)$ and let $\pi$ be a smooth representation of $P$ over $\mathcal{O}_F/\varpi^h$ that is admissible and has finite length as a representation of the Levi component $L$ (via restriction to $L\hookrightarrow P$). Then the unipotent radical $U$ acts trivially on $\pi$.
\end{pro}
\begin{proof}
We may assume without loss of generality that $P\leq G$ is a standard parabolic subgroup, ie.\ $B\subseteq P$ and $L\cong \GL_{n_1}(\Qp)\times\dots\times \GL_{n_r}(\Qp)$ with $n_1+\dots+n_r=n$. The center $Z(L)$ acts on $\pi$ via a finite quotient by \cite[Lemma 2.3.6]{EmertonOrd1}, so by passing to a finite extension of $\Fq$ we may assume that $Z(L)$ acts via a character $\xi\colon Z(L)\to\Fq^\times$. Let $x\in\pi$ be arbitrary and consider the element $$t_L:=\diag(\underbrace{p^{r-1},\dots,p^{r-1}}_{n_1},\dots,\underbrace{p^{r-j},\dots,p^{r-j}}_{n_j},\dots,\underbrace{1,\dots,1}_{n_r})\in Z(L)\ .$$ Since $\pi$ is smooth, $x$ is fixed by some open subgroup $U_1\leq U$. However, we have $\bigcup_{k\geq 0}t_L^{-k}U_1t_L^k=U$ by the construction of $t_L$, so we compute $$u\cdot x=t_L^{-k}u_1t_L^k\cdot x=\xi(t_L^k)t_L^{-k}\cdot (u_1\cdot x)=\xi(t_L^k)t_L^{-k}\cdot x=x$$
where $u\in U$ is arbitrary and $k\gg 0$ is large enough so that $u_1:=t_L^kut_L^{-k}$ lies in $U_1$. We obtain that $x$ is fixed by $U$ as claimed.
\end{proof}

\begin{pro}\label{Uinv0onopen}
Let $P=LU$ be a parabolic subgroup of $G=\GL_n(\Qp)$ containing $B$ and let $\pi$ be a smooth representation of $G$ and $M\in \mathcal{M}_\Delta(\pi^{H_{\Delta,0}})$. Then we have $\mathfrak{Y}_{\pi,M}( (G/B) \setminus (P/B) )^{U}=0$.
\end{pro}
\begin{proof}
By the Bruhat decomposition the subsets $Nww_0B/B\subset G/B$ cover $(G/B) \setminus (P/B)$ where $ww_0$ runs through the set $N_G(T)\setminus N_P(T)$. So if there is a nonzero $U$-invariant section in $x\in\mathfrak{Y}_{\pi,M}( (G/B) \setminus (P/B) )$, then there is an element $w\in N_G(T)$ with $ww_0\notin N_P(T)$ such that the restriction of $x$ to $Nww_0B/B$ is nonzero (and $U$-invariant). Hence the restriction of $w^{-1}x$ to $w^{-1}Nww_0B/B=(N\cap w^{-1}Nw)w_0B/B\subset Nw_0B/B=\mathcal{C}$ is also nonzero and $w^{-1}Uw$-invariant. By the condition that $ww_0\notin N_P(T)$, there is a positive root $\beta\in\Phi^+$ such that the root subgroup $N^{(\beta)}$ is contained in $N\cap w^{-1}Uw$ showing $w^{-1}x$ is $N^{(\beta)}$-invariant. Since $\mathcal{C}=\bigcup_{t\in T}t\mathcal{C}_0$, we obtain that the restriction of $tw^{-1}x$ to $\mathcal{C}_0\cap (N\cap w^{-1}Nw)w_0B/B$ is nonzero and $N^{(\beta)}$-invariant for some $t\in T$. This contradicts to the fact that $\mathfrak{Y}_{\pi,M}(\mathcal{C}_0)\subset \mathbb{M}_{\infty,0}(M^\vee[X_\Delta^{-1}])$ has no nonzero $\Fq\bs N_0\js$-torsion.
\end{proof}

\begin{rem}
The precise relationship between the functors $J_{\overline{P}}$ and $J_{\overline{P}}^{\mathrm{sheaf}}$ is an important open question. In the remaining of the section we partially answer this question in case $L$ is the direct product of a torus and a subgroup isomorphic to $\GL_2(\Qp)$.
\end{rem}

For a simple root $\alpha\in\Delta$ let $L^{(\alpha)}$ be the Zariski closed subgroup of $G$ generated by $N_\alpha,N_{-\alpha}$, and $T$. By \cite[Lemma 3.1.4]{BH1} we have $L^{(\alpha)}= \GL_2(\Qp)^{(\alpha)}\times T^{(\alpha)}$ for some subgroups $T^{(\alpha)}\leq T$ and $\GL_2(\Qp)\cong \GL_2(\Qp)^{(\alpha)}\leq L^{(\alpha)}$. The parabolic subgroup $\overline{P}^{(\alpha)}:=L^{(\alpha)}\overline{N}=L^{(\alpha)}\overline{U}^{(\alpha)}$ has Levi component $L^{(\alpha)}$ and unipotent radical $\overline{U}^{(\alpha)}=\prod_{\beta\in\Phi^+\setminus\{\alpha\}}N_{-\beta}$. Recall we denote the reflection corresponding to the simple root $\alpha\in\Delta$ by $w_\alpha\in N_G(T)/T$ so we have $\overline{P}^{(\alpha)}=\overline{B}\cup \overline{B}w_\alpha \overline{B}$. Hence we also obtain $w_0\overline{P}^{(\alpha)}w_0^{-1}=B\cup Bw_{\alpha'}B=P^{(\alpha')}$ where $\alpha':=-w_0(\alpha)\in\Delta$ with $w_{\alpha'}=w_0w_\alpha w_0^{-1}$. The following result characterizes representations in $\SP_{\Fq}$ that are in the image of parabolic induction from $\overline{P}^{(\alpha)}$.

\begin{thm}\label{lookslikeparindisparind}
Let $\pi$ be an object in $\SP_{\Fq}$. There is a smooth representation $\pi_\alpha$ of $\GL_2(\mathbb{Q}_p)^{(\alpha)}\times T^{(\alpha)}$ such that the image $Q(\pi)$ in $\overline{\SP}_{\mathbb{F}_q}$ is isomorphic to $Q(\Ind_{\overline{P}^{(\alpha)}}^{\GL_n(\mathbb{Q}_p)}\pi_\alpha)$ if and only if $\pi$ lies in the subcategory $\SP_{\mathbb{F}_q}^{\alpha,ab}$.
\end{thm}
\begin{thm}\label{irredjacquet}
Let $\pi$ be a smooth irreducible representation of $G$ and assume $D^\vee_\Delta(\pi)\neq 0$. Then there exists a smooth irreducible admissible representation $\pi_\alpha$ of $\GL_2(\mathbb{Q}_p)^{(\alpha)}\times T^{(\alpha)}$ such that $\pi$ is isomorphic to a subquotient of $\Ind_{\overline{P}^{(\alpha)}}^{\GL_n(\mathbb{Q}_p)}\pi_\alpha$ if and only if $\GQpDa$ acts on $\mathbb{V}^\vee\circ D^\vee_\Delta(\pi)$ via a character.
\end{thm}
We prove Theorems \ref{lookslikeparindisparind} and \ref{irredjacquet} parallel to each other. The difficulty lies in the ``if part'' in both cases as the ``only if'' part essentially follows from \cite[Thm.\ B]{MultVar}. In order to proceed we need some preliminaries. Put $\overline{B}_2^{(\alpha)}:=\GL_2(\Qp)^{(\alpha)}\cap \overline{B}$ and $\iota_\alpha:=\iota_{\overline{P}^{(\alpha)}}\colon \GL_2(\Qp)^{(\alpha)}/\overline{B}_2^{(\alpha)}\hookrightarrow G/\overline{B}\overset{\sim}{\to}G/B$. By construction, the image of $\iota_\alpha$ is $\overline{P}^{(\alpha)}w_0B/B$. 

Let $\pi$ be either an object in the category $\SP_{\mathbb{F}_q}^{\alpha,ab}$ or an irreducible smooth representation of $G$ such that $D^\vee_\Delta(\pi)\neq 0$ and $\GQpDa$ acts on $\mathbb{V}^\vee\circ D^\vee_\Delta(\pi)$ via a character. In either case let $M\in\mathcal{M}_\Delta^0(\pi^{H_{\Delta,0}})$ be the object with $M^\vee[X_\Delta^{-1}]=D^\vee_\Delta(\pi)=:D$. By assumption (and the definition of the category $\SP_{\mathbb{F}_q}^{\alpha,ab}$) the action of $\GQpDa\times\Qp^\times$ on the attached representation $\mathbb{V}^\vee(D)$ factors through the maximal abelian quotient $\GQpDa^{ab}\times\Qp^\times$. Hence by Proposition \ref{Dnrmultvar} we obtain $D=E_\Delta\otimes_{E_{\{\alpha\}}}D^{\nr}_{\Delta\setminus\{\alpha\}}$. We have a $G$-equivariant map $\beta_{M,G/B}\colon \pi^\vee\to \mathfrak{Y}_{\pi,M}(G/B)$ (cf.\ \cite[Thm.\ F]{MultVar}). In our second case when $\pi$ is irreducible with $D^\vee_\Delta(\pi)\neq 0$, the map $\beta_{M,G/B}$ is injective. In the other case when $\pi$ is in the category $\SP_{\mathbb{F}_q}^{\alpha,ab}$, the Pontryagin dual of the kernel of $\beta_{M,G/B}$ lies in the category $\Fin^{>B}_{\Fq}(G)$ by Theorem \ref{DveeDeltagrw_0}, Lemma \ref{KerDveeDelta}, and \cite[Thm.\ C]{MultVar}. 

We put $\pi_\alpha:=J_{\overline{P}^{(\alpha)}}^{\mathrm{sheaf}}(\pi)$ (see \eqref{jacquet}) and let $M_0$ be the Pontryagin dual of the image of $\pi_\alpha^\vee$ under the composite map $$\pi_\alpha^\vee\hookrightarrow \mathfrak{Y}^\alpha_{\pi,M}(\GL_2(\Qp)^{(\alpha)}/\overline{B}_2^{(\alpha)})\overset{\res}{\to} \mathfrak{Y}^\alpha_{\pi,M}(N_0^{(\alpha)}\overline{B}_2^{(\alpha)}/\overline{B}_2^{(\alpha)})$$ where $N_0^{(\alpha)}:=N_0\cap \GL_2(\Qp)^{(\alpha)}=N_0\cap N_\alpha$. 

\begin{pro}\label{gl2admissprinc}
The representation $\pi_\alpha$ of $\GL_2(\Qp)^{(\alpha)}$ is admissible and has finite length. The value of Colmez' Montréal functor at $\pi_\alpha$ is $D^\vee_{\{\alpha\}}(\pi_\alpha)\cong M_0^\vee[X_\alpha^{-1}]$ which admits an additional linear action of the direct product of $T^{(\alpha)}$ and of the center of $\GL_2(\Qp)^{(\alpha)}$. Moreover, in case $\pi$ is in the category $\SP_{\mathbb{F}_q}^{\alpha,ab}$, the Jordan--Hölder factors of $\pi_\alpha$ are isomorphic to subquotients of principal series.
\end{pro}
\begin{proof}
Since the diagram 
\begin{align*}
\xymatrixcolsep{5pc}\xymatrix{ N_0^{(\alpha)}\overline{B}_2^{(\alpha)}/\overline{B}_2^{(\alpha)} \ar[r]^-{u\overline{B}_2^{(\alpha)}\mapsto uw_0B} \ar@{^{(}->}[d] & \mathcal{C}_0=N_0w_0B/B \ar@{^{(}->}[d] \\ \GL_2(\Qp)^{(\alpha)}/\overline{B}_2^{(\alpha)} \ar[r]^-{\iota_\alpha} &G/B}
\end{align*}
commutes, we may identify  
\begin{align*}
M_0^\vee\cong \beta_{M,\mathcal{C}_0}(\pi^\vee)\cap \Ker\left(\res^{\mathcal{C}_0}_{\mathcal{C}_0\setminus (N_0^{(\alpha)}w_0B/B)}\colon \mathfrak{Y}_{\pi,M}(\mathcal{C}_0)\to \mathfrak{Y}_{\pi,M}(\mathcal{C}_0\setminus (N_0^{(\alpha)}w_0B/B))\right)=M_\infty^\vee\cap D_{w_{\alpha'}}
\end{align*}
since we have $$U_{w_{\alpha'}}=\mathcal{C}_0\setminus w_0^{-1}w_{\alpha'}^{-1}\mathcal{C}^{w_{\alpha'}}=\mathcal{C}_0\setminus w_0^{-1}w_{\alpha'}^{-1}Nw_{\alpha'}B=\mathcal{C}_0\setminus w_0^{-1}w_{\alpha'}^{-1}N^{(\alpha')}w_{\alpha'}B=\mathcal{C}_0\setminus N_0^{(\alpha)}w_0B\ .$$ Now $w_{\alpha'}$ is a simple reflection, so we compute 
$$N_{0,w_0^{-1}w_{\alpha'}^{-1}}=N_0\cap w_0^{-1}w_{\alpha'}^{-1}N_0w_{\alpha'}w_0= N_0\cap w_\alpha^{-1}\overline{N}w_\alpha= N_0^{(\alpha)}\ .$$
In particular, the ring $\Fq\bg N_{\Delta,\infty,w_0^{-1}w_{\alpha'}^{-1}}\jg$ is isomorphic to the $1$-variable Laurent series field $\Fq\bg X_\alpha\jg$, ie.\ the field of fractions of $\Fq\bs N_0^{(\alpha)}\js\cong\Fq\bs X_\alpha\js$. Hence by Proposition \ref{Dwprojlim}, we may identify $D_{w_{\alpha'}}$ with $D^{\nr}_{\Delta\setminus\{\alpha\}}$. Therefore $M_0^\vee$ is a finitely generated module over $\Fq\bs N_0^{(\alpha)}\js$ by Proposition \ref{MbdcapDwfin}. Let $U^{(1,\alpha)}\leq \GL_2(\Zp)^{(\alpha)}\leq \GL_2(\Qp)^{(\alpha)}$ be the kernel of the projection map $\GL_2(\Zp)^{(\alpha)}\to \GL_2(\Fp)^{(\alpha)}$. Note that the open subset $N_0^{(\alpha)}\overline{B}_2^{(\alpha)}/\overline{B}_2^{(\alpha)}\subset \GL_2(\Qp)^{(\alpha)}/\overline{B}_2^{(\alpha)}$ is $U^{(1,\alpha)}$-invariant, so the projection map $\pi_\alpha^\vee\twoheadrightarrow M_0^\vee$ is $\Fq\bs U^{(1,\alpha)}\js$-linear. Since the index of $U^{(1,\alpha)}\cap N_0^{(\alpha)}$ in $N_0^{(\alpha)}$ is finite, $M_0^\vee$ is also finitely generated as a module over $\Fq\bs U^{(1,\alpha)}\js$. Moreover, $U^{(1,\alpha)}$ is normalized by the Weyl element $w_\alpha$, so $(w_\alpha M_0)^\vee$ is also finitely generated over $\Fq\bs U^{(1,\alpha)}\js$. Finally, since we have a covering $$ \GL_2(\Qp)^{(\alpha)}/\overline{B}_2^{(\alpha)}\cong \mathbb{P}^1(\Qp)=\Zp\cup\frac{1}{\Zp}\cong N_0^{(\alpha)}\overline{B}_2^{(\alpha)}/\overline{B}_2^{(\alpha)}\cup w_\alpha N_0^{(\alpha)}\overline{B}_2^{(\alpha)}/\overline{B}_2^{(\alpha)}\ ,$$ the map $\pi_\alpha^\vee\to M_0^\vee\oplus (w_\alpha M_0)^\vee$ is injective. The admissibility of $\pi_\alpha$ follows as $\Fq\bs U^{(1,\alpha)}\js$ is noetherian.

By Theorem \ref{DveeDeltagrw_0} applied to $G=\GL_2(\Qp)^{(\alpha)}$ (ie.\ $n=2$) we obtain $D^\vee_{\{\alpha\}}(\pi_\alpha)\cong M_0^\vee[X_\alpha^{-1}]$. Since $\pi_\alpha$ is also a representation of $L^{(\alpha)}$, $D^\vee_{\{\alpha\}}(\pi_\alpha)$ admits an action of the center $Z(L^{(\alpha)})=T^{(\alpha)}\times Z(\GL_2(\Qp)^{(\alpha)})$. Note that in case of $\GL_2(\Qp)$ both $D^\vee_\Delta=D^\vee_{\{\alpha\}}$ and Breuil's functor $D^\vee_\xi$ give back Colmez' \cite{Colmez1} Montréal functor \cite[Prop.\ 3.2]{Breuil}.

We obtain $\pi_\alpha=M_0+w_\alpha M_0$ using again the injectivity of the map $\pi_\alpha^\vee\to M_0^\vee\oplus(w_\alpha M_0)^\vee$. Therefore $\pi_\alpha$ is a finitely generated representation of $\GL_2(\Qp)^{(\alpha)}$ since $M_0$ is finitely generated as a representation of the monoid $B_{2,+}^{(\alpha)}$ (being an object in $\mathcal{M}_{\{\alpha\}}(\pi_\alpha)$). It follows that $\pi_\alpha$ has finite length using the admissibility and \cite[Thm.\ 2.3.8]{EmertonOrd1}.

Finally, assume that $\pi$ is in the category $\SP_{\mathbb{F}_q}^{\alpha,ab}$. Since Colmez' Montréal functor is exact \cite[Thm.\ 0.9]{Colmez1} and produces $2$-dimensional irreducible Galois representations at supersingular smooth mod $p$ representations of $\GL_2(\Qp)$ \cite[Thm.\ 0.10]{Colmez1}, $\pi_\alpha$ is a successive extension of subquotients of principal series as $D_{w_{\alpha'}}$ is an extension of rank $1$ objects (since so is $D$ because $\pi$ is in the category $\SP_{\Fq}$).
\end{proof}

\begin{cor}\label{Uactstrivially}
The representation $\pi_\alpha$ of $\overline{P}^{(\alpha)}$ factors through the quotient $\overline{P}^{(\alpha)}\twoheadrightarrow L^{(\alpha)}$.
\end{cor}
\begin{proof}
This follows immediately combining Propositions \ref{uniptriv} and \ref{gl2admissprinc}.
\end{proof}

\begin{pro}\label{asbigasitcanbe}
We have $D^\vee_{\{\alpha\}}(\pi_\alpha)\cong M_0^\vee[X_\alpha^{-1}]\cong D_{w_{\alpha'}}\cong D^\vee_\Delta(\pi)^{\nr}_{\Delta\setminus\{\alpha\}}$.
\end{pro}
\begin{proof}
By Proposition \ref{gl2admissprinc} it remains to show that the containment $M_0^\vee[X_\alpha^{-1}]\leq D_{w_{\alpha'}}$ is an equality. Recall \cite[Cor.\ II.5.12]{Colmez2} that there is a smallest $\Fq\bs N_0^{(\alpha)}\js$-lattice $D^\natural_{w_{\alpha'}}$ in $D_{w_{\alpha'}}$ stable under $\psi_t$ for all $t\in T_+$. Let $D^\natural$ be the $E_\Delta^+$-submodule in $D$ generated by $D^\natural_{w_{\alpha'}}$ under the identification $D_{w_{\alpha'}}\cong D^\nr_{\Delta\setminus\{\alpha\}}\subset D=D^\vee_\Delta(\pi)=E_\Delta\otimes_{E_{\alpha}}D^\nr_{\Delta\setminus\{\alpha\}}$.
\begin{lem}\label{Dnaturalsmallest}
$D^\natural$ is the smallest $E_\Delta^+$-lattice in $D$ stable under the operators $\psi_t\colon D\to D$ for all $t\in T_+$.
\end{lem}
\begin{proof}
Let $D_0\leq D$ be an $E_\Delta^+=\Fq\bs N_{\Delta,0}\js$-submodule stable under $\psi_t$ for all $t\in T_+$ such that $D_0[X_\Delta^{-1}]=D$. Pick an element $x\in D_{w_{\alpha'}}\subset D$. So there is an integer $k\geq 0$ such that $X_\Delta^kx\in D_0$. Applying $\psi_{\Delta\setminus\{\alpha\}}^r$ for some integer $r>\log_p k$ we obtain $X_\alpha^kx\in D_0$ since $\psi_{\Delta\setminus\{\alpha\}}$ commutes with $X_\alpha$, acts via a character on $x$, and satisfies $\psi_{\Delta\setminus\{\alpha\}}^r(X_{\Delta\setminus\{\alpha\}})=\pm 1$. This shows that $D_0\cap D_{w_{\alpha'}}$ is an $\Fq\bs N_0^{(\alpha)}\js$-lattice in $D_{w_{\alpha'}}$ stable under $\psi_t$ for all $t\in T_+$ whence we obtain $D_0\cap D_{w_{\alpha'}}\supseteq D^\natural_{w_{\alpha'}}$ by the minimality of $D^\natural_{w_{\alpha'}}$. Since $D_0$ is an $E_\Delta^+$-submodule of $D$, it contains $D^\natural$ as claimed.
\end{proof}
Let $\mathfrak{M}^\natural$ be the $\Fq\bs N_0\js$-submodule in $D$ generated by $D^\natural_{w_{\alpha'}}$ so that we have $\mathfrak{M}^\natural\cap D_{w_{\alpha'}}=D^\natural_{w_{\alpha'}}$ and $\pr_{0,\infty}(\mathfrak{M}^\natural)=D^\natural$.
\begin{lem}\label{Mnaturalfree}
$\mathfrak{M}^\natural$ (resp.\ $D^\natural$) is a free module of rank $\rk_{E_\Delta}(D)$ over $\Fq\bs N_0\js$ (resp.\ over $\Fq\bs N_{\Delta,0}\js\cong E_{\Delta}^+$). In particular, the quotient map $\Fq\bs N_{\Delta,0}\js\otimes_{\Fq\bs N_0\js}\mathfrak{M}^\natural\twoheadrightarrow D^\natural$ is an isomorphism.
\end{lem}
\begin{proof}
By Proposition \ref{Dnrmultvar} we get $E_\Delta\otimes_{E_{\{\alpha\}}}D_{w_{\alpha'}}\cong D$. Since there is a section $E_{\{\alpha\}}\hookrightarrow \Fq\bg N_{\Delta,\infty}\jg$, we also obtain the isomorphism $\Fq\bg N_{\Delta,\infty}\jg\otimes_{E_{\{\alpha\}}}D_{w_{\alpha'}}$. On the other hand, $\Fq\bs N_0^{(\alpha)}\js\cong \Fq\bs X_\alpha\js$ is a DVR, so $D^\natural_{w_{\alpha'}}$ is a free $\Fq\bs N_0^{(\alpha)}\js$-module of rank $\rk_{E_{\{\alpha\}}}(D_{w_{\alpha'}})=\rk_{E_\Delta}(D)$ with $E_{\{\alpha\}}\otimes_{\Fq\bs N_0^{(\alpha)}\js}D^\natural_{w_{\alpha'}}\cong D_{w_{\alpha'}}$ by the theorem of elementary divisors. We deduce $\mathfrak{M}^\natural\cong \Fq\bs N_0\js\otimes_{\Fq\bs N_0^{(\alpha)}\js}D^\natural_{w_{\alpha'}}$ (resp.\ $D^\natural\cong \Fq\bs N_{\Delta,0}\js\otimes_{\Fq\bs N_0^{(\alpha)}\js}D^\natural_{w_{\alpha'}}$) is a free $\Fq\bs N_0\js$-module (resp.\ $\Fq\bs N_{\Delta,0}\js$-module) of the same rank.
\end{proof}
The operators $\psi_t$ ($t\in T_+$) act surjectively on $D^\natural_{w_{\alpha'}}$ \cite[Cor.\ II.5.12]{Colmez2} hence also on $\mathfrak{M}^\natural$. So each of the three compact $\Fq\bs N_0\js$-submodules $\mathfrak{M}^\natural,M_\infty^\vee$, and $\mathfrak{M}^\natural+M_\infty^\vee$ are contained in $\mathbb{M}^{bd}_\infty(D^\#)$ by Proposition \ref{prcontainedinDhash}. So by Proposition \ref{kerintersect}$(b)$ we may apply Proposition \ref{fingen+kertors} to each $\mathfrak{M}^\natural,M_\infty^\vee$, and $\mathfrak{M}^\natural+M_\infty^\vee$ to deduce isomorphisms
\begin{align*}
\Fq\bg N_{\Delta,0}\jg\otimes_{\Fq\bs N_0\js}\mathfrak{M}^\natural& \cong D\\
\Fq\bg N_{\Delta,0}\jg\otimes_{\Fq\bs N_0\js}M_\infty^\vee& \cong D\\
\Fq\bg N_{\Delta,0}\jg\otimes_{\Fq\bs N_0\js}(\mathfrak{M}^\natural+M_\infty^\vee)& \cong D\ .
\end{align*}
Applying $\Fq\bg N_{\Delta,0}\jg\otimes_{\Fq\bs N_0\js}\cdot$ to the short exact sequence $$0\to \mathfrak{M}^\natural\cap M_\infty^\vee\to \mathfrak{M}^\natural\oplus M_\infty^\vee\to \mathfrak{M}^\natural+M_\infty^\vee\to 0$$ we obtain that the map $$\Fq\bg N_{\Delta,0}\jg\otimes_{\Fq\bs N_0\js}(\mathfrak{M}^\natural\cap M_\infty^\vee)\twoheadrightarrow \Fq\bg N_{\Delta,0}\jg\otimes_{\Fq\bs N_{\Delta,0}\js}\pr_{0,\infty}(\mathfrak{M}^\natural\cap M_\infty^\vee) \to D$$ is onto. Hence we compute $D^\natural\subseteq \pr_{0,\infty}(\mathfrak{M}^\natural\cap M_\infty^\vee)\subseteq \pr_{0,\infty}(\mathfrak{M}^\natural)=D^\natural$ by Lemma \ref{Dnaturalsmallest}. Since the kernel of the ring homomorphism $\Fq\bs N_0\js \twoheadrightarrow \Fq\bs N_{\Delta,0}\js$ is contained in the Jacobson radical $\Ker(\Fq\bs N_0\js\twoheadrightarrow \Fq)$ of $\Fq\bs N_0\js$, so we deduce $\mathfrak{M}^\natural\cap M_\infty^\vee=\mathfrak{M}^\natural$ from Lemma \ref{Mnaturalfree} using Nakayama's lemma. Finally, we obtain $M_0^\vee=M_\infty^\vee\cap D_{w_{\alpha'}}\supseteq \mathfrak{M}^\natural\cap D_{w_{\alpha'}}=D^\natural_{w_{\alpha'}}$ whence $M_0^\vee[X_\alpha^{-1}]=D_{w_{\alpha'}}$ as claimed.
\end{proof}
\begin{proof}[Proof of Theorem \ref{lookslikeparindisparind}]
Assume first there is a smooth representation $\pi_\alpha$ of $L^{(\alpha)}$ such that $Q(\pi)\cong Q(\Ind_{\overline{P}^{(\alpha)}}^G\pi_\alpha)$. Then by Lemma \ref{KerDveeDelta} and \cite[Thm.\ B]{MultVar} we have $D^\vee_\Delta(\pi)\cong D^\vee_\Delta(\Ind_{\overline{P}^{(\alpha)}}^G\pi_\alpha)\cong E_\Delta\otimes_{E_{\{\alpha\}}}D^\vee_{\{\alpha\}}(\pi_\alpha)$, so $\pi$ lies in the subcategory $\SP_{\Fq}^{\alpha,ab}$.

For the other direction assume $\pi$ is an object in $\SP_{\Fq}^{\alpha,ab}$. We use \cite[Thm.\ B]{MultVar}, Proposition \ref{Dnrmultvar}, and Proposition \ref{asbigasitcanbe} to compute
\begin{align*}
D^\vee_\Delta(\Ind_{\overline{P}^{(\alpha)}}^G\pi_\alpha)\cong E_\Delta\otimes_{E_{\{\alpha\}}}D^\vee_{\{\alpha\}}(\pi_\alpha)\cong E_\Delta\otimes_{E_{\{\alpha\}}}D^\vee_\Delta(\pi)^{\nr}_{\Delta\setminus\{\alpha\}}\cong D^\vee_\Delta(\pi)\ .
\end{align*}
By Proposition \ref{gl2admissprinc}, both $\Ind_{\overline{P}^{(\alpha)}}^G\pi_\alpha$ and $\pi$ are objects in $\SP_{\Fq}$, so we obtain $Q(\Ind_{\overline{P}^{(\alpha)}}^G\pi_\alpha)\cong Q(\pi)$ using Theorem \ref{fullyfaithful}.
\end{proof}

The following is a generalization of \cite[Prop.\ IV.3.2(ii)]{Colmez1}.

\begin{cor}\label{twojacquetsame}
Assume that $\pi$ lies in the category $\SP_{\Fq}^{\alpha,ab}$ and has no nonzero quotient in the subcategory $\Fin^{>B}_{\Fq}(G)$. Then we have  $J_{\overline{P}^{(\alpha)}}(\pi)\cong J_{\overline{P}^{(\alpha)}}^{\mathrm{sheaf}}(\pi)$ as smooth representations of $L^{(\alpha)}$.
\end{cor}
\begin{proof}
By assumption on the nonexistence of nontrivial quotients in $\Fin^{>B}_{\Fq}(G)$, the map $\beta_{M,G/B}\colon \pi^\vee\to \mathfrak{Y}_{\pi,M}(G/B)$ is injective. By Corollary \ref{Uactstrivially} $\overline{U}^{(\alpha)}$ acts trivially on $J_{\overline{P}^{(\alpha)}}^{\mathrm{sheaf}}(\pi)$, so we have $J_{\overline{P}^{(\alpha)}}^{\mathrm{sheaf}}(\pi)^\vee\subseteq (\pi^\vee)^{U^{(\alpha)}}=J_{\overline{P}^{(\alpha)}}(\pi)^\vee$. The other containment follows from Proposition \ref{Uinv0onopen}. 
\end{proof}

\begin{proof}[Proof of Theorem \ref{irredjacquet}]
Assume first there is a smooth admissible representation $\pi_\alpha$ of $L^{(\alpha)}$ such that $\pi$ is isomorphic to a subquotient of $\Ind_{\overline{P}^{(\alpha)}}^G\pi_\alpha$. We may assume without loss of generality that $\pi_\alpha$ is irreducible. By Lemma \ref{KerDveeDelta} the classification \cite{Herzigclass} of irreducible admissible smooth representations of $G$, the representation $\Ind_{\overline{P}^{(\alpha)}}^G\pi_\alpha$has a unique subquotient on which $D^\vee_\Delta$ does not vanish which must therefore be isomorphic to $\pi$. Hence by \cite[Thm.\ B]{MultVar} we have $D^\vee_\Delta(\pi)\cong D^\vee_\Delta(\Ind_{\overline{P}^{(\alpha)}}^G\pi_\alpha)\cong E_\Delta\otimes_{E_{\{\alpha\}}}D^\vee_{\{\alpha\}}(\pi_\alpha)$. So $\GQpDa$ acts on $\mathbb{V}^\vee\circ D^\vee_\Delta(\pi)$ by a character as claimed.

For the other direction assume that $\pi$ is a smooth irreducible representation of $G$ such that $\GQpDa$ acts on $\mathbb{V}^\vee\circ D^\vee_\Delta(\pi)\neq 0$ by a character. As in the proof of Theorem \ref{lookslikeparindisparind} we use \cite[Thm.\ B]{MultVar}, Proposition \ref{Dnrmultvar}, and Proposition \ref{asbigasitcanbe} to compute
\begin{align*}
D^\vee_\Delta(\Ind_{\overline{P}^{(\alpha)}}^G\pi_\alpha)\cong E_\Delta\otimes_{E_{\{\alpha\}}}D^\vee_{\{\alpha\}}(\pi_\alpha)\cong E_\Delta\otimes_{E_{\{\alpha\}}}D^\vee_\Delta(\pi)^{\nr}_{\Delta\setminus\{\alpha\}}\cong D^\vee_\Delta(\pi)\ .
\end{align*}
Finally, the statement follows from Corollaries \ref{irredtoirred} and \ref{detectsiso}.
\end{proof}

We end this section with noting that in case $n=3$ one can say slightly more since the parabolic subgroups $\overline{P}^{(\alpha)}$ are maximal.

\begin{cor}\label{gl3supercuspidal}
Assume that $\pi$ is a smooth irreducible representation of $\GL_3(\Qp)$ such that $D^\vee_\Delta(\pi)\neq 0$. Then $\pi$ is supercuspidal if and only if the action of $\GQpa$ on $\mathbb{V}^\vee\circ D^\vee_\Delta(\pi)$ does \emph{not} factor through the maximal abelian quotient $\GQpa^{ab}$ for any $\alpha\in\Delta$.
\end{cor}
\begin{proof}
This follows immediately from Theorem \ref{irredjacquet} noting that any maximal parabolic subgroup of $\GL_3(\Qp)$ containing $\overline{B}$ is of the form $\overline{P}^{(\alpha)}$ for one of the two simple roots $\alpha\in\Delta$.
\end{proof}

\begin{rem}
We expect Corollary \ref{gl3supercuspidal} to hold in general. The striking question is, of course, on which smooth irreducible admissible supercuspidal representations vanishes the functor $D^\vee_\Delta$. 
\end{rem}

\subsection{Extensions of characters on the Galois side}

We start this subsection by the following general group-theoretic lemma having its origin in Schur's lemma.

\begin{lem}\label{one_splits_lemma}
Let $G_1$ and $G_2$ be groups and assume that $V$ is a $2$-dimensional representation of $G_1\times G_2$ over $\mathbb{F}_q$ admitting a short exact sequence 
$$0 \to \eta'\to V\to \eta\to 0$$
for some characters $\eta, \eta'\colon G_1\times G_2\to \mathbb{F}_q^\times$. If $V$ does not split as a representation of $G_1$ then we have $\eta_{\mid G_2}=\eta'_{\mid G_2}$. In particular, at least one of the restrictions of $V$ to $G_1$ and $G_2$ splits if $\eta\neq\eta'$.
\end{lem}
\begin{proof}
Assume $V$ does not split as a representation of $G_1$. Then there is an element $g\in G_1$ with image in $\GL(V)$ of order $p$ as $\GL(V)$ is finite and finite groups of order coprime to $p$ only have semisimple representations in characteristic $p$. Since $g$ commutes with any element in $G_2$, we deduce that $g-1$ is a nonzero endomorphism of the restriction $V_{\mid G_2}$ of $V$ to $G_2$. Since $(g-1)^pV=(g^p-1)V=0$, we have $(g-1)^2V=0$ as $\dim V=2$. We deduce $\Ker(g-1)=\mathrm{Im}(g-1)$ whence $g-1$ induces an isomorphism between $\eta_{\mid G_2}$ and $\eta'_{\mid G_2}$. Similarly, if $V$ does not split as a representation of $G_2$ then the restrictions $\eta_{\mid G_1}$ and $\eta'_{\mid G_1}$ are equal. The statement follows from the assumption that $\eta\neq \eta'$. 
\end{proof}

\begin{pro}\label{splits_all_but_one_pro}
Let $V$ be a $2$-dimensional continuous representation of $\GQpD$ over $\mathbb{F}_q$ admitting a short exact sequence 
$$0 \to \eta'\to V\to \eta\to 0$$
such that $\eta,\eta'\colon \GQpD\to \mathbb{F}_q^\times$ are continuous characters. If $V$ splits as a representation of $\GQpa$ for some $\alpha\in\Delta$ then we have $\eta_{\mid G_{\Qp,\beta}}=\eta'_{\mid G_{\Qp,\beta}}$ for all $\beta\neq\alpha\in\Delta$. In particular, if $\eta\neq\eta'$ then all but at most one of the restrictions of $V$ to $\GQpa$ ($\alpha\in\Delta$) splits.
\end{pro}
\begin{proof}
This follows immediately from Lemma \ref{one_splits_lemma}.
\end{proof}

\begin{rem}
One can reprove the result of Hauseux \cite{Hauseux1} on the computation of $\Ext^1(\Ind_{\overline{B}}^G\chi',\Ind_{\overline{B}}\chi)$ in the \emph{generic} case when $\chi\neq \chi'$ and both $\Ind_{\overline{B}}^G\chi'$ and $\Ind_{\overline{B}}^G\chi$ are irreducible as follows. Assume we are given a nontrivial extension class in $\Ext^1(\Ind_{\overline{B}}^G\chi',\Ind_{\overline{B}}\chi)$ represented by a smooth representation $\pi$. Then the center $Z(G)$ must act on $\pi$ via a character using Schur's lemma. By the remark after Corollary \ref{Qisom} we apply $D^\vee_\Delta$ to obtain a nonzero class in $\Ext^1(D^\vee_\Delta(\Ind_{\overline{B}}^G\chi),D^\vee_\Delta(\Ind_{\overline{B}}\chi'))$. Using the equivalence of categories $\mathbb{V}$ we get a nontrivial extension $V$ of two distinct characters of $\GQpD\times \Qp^\times$ which splits when restricted to $\GQpa$ for all but exactly one $\alpha\in\Delta$ by Proposition \ref{splits_all_but_one_pro} (since the center $\Qp^\times$ acts via a character). Hence $\pi$ lies in the category $\SP^{\alpha,\xi}_{\Fq}$ for some character $\xi\colon\GQpDa\times\Qp^\times\to \Fq^\times$. By Theorem \ref{lookslikeparindisparind} and our assumption on the irreducibility of both principal series we deduce that $\pi$ can be parabolically induced from $\overline{P}^{(\alpha)}$. Hence the space $\Ext^1(\Ind_{\overline{B}}^G\chi',\Ind_{\overline{B}}\chi)$ is $1$-dimensional by \cite[Prop.\ VII.4.7]{Colmez1}.
\end{rem}

Now let us explicitly describe the value of $\mathbb{V}^\vee\circ D^\vee_\Delta$ at extensions of not necessarily irreducible principal series. Here $\mathbb{V}^\vee$ denotes the composite of the functor $\mathbb{V}\colon \mathcal{D}^{et}(T_+,\mathcal{O}_F/\varpi^h\bg N_{\Delta,0}\jg)\to \Rep_h(\GQpD\times\Qp^\times)$ with $(\cdot)^\vee\colon \Rep_h(\GQpD\times\Qp^\times)\to \Rep_h(\GQpD\times\Qp^\times)$ of taking $F/\mathcal{O}_F$-duals. This way $\mathbb{V}^\vee\circ D^\vee_\Delta$ is a left exact \cite[Thm.\ A]{MultVar} covariant functor from the category $\Rep^{fin}_h(G)$ of smooth representations of $G=\GL_n(\Qp)$ of finite length over $\mathcal{O}_F/\varpi^h$ to $\Rep_h(\GQpD\times\Qp^\times)$ (see Corollary \ref{finitenessofDveeDelta}). Let $\chi_1\otimes\dots\otimes\chi_n\colon T\to \Fq^\times$ be a smooth character. Denote by $\widetilde{\chi_j}\colon G_{\Qp}\to \Fq^\times$ the character corresponding to $\chi_j$ via local class field theory. Then $\mathbb{V}^\vee\circ D^\vee_\Delta(\Ind_{\overline{B}}^G\chi_1\otimes\dots\otimes\chi_n)$ is a $1$-dimensional representation of $\GQpD\times\Qp^\times$ on which the $j$th copy of $G_{\Qp}$ acts via $\widetilde{\chi_1}\cdots\widetilde{\chi_j}$ ($j=1,\dots,n-1$) and $\Qp^\times$ acts via $\chi_1\cdots\chi_n$.

\begin{thm}\label{torabel}
The composite functor $\mathbb{V}^\vee\circ D_\Delta^\vee\circ\Ind_{\overline{B}}^{G}\circ\operatorname{Inf}_T^{\overline{B}}$ is an equivalence of categories between the category of continuous representations of $T$ on finitely generated $\mathcal{O}_F/\varpi^h$-modules to the full subcategory $\Rep_{h}^{ab}(\GQpD\times \Qp^\times)$ of $\Rep_{h}(\GQpD\times\Qp^\times)$ consisting of representations with abelian image.
\end{thm}
\begin{proof}
By the local Kronecker--Weber theorem the maximal abelian quotient of $G_{\Qp}$ is $$G_{\Qp}^{ab}=\Gal(\Qp(\mu_{\infty})/\Qp)\cong \hat{\mathbb{Z}}\times \Zp^\times\ .$$ On the other hand, $\mathcal{O}/\varpi^h$ has finite cardinality, so any continuous representation of $\Qp^\times=p^{\mathbb{Z}}\times \Zp^\times$ extends uniquely to a continuous representation of $p^{\hat{\mathbb{Z}}}\times \Zp^\times\cong G_{\Qp}^{ab}$. Therefore the two categories are indeed equivalent, one only has to check that the equivalence is realized by the given composite functor. The quasi-inverse is given by $V\mapsto \mathbb{D}^{\nr}(V)$ using Proposition \ref{Dnrmultvar} and the identification $T_+\cong \Qp^\times\times\Gamma_\Delta\times \prod_{\alpha\in\Delta}\varphi_\alpha^{\mathbb{N}}$ noting that any continuous action of the monoid $T_+$ on a finite set extends to the group completion $T$ (as finite monoids are in fact groups).
\end{proof}

Let $\varepsilon\colon \Qp^\times\to \Fp^\times\hookrightarrow \Fq^\times$ denote the character corresponding to the mod $p$ cyclotomic character via local class field theory.

\begin{cor}
Assume $p\neq 2$ and let $\chi,\chi'\colon T\to\mathbb{F}_q^{\times}$ be two continuous characters such that either $\chi=\chi'$ or $\chi\neq\chi'$ and there exists a simple root $\alpha\in\Delta$ satisfying $\chi'=w_\alpha(\chi)\cdot \varepsilon^{-1}\circ\alpha$, $\chi\neq w_\alpha(\chi)$, and $\chi'\neq w_\alpha(\chi')$. Then the functor $\mathbb{V}^\vee\circ D_\Delta^\vee$ induces an isomorphism $$\Ext^1(\Ind_{\overline{B}}^{G}\chi,\Ind_{\overline{B}}^{G}\chi')\overset{\sim}{\rightarrow} \Ext^1_{\GQpD\times\mathbb{Q}_p^\times}(\mathbb{V}^\vee\circ D_\Delta^\vee(\Ind_{\overline{B}}^{G}\chi),\mathbb{V}^\vee\circ D_\Delta^\vee(\Ind_{\overline{B}}^{G}\chi'))\ .$$
\end{cor}
\begin{proof}
Assume $\chi\neq \chi'$. Since the center of $\GL_n(\Qp)$ is connected, $\Ext^1(\Ind_{\overline{B}}^{G}\chi,\Ind_{\overline{B}}^{G}\chi')$ is $1$-dimensional by \cite[Thm.\ 1.2]{Hauseux}. Further, the unique nontrivial extension $\pi$ is parabolically induced from $\overline{P}^{(\alpha)}$, ie.\ we have a smooth representation $\pi_\alpha$ of $L^{(\alpha)}$ such that $\pi\cong\Ind_{\overline{P}^{(\alpha)}}^G\pi_\alpha$. Further, by our assumptions $\chi\neq \chi'$, $\chi\neq w_\alpha(\chi)$, $\chi'\neq w_\alpha(\chi')$, $\pi_\alpha$ is an extension of $2$ nonisomorphic irreducible principal series. Hence we may apply \cite[Proposition VII.4.7]{Colmez1} to deduce that the corresponding Galois representation is a nonsplit extension of $2$ characters. The statement follows from the isomorohism of $\Ext$ groups in \cite[Proposition VII.4.7]{Colmez1} using \cite[Theorem B]{MultVar}.

The case $\chi=\chi'$ follows from Theorem \ref{torabel} using \cite[Thm.\ 1.2]{Hauseux}.
\end{proof}
\begin{rem}
Even in case $n=2$ there are characters $\chi,\chi'$ of the torus such that $\Ext^1_{G}(\Ind_B^{G}\chi,\Ind_B^{G}\chi')= 0$, but $\Ext^1_{\GQpD\times\mathbb{Q}_p^\times}(\mathbb{V}^\vee\circ D_\Delta^\vee(\Ind_{\overline{B}}^{G}\chi),\mathbb{V}^\vee\circ D_\Delta^\vee(\Ind_{\overline{B}}^{G}\chi'))$. For instance let $\chi=\chi_1\otimes\chi_2$ and $\chi'=\chi_3\otimes\chi_4$ such that the central characters agree: $\chi_1\chi_2=\chi_3\chi_4$, but $\chi_3\neq\varepsilon^{-1}\chi_1$.
\end{rem}

Finally, we present a new proof of another result of Hauseux:

\begin{thm}[special case of Thm.\ 3.2.1 in \cite{Hauseux3}]\label{nolength3}
Let $\chi,\chi',\chi''\colon T\to \Fq^\times$ be three distinct characters such that all three principal series representations $\Ind_{\overline{B}}^G\chi$, $\Ind_{\overline{B}}^G\chi'$, and $\Ind_{\overline{B}}^G\chi''$ are irreducible. Then there does not exist a smooth representation of $\GL_n(\Qp)$ over $\Fq$ of length $3$ with socle isomorphic to $\Ind_{\overline{B}}^G\chi$, cosocle isomorphic to $\Ind_{\overline{B}}^G\chi''$, and the unique intermediate constituant isomorphic to $\Ind_{\overline{B}}^G\chi'$.
\end{thm}
\begin{proof}
Assume for contradiction that there is such a representation $\pi$ and put $V:=\mathbb{V}^\vee\circ D^\vee_\Delta(\pi)$. By exactness \cite[Thm.\ C]{MultVar}, $V$ is a successive extension of distinct characters $\eta,\eta',\eta''\colon \GQpD\times \Qp^\times\to \Fq^\times$. Put $V_1\leq V$ for the extension of $\eta'$ by $\eta$ (ie.\ $V_1=\mathbb{V}^\vee(D^\vee_\Delta(\pi_1))$ where $\pi_1\leq \pi$ is the subrepresentation with quotient $\pi/\pi_1\cong\Ind_{\overline{B}}^G\chi''$) and $V_2:=V/\eta$ (ie.\ $V_2=\mathbb{V}^\vee(D^\vee_\Delta(\pi_2))$ where $\pi_2=\pi/\Ind_{\overline{B}}^G\chi$). By Proposition \ref{splits_all_but_one_pro} there is a simple root $\alpha\in \Delta$ such that $V_1$ does not split as a representation of $\GQpa$ and $\eta_{\mid \GQpDa\times \Qp^\times}=\eta'_{\mid \GQpDa\times \Qp^\times}$. Similarly, there is a simple root $\beta\in\Delta$ such that $V_2$ does not split as a representation of $G_{\Qp,\beta}$ and $\eta'_{\mid G_{\Qp,\Delta\setminus\{\beta\}}\times \Qp^\times}=\eta''_{\mid G_{\Qp,\Delta\setminus\{\beta\}}\times \Qp^\times}$. As $\eta\neq \eta''$, we must have $\alpha\neq\beta$. Now the dimension of $\Ext^1_{\GQpa}(\eta'_{\mid\GQpa},\eta_{\mid\GQpa})$ is $1$ (see, for example \cite[Prop.\ VII.4.7]{Colmez1}). However, $\GQpa$ acts by the character $\eta'_{\mid\GQpa}$ on $V_2$, so there must be a $1$-dimensional subspace $V_3$ of $V_2$ whose extension by $\eta_{\mid \GQpa}$ is trivial. Since $\GQpa$ is a normal subgroup in $\GQpD\times \Qp^\times$, $V_3$ is also a $\GQpD\times \Qp^\times$-subrepresentation of $V_2$. However, $V_2$ is a nonsplit extension, so it has a unique $1$-dimensional subrepresentation $\eta'$. This contradicts to $V_1$ being nonsplit as a representation of $\GQpa$.
\end{proof}

\begin{rem}
Suppose in addition to the assumptions of Theorem \ref{nolength3} that $\chi=w_\alpha(\chi')\cdot \varepsilon^{-1}\circ\alpha$ and $\chi''=w_\beta(\chi')\cdot \varepsilon^{-1}\circ\beta$ for some (necessarily distinct) simple roots $\alpha,\beta\in\Delta$. Then there exists representations $\pi_1,\pi_2$ which are nonsplit extensions of $\Ind_{\overline{B}}^G\chi'$ by $\Ind_{\overline{B}}^G\chi$, resp.\ of $\Ind_{\overline{B}}^G\chi''$ by $\Ind_{\overline{B}}^G\chi'$. So Theorem \ref{nolength3} shows that the class of the length $2$ Yoneda extension
\begin{align*}
0\to \Ind_{\overline{B}}^G\chi\to \pi_1\to \pi_2\to \Ind_{\overline{B}}^G\chi''\to 0
\end{align*}
is nonzero in $\Ext^2(\Ind_{\overline{B}}^G\chi'',\Ind_{\overline{B}}^G\chi)$. It would be interesting to generalize this to Yoneda classes in higher extension groups and possibly give a new proof of a recent result of Heyer \cite[Thm.\ 4.2.8]{Heyer} (see also \cite[Thm.\ 5.3.1]{Hauseux1} for a similar result depending on the validity of \cite[Conjecture 3.7.2]{EmertonOrd2} on the derived functor of Emerton's ordinary parts).
\end{rem}

\subsection{The conjectural essential image of $D^\vee_\Delta$ on $\SP_{\mathcal{O}_F/\varpi^h}$}\label{sec:conj}

We end this paper by formulating a conjecture on the essential image of the functor $D^\vee_\Delta$ restricted to the category $\SP_h$. Let $V$ be an object in $\Rep^{\ord}_h(\GQpD\times\Qp^\times)$. For each $\alpha=\alpha_j\in \Delta$ ($j=1,\dots,n-1$) consider the socle filtration 
\begin{align*}
0=\soc_0^{\overline{\alpha}}(V)\lneq \soc_1^{\overline{\alpha}}(V)\lneq\dots\lneq \soc_{r_\alpha}^{\overline{\alpha}}(V)=V
\end{align*}
of $V$ as a representation of $\GQpDa\times\Qp^\times$. Since $\GQpDa\times\Qp^\times$ is a normal subgroup in $\GQpD\times\Qp^\times$, $\soc_i^{\overline{\alpha}}(V)$ is a $\GQpD\times\Qp^\times$-subrepresentation of $V$ over $\Fq$ and $\GQpDa\times\Qp^\times$ acts on the graded pieces $\mathrm{gr}^{\overline{\alpha}}_i(V):=\soc_i^{\overline{\alpha}}(V)/\soc_{i-1}^{\overline{\alpha}}(V)$ ($i=1,\dots,r_\alpha$) via a direct sum of characters. For each character $\eta_{\overline{\alpha}}\colon \GQpDa\times \Qp^\times\to \Fq^\times$ put $\mathrm{gr}^{\overline{\alpha}}_i(V)(\eta_{\overline{\alpha}})$ for the $\eta_{\overline{\alpha}}$-isotypic component of $\mathrm{gr}^{\overline{\alpha}}_i(V)$ ($i=1,\dots,r_\alpha$) as a representation of the $j$th copy of $G_{\Qp}$. Further, let $\widetilde{\eta_{\overline{\alpha}}}^{(k)}\colon \Qp^\times\to\Fq^\times$ be the character corresponding to the restriction of $\eta_{\overline{\alpha}}$ to the $k$th copy of $G_{\Qp}$ for $k\neq j\in \{1,\dots,n-1\}$ and $\widetilde{\eta_{\overline{\alpha}}}^{(n)}\colon \Qp^\times\to\Fq^\times$ for the restriction of $\eta_{\overline{\alpha}}$ to the center $\Qp^\times$ and put $\widetilde{\eta_{\overline{\alpha}}}^{(0)}\colon \Qp^\times\to\Fq^\times$ for the trivial (identically $1$) character.

\begin{con}\label{conimage}
The object $V$ in $\Rep^{\ord}_h(\GQpD\times\Qp^\times)$ is in the essential image of the functor $\mathbb{V}^\vee\circ D^\vee_\Delta$ if and only if for each $\alpha\in\Delta$, $i\in\{1,\dots,r_\alpha\}$, and character $\eta_{\overline{\alpha}}\colon \GQpDa\times \Qp^\times\to \Fq^\times$ the representation $\mathrm{gr}^{\overline{\alpha}}_i(V)(\eta_{\overline{\alpha}})$ is in the image of Colmez' Montréal functor at a representation $\pi$ of $\GL_2(\Qp)$ over $\Fq$ with central character $\frac{\widetilde{\eta_{\overline{\alpha}}}^{(j+1)}}{\widetilde{\eta_{\overline{\alpha}}}^{(j-1)}}$.
\end{con}

\begin{pro}\label{onlyif}
The ``only if'' part of Conjecture \ref{conimage} is true.
\end{pro}
\begin{proof}
Let $\pi$ be an object in $\SP_h$ and put $V:=\mathbb{V}^\vee\circ D^\vee_\Delta(\pi)$. By \cite[Theorem C]{MultVar} there is a subquotient $\mathrm{gr}^{\overline{\alpha}}_i(\pi)(\eta_{\overline{\alpha}})$ of $\pi$ for each $\alpha\in\Delta$, $i\in\{1,\dots,r_\alpha\}$, and character $\eta_{\overline{\alpha}}\colon \GQpDa\times \Qp^\times\to \Fq^\times$ such that
\begin{align*}
\mathbb{V}^\vee\circ D^\vee_\Delta\left(\mathrm{gr}^{\overline{\alpha}}_i(\pi)(\eta_{\overline{\alpha}}) \right)\cong \mathrm{gr}^{\overline{\alpha}}_i(V)(\eta_{\overline{\alpha}})\ .
\end{align*}
In particular, $\mathrm{gr}^{\overline{\alpha}}_i(\pi)(\eta_{\overline{\alpha}})$ lies in the category $\SP^{\alpha,\widetilde{\eta}_{\overline{\alpha}}}_{\Fq}$. Hence by Theorem \ref{lookslikeparindisparind} there exists a representation $\pi_0(\alpha,i,\eta_{\overline{\alpha}})$ of $L^{(\alpha)}$ such that we have
\begin{align*}
D^\vee_\Delta(\mathrm{gr}^{\overline{\alpha}}_i(\pi)(\eta_{\overline{\alpha}}))\cong D^\vee_\Delta(\Ind_{\overline{P}^{(\alpha)}}^G\pi_0(\alpha,i,\eta_{\overline{\alpha}}))\ .
\end{align*}
By \cite[Thm.\ B]{MultVar}, the $\GQpa$-representation $\mathrm{gr}^{\overline{\alpha}}_i(V)(\eta_{\overline{\alpha}})$ is the value of Colmez' Montréal functor at the restriction of $\pi_0$ to $\GL_2(\Qp)\cong \GL_2(\Qp)^{(\alpha)}\leq L^{(\alpha)}$. Finally, if $\alpha=\alpha_j$ then the center of $\GL_2(\Qp)^{(\alpha)}$ consists of elements $\diag(\underbrace{1,\dots,1}_{j-1},x,x,\underbrace{1,\dots,1}_{n-j-1})=\diag(\underbrace{x,\dots,x}_{j+1},\underbrace{1,\dots,1}_{n-j-1})\diag(\underbrace{x,\dots,x}_{j-1},\underbrace{1,\dots,1}_{n-j+1})^{-1}$ so the central character of $\pi_0$ is $\frac{\widetilde{\eta_{\overline{\alpha}}}^{(j+1)}}{\widetilde{\eta_{\overline{\alpha}}}^{(j-1)}}$ as predicted by Conjecture \ref{conimage}.
\end{proof}

\begin{rems}\begin{enumerate}
\item It would be very interesting to reformulate Conjecture \ref{conimage} in terms of deformation theory in the spirit of \cite{Paskunas}. The technical difficulty is that in case $n>2$ the injective hull of principal series representations appears to contain supersingular subquotients, as well. Therefore we cannot control $D^\vee_\Delta$ on such injective hulls as we do not know the exactness or the value at supersingular representations.
\item The condition on a representation $V$ of $\GQpD\times \Qp^\times$ being ordinary is strictly stronger than just requiring that its restriction to the diagonal embedding of $G_{\Qp}\hookrightarrow \GQpD\times \Qp^\times$ be ordinary.
\item Evidence for Conjecture \ref{conimage} is provided by Proposition \ref{onlyif}, Theorem \ref{torabel}, and the construction $\Pi(\overline{\rho})^{\ord}$ (or even $\Pi(\overline{\rho})_I$ for various sets $I$ of roots) of Breuil and Herzig \cite[section 3.4]{BH1}.
\end{enumerate}
\end{rems}

\bibliographystyle{alpha}
\bibliography{ref}

\end{document}